\documentclass[11pt]{amsart}

\allowbreak

\numberwithin{equation}{section}

\usepackage{amsfonts,amssymb,amscd,amsmath,latexsym,amsbsy}
\usepackage{graphicx}
\usepackage{subcaption}
\usepackage{amssymb}
\usepackage{amsfonts}
\usepackage{latexsym}
\usepackage{mathrsfs}
\usepackage{mathtools}
\usepackage{amsthm}
\usepackage{comment}
\usepackage{listings}
\usepackage{tikz-cd}
\usepackage{tikz}
\usetikzlibrary {calc}
\usepackage{hyperref}
\usepackage{amsmath}
\usepackage{leftidx}
\usepackage[toc,page]{appendix}
\usepackage{upgreek}

\allowdisplaybreaks
\newtheorem{thm}{Theorem}[section]

\newtheorem{lmm}[thm]{Lemma}
\newtheorem{prop}[thm]{Proposition} 
\newtheorem{coro}[thm]{Corollary}

\theoremstyle{definition}
\newtheorem{definition}[thm]{Definition}
\newtheorem{remark}[thm]{Remark}
\newtheorem*{remark*}{Remark} 

\newtheorem{example}[thm]{Example}

\numberwithin{equation}{section}

\newcommand{\SM}{\mathcal M}
\newcommand{\SR}{\mathcal R}

\newcommand{\R}{\ensuremath{\mathbb{R}}}
\newcommand{\RR}{\ensuremath{\mathbb{R}}}
\newcommand{\Z}{\ensuremath{\mathbb{Z}}}
\newcommand{\N}{\ensuremath{\mathbb{N}}}

\newcommand{\WF}{\ensuremath{\mathrm{WF}}}
\newcommand{\Ell}{\ensuremath{\mathrm{Ell}}}
\newcommand{\supp}{\ensuremath{\mathrm{supp}}}

\newcommand{\mk}{\ensuremath{\mathfrak}}

\newcommand{\msf}{\ensuremath{\mathsf}}

\newcommand{\la}{\ensuremath{\langle}}
\newcommand{\ra}{\ensuremath{\rangle}}

\newcommand{\Op}{\ensuremath{\mathrm{Op}}}

\newcommand{\Id}{\ensuremath{\mathrm{Id}}}
\newcommand{\oc}{\ensuremath{\mathrm{1c}}}

\newcommand{\sct}{\ensuremath{\mathrm{sc}}}
\newcommand{\ps}{\ensuremath{\mathrm{ps}}}

\newcommand{\sub}{\ensuremath{\mathrm{sub}}}

\newcommand{\ff}{\ensuremath{\mathrm{ff}}}
\newcommand\SL{\mathcal{L}}

\newcommand\xoc{{x_{\oc}}}
\newcommand\yoc{{y_{\oc}}}
\newcommand\etaoc{{\eta_{\oc}}}
\newcommand\xioc{{\xi_{\oc}}}

\newcommand\rhops{{\rho_{\ps}}}
\newcommand\ffocps{{\ff_{\oc-\ps}}}
 
\newcommand\ffococ{{\ff_{\oc-\oc}}} 

\newcommand\xoco{x_{\oc, 1}}
\newcommand\yoco{y_{\oc, 1}}
\newcommand\xioco{\xi_{\oc, 1}}
\newcommand\etaoco{\eta_{\oc, 1}}
\newcommand\xoct{x_{\oc, 2}}
\newcommand\yoct{y_{\oc, 2}}
\newcommand\xioct{\xi_{\oc, 2}}
\newcommand\etaoct{\eta_{\oc, 2}}
\newcommand{\TSJ}{\mk{T}}

\newcommand{\xiocr}{\xi_{\oc,0}}
\newcommand{\xiocl}{\xi_{\oc,0}}

\newcommand{\xsc}{{x}}

\newcommand{\cl}{\ensuremath{\mathrm{cl}}}

\newcommand{\Char}{\ensuremath{\mathrm{Char}}}

\newcommand{\Cl}{\ensuremath{\mathfrak{Cl}_g}}
\newcommand{\Cli}{\ensuremath{\mathfrak{Cl}_{g_i}}}
\newcommand{\Clo}{\ensuremath{\mathfrak{Cl}_{g_1}}}
\newcommand{\Clt}{\ensuremath{\mathfrak{Cl}_{g_2}}}

\newcommand{\ococb}{\ensuremath{ \leftidx{^{ \mathsf{L}\oc}}{ T^*X^2_b} } }
\newcommand{\secm}{\ensuremath{ {}^{\oc,\sct,\xiocl}{T^*X^2_b} } }
\newcommand{\secmz}{\ensuremath{{}^{\oc,\sct,0}{T^*X^2_b}}}
\newcommand{\Lagsec}{\Uplambda^{\oc,\sct}_{\xi_{\oc,0}}}
\newcommand{\Lagsecz}{\Uplambda^{\oc,\sct}_{0}}
\newcommand{\Legsec}{\mathcal{L}^{\oc,\sct}_{\xi_{\oc,0}}}
\newcommand{\ffsec}{\ff_{(\oc,\sct,\xiocl)}}

\newcommand{\rHp}{\ensuremath{ \mathsf{H}^{m',0}_p } }

\newcommand\rv[1]{{\color{blue} #1}}

\newcommand{\Rp}{\mathcal{R}_+}
\newcommand{\Rm}{\mathcal{R}_-}

\newcommand\Legps{L}
\newcommand\Leg{\mathcal{L}}
\newcommand\Lagps{\Lambda}
\newcommand\Lag{\Uplambda}

\newcommand\FSR{\overline{\Lambda_-'}}
\newcommand\BSR{\overline{\Lambda_+'}}

\newcommand\Poi{\mathcal{P}}

\newcommand\Poiz{\mathcal{P}_0}
\newcommand\Poipm{\mathcal{P}_\pm}
\newcommand\Radm{\mathcal{R}_-}
\newcommand\Radp{\mathcal{R}_+}

\newcommand\base{\mathrm{base}}

\newcommand{\ococparaone}{v}
\newcommand{\ococparatwo}{w}

\newcommand\ocphase{\overline{{}^{\oc} T^* \R^n}}
\newcommand\psphase{\overline{{}^{\ps} T^* \R^{n+1}}}
\newcommand{\scphase}{{}^{\sct} T^* \overline{\R^n}}
\usepackage[letterpaper,top=2cm,bottom=2cm,left=2cm,right=2cm,marginparwidth=1.75cm]{geometry}

\title{Determining metrics from the scattering map of the time-dependent Schr\"odinger equation}

\author{Qiuye Jia}
\date{\today}

\begin{document}

\begin{abstract}
For a time dependent Schr\"odinger equation, the scattering map is the map sending the asymptotic profile of solution as $t\to-\infty$ to its asymptotic profile as $t\to+\infty$. 
In this paper we show that, for certain class of metrics, the scattering maps associated to two Schr\"odinger operators with two time dependent metrics only differ by a compact operator if and only if these two metrics are related by a pull-back of a diffeomorphism.
\end{abstract}

\maketitle

\tableofcontents
\section{Introduction}

\subsection{The scattering map and the main result}
\label{subsec:main-results}

In this article, we study the problem that how the scattering map of the time-dependent Sch\"odinger equation under a time-dependent metric and time-dependent potential determine the metric itself. 
More precisely, we consider the Schr\"odinger operator $P$ on $\R^{n+1}_{z, t}$, where $z \in \R^n$, $t \in \R$, 
\begin{equation} \label{eq:def-Schrodinger-op}
P = D_t + \Delta_{g(t)} + V(z, t), 
\end{equation}
where $D_t = -i \partial_t$, $\Delta_{g(t)}$ is the positive Laplace operator with respect to a smooth family of metrics $g(t)$ on $\R^n_z$, and $V$ is a smooth complex-valued potential function. 
Here allowing a time-dependent metric $g(t)$ that is different from the Euclidean one models that our particle is going through a curved space or an inhomogeneous medium that varies with time. 
In addition, we make the assumption that $g(t)$ is a compactly supported, in spacetime, perturbation of the flat metric:
\begin{equation} \label{eq:g-definition-perturbation}
g(t) = g_0+ \tilde{g}(t) \text{ with } \tilde{g} \text{ that is compactly supported in spacetime}, 
\end{equation}
and $V$ is also compactly supported in spacetime. 
Thus, there exist large constants $R$, $T$ such that if either $|z| \geq R$, or if $|t| \geq T$, then $g(t) = \sum_{i,j} g_{ij}(z,t) dz_i dz_j$ coincides with the flat metric $g_0 = \sum_i dz_i^2$, and $V$ vanishes identically. To put this another way, near spacetime infinity, $P$ coincides with the `free' Schr\"odinger operator $P_0 := D_t + \Delta_0$, where $\Delta_0$ is the standard (positive) Laplacian on $\RR^n$. We assume, in addition, that each $g(t)$ is non-trapping in the sense that its geodesics always escape any compact region in $\R^n$ in finite time.

The scattering map concerns final state data or scattering data of global solutions to $Pu = 0$. Every  global solution $u$ has an asymptotic expansion of the form 
\begin{equation}\label{eq:u expansion}
u \sim (4\pi it)^{-n/2} e^{i|z|^2/4t} f_\pm\big( \frac{z}{2t} \big) + O(|t|^{-n/2 - \epsilon}), \quad t \to \pm \infty, \, \epsilon>0
\end{equation}
for large positive or negative times, where this holds in a pointwise sense if $f_\pm$ are sufficiently regular, and distributionally in general. The functions $f_\pm$ are by definition the asymptotic data of the solution $u$. 
Then the \emph{scattering map} of $P$ is the map
\begin{equation} \label{eq:S-def-intro}
    S: \; f_- \to f_+.
\end{equation}
The main theorem of \cite{HJ2026-scattering-map} shows that this map is an elliptic Fourier integral operator—more precisely, a Legendre distribution—of a novel type.
See \cite[Section~1]{HJ2026-scattering-map} for more detailed discussion.

When two metric $g_1(t)$ and $g_2(t)$ are related by a pull-back:
\begin{equation} \label{eq:g2g1-pullback}
    g_1 = \psi^*g_2,
\end{equation}
where $\psi: \R^{n+1} \to \R^{n+1}$ is a diffeomorphism that preserves time slices, which means
\begin{equation} \label{eq:psi-concrete}
    \psi(t,z) = (t,\psi_1(t,z))
\end{equation}
with $\psi_1(t,\cdot): \R^n \to \R^n$ being a diffeomorphism, and $\psi_1$ is the identity for $|t| \geq T$ or $|z| \geq R$.\footnote{This can be relaxed to only require that $\psi$ becomes a fixed isometry of the Euclidean space, i.e. a translation and multiplication by an element in $SO(n)$, for $t \geq T$ and $t \leq -T$ respectively. }, one can show that the scattering map $S_{g_i}$ associated to 
\begin{equation} \label{eq:def-Pi}
    P_i=D_t+\Delta_{g_i(t)}+V_i(z,t),\; i=1,2
\end{equation}
are the same on the leading order and $S_{g_1} - S_{g_2}$ is a compact operator from $L^2(\R^n)$ to itself. 
In fact, due to the lack of invariance of the time-dependent Schr\"odinger equation under pull-back, even this `forward problem', though much easier than the inverse problem discussed below, is not entirely straightforward. 
See Appendix~\ref{sec:appendix-forward-invariance} for a more precise version and details.

The aim of this article is to answer the inverse problem arises from this. That is, suppose either of the following holds:
\begin{align}
& \text{the scattering map } S_{g_i} \text{ associated to } P_i  \text{ in \eqref{eq:def-Pi} coincide on the leading order, }  \label{eq:S1-S2-condition-1} \\
  & S_{g_1}-S_{g_2}  \text{ is a compact operator from } L^2(\R^n) \text{ to itself},  \label{eq:S1-S2-condition-2}
\end{align}
then does this information from the PDE side determine the geometric information?
Or more precisely do $g_1$ and $g_2$ have to be related as in \eqref{eq:g2g1-pullback}?
We will give an affirmative answer to this question in this article, with certain restriction on one of $g_i$, which arises naturally in the boundary rigidity problem. Concretely, we will consider the following class of metrics.

\begin{definition} \label{definition:metric-classes}
Let $B_R(0) \subset \R^n$ be the ball of radius $R$ centered at the origin and $T \in \R_+$, such that $[-T,T] \times B_R(0)$ contains the support of $g-g_0$, then we say $g$ \emph{admits a convex function} if there is a function $f \in C^\infty([-T,T] \times B_R(0))$ such that $\mathrm{Hess} \, f$ is strictly positive, where the Hessian is with respect to $g(t)$ on $B_R(0)$.
\end{definition}

\begin{example} \label{example:curvature condition}
A simple example is $f=z_1^2+...+z_n^2$ on $\R^n$ with Euclidean metric.
The geometric importance of (the existence of) such a function is that its level sets give a convex foliation of $B_R(0)$, which in turn allows us to determine the metric using the lens data. See Section~\ref{sec:preliminaries-inverse-problems} for more discussion.

The existence of such $f$ can be guaranteed by certain curvature conditions.
For example, it exists when for all $t \in [-T,T]$ either of the following holds:
\begin{enumerate}
\item $g(t)$ has non-positive sectional curvature. \label{condition:defocusing-metric}
\item $g(t)$ has non-negative sectional curvature. \label{condition:focusing-metric}
\end{enumerate}
See \cite[Lemma~2.1]{paternain2019geodesic} for details.
Roughly speaking, the conditions require that for each fixed time $t$, the geodesic flow of $g(t)$ only has the effect of either focusing or defocusing.
We should emphasize that we do allow a transition between those two scenarios. That is, $g(t)$ satisfies \eqref{condition:defocusing-metric} for some $t$ while $g(t)$ satisfies \eqref{condition:focusing-metric} for some other $t$.
At the moment of transition, one would have the Euclidean metric, since the sectional curvature will vanish identically and our space (consider the entire $\R^n$ now) is complete and simply connected.

As already pointed out after \cite[Corollary~1.1]{SUV2021local-gloabl}, using \cite[Lemma~2.1]{paternain2019geodesic}, the condition in Case~\eqref{condition:focusing-metric} can be relaxed to be $K \geq -\kappa$, where $\kappa>0$ is a constant depending only on $R$. 
We choose to keep using $K \geq 0$ to have a clearer and simpler formulation.
\end{example}



Then our main result is the following. 
\begin{thm} \label{thm:main-intro}
    Suppose that either $g_1$ or $g_2$ admits a convex function in the sense of Definition~\ref{definition:metric-classes}, 
    then $S_{g_1}-S_{g_2}$ is a compact operator from $L^2(\R^n)$ 
    to itself if and only if $g_1 = \psi^*g_2$ with $\psi$ as in \eqref{eq:psi-concrete}.
    In particular, this shows that the other $g_i$ admits a convex function as well.
\end{thm}
See Theorem~\ref{thm:forward-coincide-leading-order} in Appendix~\ref{sec:appendix-forward-invariance} for the `if' direction and the primary goal of this article is to establish the “only if’’ direction.
As we will discuss in Remark~\ref{remark:coincide-leading-order-vs-compact}, the condition that $S_{g_1}-S_{g_2}$ is a compact operator from $L^2(\R^n)$ to itself can be implied from the statement that $S_{g_1}$ and $S_{g_2}$ coincide on the leading order in the sense that \eqref{eq:Sg1-Sg2-leading-same-remark} holds, so we have answered both questions we proposed above in \eqref{eq:S1-S2-condition-1}\eqref{eq:S1-S2-condition-2}.

The condition that either $g_1$ or $g_2$ admits a convex function only enters in the last step invoking Theorem~\ref{thm:boundary-rigidity}, 
one can replace this condition by any other conditions that allow lens data (i.e. the scattering relation and length of geodesics) to determine the metric up to a pull-back.
The core of this paper is to determine the geometric information (i.e. lens data) from the analytic information (i.e. the scattering map).

An interesting direction for the future work is to determine the potential from the scattering map. 
The effect of potentials does not enter the leading order behavior of the scattering map, hence can't be detected by the compactness of $S_{g_1}-S_{g_2}$.
It is expected that the potential and even certain information about the diffeomorphism $\psi$ can be determined if we exploit the information of lower order symbols of the scattering map.

\subsection{Strategy of the proof} \label{subsec:strategy-sketch}

We briefly summarize the strategy of the proof here.
Generally speaking, in geometric inverse problems, one can determine the metric on a manifold with boundary up to a pullback by a diffeomorphism fixing the boundary through the scattering relation in \eqref{eq:scattering-relation} below and the length of geodesics under certain restrictions. 
See \cite[Section~1]{SUV2021local-gloabl}\cite{uhlmann2016journey}\cite{Stefanov-survey} and references therein for a complete survey. 
We will choose this manifold with boundary to be a large ball $B_R(0)$ containing the support of $g(t)-g_0$ for all $t$.
Then the question is reduced to determining the scattering relation and the length of geodesics from the information given by our PDE, i.e. the scattering map.
We will show that these two types of information are encoded in the graph of the classical scattering map as in \eqref{eq:Cl-def-intro} below lifted to the 1c–1c phase space (see \eqref{eq:def-lifted-ococ-bundle}) by demonstrating following two principles:
\begin{enumerate}
    \item The scattering relation is encoded in the restriction of the classical scattering map to the boundary of the 1c-phase space. See Proposition~\ref{prop:determine-truncated-SC-map} for a precise version. \label{principle:1}
    \item The lengths of geodesics are captured by the 1-jet (at the boundary) of the lifted graph of the scattering map.
    See Theorem~\ref{thm:Cl-determines-length} for a precise version. \label{principle:2}
\end{enumerate}
We briefly explain these two principles below. 
First, the scattering relation is encoded by the asymptotic directions of the bicharacteristic flow and the location escaping $B_R(0)$, which are in turn
recorded as the location and frequency variables 
of our 1c-1c Lagrangian submanifold capturing the oscillation of the scattering map.
Notice that the scattering map records the information living at spacetime infinity and for each fixed $t$ there is a $(n-1)$-dimensional family of geodesics that tend to the same point at spacetime infinity\footnote{One can fix the escaping direction and translate the escaping location on half of $\partial B_R(0)$. }, hence a total $n$-dimensional family of geodesics tending to the same point at the equator $\{ t/|z| = 0, 1/|z| = 0 \}$ since all different $t$ are compressed together. 
So distinguishing different geodesics from the analytic information at infinity is non-trivial.
This is precisely one of the key advantages of introducing this 1-cusp frame work: one can distinguish geodesics within this family of geodesics having the same asymptotic direction using the 1-cusp frequency!

More precisely, we will use the following bulk-boundary duality between the parabolic scattering cotangent bundle over the bulk $\overline{\R^{n+1}}$ (the radially compactified $\R^{n+1}$), which is denoted by $\overline{{}^{\ps}T^*\R^{n+1}}$, and the 1-cusp cotangent bundle over half of its boundary identified as $\overline{\R^n}$, which is denoted by $\overline{{}^{\oc}T^*\R^n}$.
The frequency in $\overline{{}^{\ps}T^*\R^{n+1}}$ at the place escaping $B_R(0)$ determines the asymptotic direction of bicharacteristics, which in turn corresponds to the position part of $\overline{{}^{\oc}T^*\R^n}$.
On the other hand, the position on the bulk, or more precisely its component that is orthogonal to the frequency at the place escaping $B_R(0)$, parametrizes the family of geodesics having the same asymptotic direction and will correspond to the tangential component of 1-cusp frequency in $\overline{{}^{\oc}T^*\R^n}$.
In this way, we obtain the \emph{classical scattering map}
\begin{equation} \label{eq:Cl-def-intro}
    \Cl : \; \overline{{}^{\oc}T^*\R^n} \to \overline{{}^{\oc}T^*\R^n}
\end{equation}
such that it sends the endpoint of a bicharacteristic line at far past (in terms of the bicharacteristic flow, not necessarily the actual time) to the endpoint at the far future.

Such duality implicitly appeared in \cite{gell2022propagation} in the form of module regularity. Then \cite{HJ2026final} pointed out that such regularity is the 1-cusp regularity and applied such duality to the final state problem of the nonlinear Schr\"odinger equation.
The way in which we will exploit such duality will be more close to the theory of 1c-ps Fourier integral operators developed in \cite{HJ2026-scattering-map}.


Finally, we are left to determine the length of geodesics. We will show that this is encoded in the 1-jet of the graph of the classical scattering map, which is denoted by $\mathrm{Gr}(\Cl)$.
We will show that the subleading part of the radial 1-cusp frequency encodes the `sojourn time' introduced by Guillemin \cite{Guillemin-sojourn}, which in turn determines the length of geodesics.
In order to `extract' this subleading order information, we introduce the technique of second-microlocalization in Section~\ref{sec:second-micro}. In our setting, this means introducing blow-up at certain parts of the phase space to create new frequencies characterizing oscillations on a slower scale.
We will show that if one blow-up the 1-cusp phase space at a fixed radial frequency, then the front face is effectively the scattering phase space, and the radial frequency in this scattering phase space determines the sojourn time of geodesics, which justifies principle \eqref{principle:2} above.

Technically, such radial frequency will be recovered from the function value of phase functions parametrizing our Lagrangian submanifolds at their critical points.
This reflects a major difference between classical Fourier integral operators with homogeneous phase functions and
Legendre distributions (or Fourier integral operators with inhomogeneous phase functions).
For the former, the phase function always vanishes at its critical points due to the homogeneity and Euler's homogeneous function theorem.
For the latter, the value of the phase function at its critical points won't necessarily vanish and can record the travel time of the bicharacteristic flow.
Such property is also exploited in \cite{hassell-wunsch2005schrodinger}\cite{Hassell-Wunsch-semiclassical-resolvent}.


\subsection{Literature review}

Recovering metrics from the scattering data has been a classical topic in the scattering theory.
In the setting of the time independent Schr\"odinger equation, Joshi and S\'a Barreto established results recovering jets of asymptotically conic metrics \cite{Joshi-Sabarreto-conic-metric} and asymptotically hyperbolic metrics \cite{inverse-scattering-hyperbolic} from the scattering data at a fixed energy level.

In \cite{BK-CPDE}, Belishev and Kurylev showed that the `spectral data' of the Laplace operator on a bounded domain can be used to recover the metric.
A similar result is deduced in \cite{BK-wave} for the wave equation. See also \cite{BK-BC-IP} for more about this boundary control approach.
Eskin \cite{Eskin-recover-metric} considered such question of recovering metric on $\R^n$ from the scattering data of the time-independent Sch\"odinger equation under the assumption that the perturbation is small.

On the other hand, in the setting of time-dependent Schr\"odinger equation with a time-dependent-metric, the literature is relatively sparse. The classes of Fourier integral operators characterizing the Poisson operator and the scattering map in this setting are only developed very recently by the author and Hassell \cite{HJ2026-scattering-map}.
To the author's knowledge, this is the first result of determining a time-dependent metric from the scattering map of a time-dependent Sch\"odinger equation.
For nonlinear wave equations, the problem of recovering time-dependent metrics from the scattering map is considered by Hintz, S\'a Barreto, Uhlmann, and Zhang \cite{hintz2024inverse}.

Though we are not treating the problem of recovering potentials here, it is worthwhile to mention that this problem has even longer history.
In the inverse scattering approach, for integrable systems, Gel'fand, Levitan, and Marchenko developed methods to explicitly recover potentials from the scattering data.
Later, Faddeev \cite{Faddeev} further developed this approach. See \cite[Section~3]{Melrose-geometric-scattering} and references therein for more discussions.
For the time-dependent Sch\"odinger equation, the question of determining a time-dependent potential in 1+1 dimension is studied by Zhou \cite{Zhou-TDS-1Dinverse}.


Now we turn to discuss related developments in microlocal analysis.
See \cite[Section~1.3]{HJ2026-scattering-map} for a more complete discussion about the development of Fourier integral operators and Legendre distributions.
We only mention here that the theory of Lagrangian distributions and Fourier integral operators arose in the 1960s through work of Maslov \cite{Maslov1967}, Egorov \cite{Egorov1971} and others.
Then works of H\"ormander \cite{FIO1}, Duistermaat-H\"ormander \cite{FIOII} established a geometrical, invariant symbolic calculus of it.
Then this is extended to allow clean parametrizations in Duistermaat-Guillemin \cite{duistermaat-guillemin1975spectrum}.
The theory of Legendre distributions, which takes the spatial infinity arising in compactification instead of fiber infinity as the boundary, is established in Melrose-Zworski \cite{melrose1996scattering}, which allows one to study global properties on non-compact manifolds using such a microlocal theory.
Later Hassell and Vasy \cite{hassell1999spectral}\cite{hassell2001resolvent} extended such theory to handle manifolds with codimension two corners. 
In addition, an analogue of \eqref{eq:S formula} expressing the scattering map in terms of microlocalized Poisson operators is used by \cite{vasy1998geometric} to avoid treating conic intersection pairs Legendre submanifolds in \cite{melrose1996scattering}.

In the setting of time-dependent Schr\"odinger equation, Hassell and Wunsch \cite{hassell-wunsch2005schrodinger} gave the propagator of Schr\"odinger equation a Fourier integral operator (or Legendre distribution) characterization.
In terms of terminology we will introduce in Section~\ref{sec:second-micro}, one can view one of the strategies in this work as completely removing the 1-cusp level oscillation and focus on the analysis on the scattering level.
In \cite{HJ2026-scattering-map}, the author and Hassell developed the 1c-ps Fourier integral operator characterizing the Poisson operator and 1c-1c Fourier integral operator class characterizing the scattering map.
Here 1c stands for `1-cusp', which arises from the pseudodifferential algebra constructed in \cite{Zachos:Thesis}\cite{zachos2022inverting}, and ps stands for `parabolic scattering', which is used to study time-dependent Schr\"odinger equations in \cite{gell2022propagation}\cite{gell2023scattering}.

The second microlocalization is formulated by Bony \cite{Bony-second-micro} in a way that is different from our approach.
The implementation here follows Vasy \cite{Vasy-second-microlocal-1}\cite{Vasy-second-microlocal-3}.
See also \cite{GHS2}\cite{NRL-I}\cite{NRL-II} for similar constructions.


\section{Some preliminaries on Inverse problems}
\label{sec:preliminaries-inverse-problems}

As mentioned in Section~\ref{subsec:strategy-sketch}, the first reduction in our proof relies on building connections to results in geometric inverse problems, so we recall some basic ingredients in this direction here.

Let $(\mk{B}, \mk{g})$ be a compact manifold with boundary with metric $\mk{g}$, then generally speaking the goal of a geometric inverse problem is the following:
\begin{center}
Determine $\mk{g}$ through `measurements' at $\partial \mk{B}$.
\end{center}

Of course, those measurements have to contain information about the interior of $\mk{B}$ to achieve this. One typical example and what we will use is the lens relation, which consists of two parts: the scattering relation and length data. To define the lens data, 
we introduce $\partial_- S^* \mk{B}$ and $\partial_+ S^* \mk{B}$, which are the inward and outward pointing part of $S^*_{\partial \mk{B}} \mk{B}$\footnote{This denotes the restriction of the unit (with respect to the dual metric induced by $\mk{g}$) cosphere bundle of $\mk{B}$ to $\partial \mk{B}$.}.

The scattering relation
\begin{equation} \label{eq:scattering-relation}
\mathsf{SR}_{\mk{g}}: \quad \partial_- S^* \mk{B} \to \partial_+ S^* \mk{B}
\end{equation}
is defined as follows: let $(p,\xi) \in \partial_- S^* \mk{B}$ and denote the unique lifted geodesic\footnote{By lifted geodesics, we always mean the lift in the cotangent bundle or its restriction to the cosphere bundle. Which is the same, via identifying the cotangent and tangent bundle using the Riemannian structure, as the perhaps more conventional geodesic flow on the tangent bundle.} by $\gamma_{p,\xi}$, then we set $\mathsf{SR}(p,\xi)$ to be the first point when $\gamma_{p,\xi}$ hits $\partial_+ S^*\mk{B}$. 
The length data is based on this: we define 
\begin{equation} \label{eq:def-length-g}
\ell_{\mk{g}}: \partial_- S^* \mk{B} \to [0,\infty)
\end{equation}
to be the map sending $(p,\xi)$ to the length of this curve. We will only concern the case where $\mk{g}$ is \emph{non-trapping} in this paper, which means that the geodesic flow always exits $\mk{B}$ within finite time, or equivalently $\ell_{\mk{g}}$ always has a finite value.

So a more concrete formulation of the goal above is to determine $\mk{g}$ from $(\mathsf{SR}_{\mk{g}},\ell_{\mk{g}})$. But this has the following clear obstruction: let $f: \mk{B} \to \mk{B}$ be a diffeomorphism fixing the boundary, then those datum of $f^*\mk{g}$ will be the same as those of $\mk{g}$.

Consequently, the meaningful goal is the following:
\begin{center}
Suppose $\mk{g}_1,\mk{g}_2$ are two metrics on $\mk{B}$ having the same lens data (that is $(\mathsf{SR}_{\mk{g_1}},\ell_{\mk{g_1}})=(\mathsf{SR}_{\mk{g_2}},\ell_{\mk{g_2}})$), then does it have to be the case that $\mk{g}_1 = f^* \mk{g}_2$ for some diffeomorphism $f$?
\end{center}

The answer is affirmative in a lot of cases. We state the result that we will use by Stefanov, Uhlmann and Vasy \cite{SUV2021local-gloabl}, which is currently the most complete answer to this question,
and then briefly recall some results that are directly related to our situation.
As geometric inverse problems are not the main subject of this paper, we can't give a complete survey on the vast literature on this problem.


\begin{thm}[Adapted from Stefanov-Uhlmann-Vasy \cite{SUV2021local-gloabl}] \label{thm:boundary-rigidity}
Suppose $\mk{g}_1$ and $\mk{g}_2$ are two metrics on $\mk{B}$, which is a compact manifold of dimension at least $3$. 
If either of $\mk{g}_i$ admits a convex function $f$ with $\{f=0\} = \partial \mk{B}$,
then there is a diffeomorphism $\psi: \mk{B} \to \mk{B}$, $\psi|_{\partial \mk{B}} = \mathrm{Id}$, such that $\mk{g}_1 = \psi^* \mk{g}_2$.

In addition, if we have two families of $\mk{g}_1(t),\mk{g}_2(t)$ depending smoothly on a parameter $t$ (ranging over a compact interval) as above, with one of $\mk{g}_i(t)$ admitting a convex function for all $t$,
then we can choose a family of diffeomorphisms $\psi_t$ depending smoothly
on $t$ as above such that $\mk{g}_1(t) = \psi_t^* \mk{g}_2(t)$.
\end{thm}


\begin{proof}
The main part of the theorem follows from \cite[Theorem~8.1]{SUV2021local-gloabl}. 
The only point that is not explicitly included there is the smooth dependence we stated in the last part of the theorem when we consider a family of metrics. But this follows from the construction in \cite{SUV2021local-gloabl}.

Concretely, the proof of the `local step' \cite[Lemma~7.1]{SUV2021local-gloabl}, which constructs a diffeomorphism matches the jet of two metrics at $\partial \mk{B}$ via geodesic normal coordinates, depends on metrics smoothly. 
In the step extending this $\psi$ to a global one in \cite[Section~7.4]{SUV2021local-gloabl}, the step size is uniformly lower bounded by a positive constant that is robust under small perturbations of the metric, and the construction in each step depends on the metric smoothly. So we know that the global diffeomorphism constructed there depends on the metric and the convex foliation smoothly.
On the other hand, our condition guarantees that the given metric and the convex foliation depends on $t$ smoothly, so $\psi_t$ depends on $t$ smoothly as well.
\end{proof}

As we have mentioned after Theorem~\ref{thm:main-intro}, Theorem~\ref{thm:boundary-rigidity} is only used in Section~\ref{sec:determine-metric}, which is the last step of our proof: determining the metric after reading all needed geometric information from our Legendrian distributions. We use this version because one of the major advantages of our calculus of Legendrian distributions is that we can handle the presence of conjugate points. 
In the non-positive curvature case of Example~\ref{example:curvature condition},
there is no conjugate point. Since our metric can be embedded into a compact region of the Euclidean space of the same dimension, one can use the result of Croke \cite{croke1991rigidity}, which is extended to the case where one only assumes some decay condition by Guillarmou, Mazzucchelli and Tzou \cite{guillarmou2019asymptotically}.
See also results by Gromov in \cite[Section~5.5 B$'$]{gromov1983filling} and by Michel in \cite[Theorem~B]{michel1981rigidite}, under similar conditions.

\section{The scattering and parabolic scattering pseudodifferential algebras}
\label{sec: parabolic sc PsiDO}
In this section we give a rather brief introduction of the scattering and parabolic scattering pseudodifferential algebra. 


\subsection{Scattering pseudodifferential algebra}
\label{subsec:sc-bundle-PsiDO}

The scattering cotangent bundle, symbol class and pseudodifferential algebra are introduced by Melrose \cite{Melrose1994}.

We use $\overline{\R^n}$ to denote the radially compactified $\R^n$, which is $\R^n \bigsqcup [0,1)_x \times \mathbb{S}^{n-1}/\sim$ with $\sim$ being the identification $Z \in \R^n \sim (1/|Z|,Z/|Z|) \in [0,1)_x \times \mathbb{S}^{n-1}$ for $|Z|>1$.
The boundary defining function of $\partial \overline{\R^n}$ can be taken as $x = |Z|^{-1}$ away from the origin.
Let $y$ be a coordinate system of $\partial \overline{\R^n} = \mathbb{S}^{n-1}$, then $(x,y)$ forms a local coordinate system of $\overline{\R^n}$ near infinity.
The scattering cotangent bundle over $\R^n$ (or in fact over $\overline{\R^n}$) is the one taking $\frac{d\xsc}{\xsc^2}, \frac{dy}{\xsc}$ as its frame, and the canonical 1-form in terms its frequencies is written as 
\begin{equation} \label{eq:sc-1-form}
    \xi_{\sct} \frac{d\xsc}{\xsc^2} + \eta_{\sct} \frac{dy}{\xsc}.
\end{equation}
We denote the scattering cotangent bundle after compactifying fibers radially to be a ball by $\overline{{}^{\sct}T^*\R^n}$ and denote the boundary defining function of fiber infinity by $\rho_{\sct}$.

For $m,l \in \R$, the class of ``scattering symbols'' $S_{\mathrm{sc}}^{m,l}(\R^n)$ with differential order $m$ and spacetime order $s$,  
is defined to be the set of smooth functions $a(Z,Z^*)$ on $T^*\R^n$ (with $Z^*$ being the fiber variable) such that there exist constants $C_{\alpha \beta}$ such that
\begin{equation} 
	\label{eq: sc symbol}
	|\partial_Z^\alpha\partial_{Z^*}^\beta a(Z,Z^*)| \leq C_{\alpha\beta} \la Z^* \ra^{m-|\beta|} \la Z \ra^{l -|\alpha|} 
\end{equation}
for all multi-indices $\alpha,\beta\in \N^{n}$.
The topology on $S_{\mathrm{sc}}^{m,l}(\R^n)$ is the Fr\'echet topology induced by the countably many seminorms determined by \eqref{eq: sc symbol}.
Then the space of scattering pseudodifferential operators of order $(m,l)$, which is denoted by $\Psi_{\mathrm{sc}}^{m,l}(\R^n)$, consists of operators acting by the standard (left) quantization 
\begin{equation}
	q_L(a) f(z)  = \frac{1}{(2\pi)^{n}} \int   e^{i Z^* \cdot (Z-Z')} a(Z,Z^*) f(Z') dZ' dZ^*,
	\label{eq:quant_sc}
\end{equation}
with $a \in S_{\mathrm{sc}}^{m,l}(\R^n)$.

The principal symbol of $A = q_{L}(a) \in \Psi_{\sct}^{m,l}(\R^{n+1})$ is defined to be the equivalence class of $a$ in $S_{\sct}^{m,l}(\R^{n+1})$ quotient by $S_{\sct}^{m-1,l-1}(\R^{n+1})$:
\begin{align*}
[a] \in S_{\sct}^{m,l}(\R^{n+1})/S_{\sct}^{m-1,l-1}(\R^{n+1}).
\end{align*}

A symbol $a \in S_{\sct}^{m,l}(\R^{n+1})$ (and corresponding operator $A=q_{L}(a)$) is said to be elliptic if it satisfies:
\begin{align*}
|\rho_{\sct}^m x^{l}a| \geq C, \text{ when } \rho_{\sct} \leq \epsilon 
\text{ or } x \leq \epsilon, 
\end{align*}
for some $\epsilon>0,C>0$. And we say that $a$ (and corresponding operator $A=q_{L}(a)$) is elliptic at $q \in \partial \overline{{}^{\sct}T^*\R^{n}}$ if there is a neighborhood of $q$ on which above inequality is satisfied. And the set of all such $q \in \partial \overline{{}^{\sct}T^*\R^{n}}$ is denoted by $\Ell_{\sct}^{m,l}(a)$ (or $\Ell_{\sct}^{m,l}(A)$).


\subsection{Parabolic scattering pseudodifferential algebra}
\label{subsec:ps-bundle-PsiDO}

We discuss the parabolic scattering pseudodifferential algebra in this subsection.
As the name indicates, it is the `parabolic version' of the scattering pseudodifferential algebra, where `parabolic' refers to the way to compactify the fiber infinity. 
More precisely, let 
\begin{align}
(t,z,\tau,\zeta),
\end{align}
be coordinates on $T^*\R^{n+1}$, with $\tau dt + \zeta \cdot dz$ being the canonical form.
Then the (compactified) parabolic scattering cotangent bundle, denoted by $\overline{^{\ps}T^*\R^{n+1}}$ , is defined by compactifying $T^*\R^{n+1}$ in following way.
The `base' $\R^{n+1}$ is compactified radially to be a ball with boundary defining function
\begin{align}\label{eq:xps defn}
x_{\ps} = (1+t^2+|z|^2)^{-1/2},
\end{align}
while each fiber is compactified `parabolically' to be a ball with boundary defining function
\begin{align}\label{eq:rhops defn}
\rho_{\ps} = (1+\tau^2+|\zeta|^4)^{-1/4},
\end{align}
with $\ps$ standing for `parabolic scattering'. 

\begin{definition}
The symbol class $S_{\ps}^{m,l}(\R^{n+1})$, with $m,l$ differential and decay (in fact growth) order respectively, we use, is defined to be the space of $a \in C^\infty(T^*\R^{n+1})$ (in fact also their extension to $\overline{^{\ps}T^*\R^{n+1}}$) such that
\begin{align*}
\|a\|_{ S_{\ps,N}^{m,l}} 
:= \sum_{|\alpha|+k+|\beta|+j \leq N} \sup_{T^*\R^{n+1}} |x_{\ps}^{l-|\alpha|-k} \rho_{\ps}^{m-\beta-2j} \partial_z^\alpha \partial_t^k \partial_\zeta^\beta \partial_\tau^j a(z,t,\zeta,\tau) | < \infty,
\end{align*}
for any $N \in \N$. And these norms give $S_{\ps}^{m,l}(\R^{n+1})$ a structure of Fr\'echet space. In addition, when $\rho_{\ps}^m x_{\ps}^l a$ extends to a smooth function on $\overline{^{\ps}T^*\R^{n+1}}$, then we say $a$ is classical, and the corresponding symbol class is denoted by $S_{\ps,\cl}^{m,l}(\R^{n+1})$.
\end{definition}
\begin{remark}
In the definition above, each $\partial_\tau$ improves the decay at fiber infinity by 2 order whereas $\partial_\zeta$ only improves one order. This reflects the parabolic nature of this symbol class.

Geometrically, the difference between the large symbol class $S_{\ps}^{m,l}(\R^{n+1})$ and the `small' symbol class $S_{\ps,\cl}^{m,l}(\R^{n+1})$ is that the membership of the former means only conormal property on $\overline{^{\ps}T^*\R^{n+1}}$, while the membership of the latter means smoothness on $\overline{^{\ps}T^*\R^{n+1}}$ (in particular, down to its boundaries).
\end{remark}

Then the corresponding parabolic scattering pseudodifferential operators 
$\Psi_{\ps}^{m,l}(\R^{n+1})$ are operators that are quantizations of symbols in
$S_{\ps}^{m,l}(\R^{n+1})$, which means they have Schwartz kernels of the form
\begin{align*}
q_{L,\ps}(a) = (2\pi)^{-(n+1)} \int e^{ i (t-t')\tau+(z-z')\cdot \zeta}
a(t,z,\tau,\zeta) d\zeta d\tau,
\end{align*}
in the distributional sense.

The principal symbol of $A = q_{L,\ps}(a) \in \Psi_{\ps}^{m,l}(\R^{n+1})$ is defined to be the equivalence class of $a$ in $S_{\ps}^{m,l}(\R^{n+1})$ quotient by $S_{\ps}^{m-1,l-1}(\R^{n+1})$:
\begin{align*}
[a] \in S_{\ps}^{m,l}(\R^{n+1})/S_{\ps}^{m-1,l-1}(\R^{n+1}).
\end{align*}

A symbol $a \in S_{\ps}^{m,l}(\R^{n+1})$ (and corresponding operator $A=q_{L,\ps}(a)$) is said to be elliptic if it satisfies:
\begin{align*}
|\rho_{\ps}^m x_{\ps}^{l}a| \geq C, \text{ when } \rho_{\ps} \leq \epsilon 
\text{ or } x_{\ps} \leq \epsilon, 
\end{align*}
for some $\epsilon>0,C>0$. And we say that $a$ (and corresponding operator $A=q_{L,\ps}(a)$) is elliptic at $q \in \partial \overline{^{\ps}T^*\R^{n+1}}$ if there is a neighborhood of $q$ on which above inequality is satisfied. And the set of all such $q \in \partial \overline{^{\ps}T^*\R^{n+1}}$ is denoted by $\Ell_{\ps}^{m,l}(a)$ (or $\Ell_{\ps}^{m,l}(A)$).

Let $A \in \Psi_{\ps}^{m,l}(\R^{n+1})$ be classical, i.e. $A = q_{L, \ps}(a)$ where the symbol $a$ is classical. Thus, by assumption, $\rho_{\ps}^{m} x_{\ps}^{l}  a$ extends smoothly to the boundary of $\overline{^{\ps}T^*\R^{n+1}}$. 
We define $\Sigma(A)$ to be the set where the normalized symbol $\rho_{\ps}^{m} x_{\ps}^{l}  a$ vanishes on $\overline{^{\ps}T^*\R^{n+1}}$, and define the 
 characteristic set $\Char_{\ps}^{m,l}(A)$ (or $\Char_{\ps}^{m,l}(a)$) to be the intersection of $\Sigma(A)$ with the boundary of phase space: 
\begin{align}
\Char_{\ps}^{m,l}(A) = \{ q \in \partial \overline{^{\ps}T^*\R^{n+1}} \ | 
\  (\rho_{\ps}^m x_{\ps}^{l}a)(q) = 0 \}.
\end{align}
Notice that $\Char_{\ps}^{m,l}(A)$ (but not $\Sigma(A)$) depends only on the principal symbol of $A$. Microlocal propagation takes place at the boundary of phase space, therefore strictly speaking only $\Char_{\ps}^{m,l}(A)$ is relevant, but we will sometimes find it convenient to discuss propagation through the interior of phase space, that is, in $\Sigma(A) \setminus \Char_{\ps}^{m,l}(A)$, as well. (Actually only the jet of $\Sigma(A)$ at the boundary of $\overline{^{\ps}T^*\R^{n+1}}$ is relevant.) 


When $a \in S_{\ps,\cl}^{m,l}(\R^{n+1})$, we define the rescaled Hamilton vector field to be
\begin{align}\label{eq:rescaled Hvf}
H^{m,l}_a = \rho_{\ps}^{m-1} x_{\ps}^{l-1} H_a.
\end{align}
Due to the parabolic nature, $\partial_\tau a \partial_t - \partial_ta \partial_\tau$ is one more order lower than other parts of the expression at fiber infinity:
\begin{align}
H^{m,l}_a|_{\rho_{\ps}=0}
= \rho_{\ps}^{m-1} x_{\ps}^{l-1}(\partial_\zeta a \cdot \partial_z - \partial_z a \cdot \partial_\zeta).
\end{align}
As shown in \cite[Section~2]{gell2022propagation}, $H^{m,l}_a$ extends to a smooth vector field on $\overline{^{\ps}T^*\R^{n+1}}$ that is tangent to the boundary. In addition, it is tangent to $\Char_{\ps}^{m,l}(a)$, and the vector field restricted to the characteristic set depends only on the principal symbol.

Finally we recall the concept of (parabolic) wavefront sets, which is also called the micro-support.
For $A = q_{L,\ps}(a)$, its parabolic scattering operator wavefront set $WF'_{\ps}(A)$ is defined, as a subset of $\partial \overline{^{\ps}T^*\R^{n+1}}$, as follows.
For $q \in \partial \overline{^{\ps}T^*\R^{n+1}}$, we say
$q \notin \WF'_{\ps}(A)$ if and only if there is $\chi \in C^\infty(\overline{^{\ps}T^*\R^{n+1}})$ with $\chi(q) = 1$ such that $\chi a \in \mathcal{S}(\overline{^{\ps}T^*\R^{n+1}})$.
In particular, $\WF'_{\ps} (A) = \emptyset$ if and only if $A \in \Psi_{\ps}^{-\infty,-\infty}(\R^{n+1})$.

Similarly, for $u \in \mathcal{S}'(\R^{n+1})$ its $\ps$-wavefront set of order $s,r$, denoted by $\WF_{\ps}^{s,r}(u)$, is defined to be the (microlocal) loci of where $u$ fails to lie in $H_{\ps}^{s,r}(\R^{n+1})$.
Concretely, for $q \in \partial \overline{^{\ps}T^*\R^{n+1}}$, we have
\begin{equation}
    q \notin \WF_{\ps}^{s,r}(u) \iff \text{ there exists } A \in \Psi_{\ps}^{s,r} \text{ that is elliptic at } q, \, \text{ such that } Au \in L^2(\R^{n+1}).
\end{equation}
Then the (microlocal) loci of where $u$ fails to lie in $\mathcal{S}(\R^{n+1})$ is denoted by
\begin{equation}
    \WF_{\ps}(u) = \overline{\bigcup_{s,r} \WF_{\ps}^{s,r}(u)}.
\end{equation}

\section{The 1-cusp pseudodifferential algebra}
\label{sec:the_1_cusp_pseudodifferential_algebra}

\subsection{The 1-cusp pseudodifferential algebra}
\label{subsec:the_1_cusp_pseudodifferential_algebra}

In this subsection we briefly introduce the 1-cusp pseudodifferential algebra, and refer readers to \cite{Zachos:Thesis}\cite[Section~2]{zachos2022inverting}\cite[Section~2]{jia2022tensorial} for more details.
For its connection to time-dependent Sch\"odinger equations, see \cite[Section~3.3]{HJ2026-scattering-map}.

Let $M$ be an $m-$dimensional manifold with boundary, with boundary defining function $x_{\oc}$, and suppose $y=(y_1,..,y_{m-1})$ is a coordinate system of $\partial M$, which together with $x_{\oc}$ form a local coordinate system of $M$ near the boundary. Then the space of 1-cusp vector fields, denoted by $\mathcal{V}_{\oc}$, is locally spanned over $C^\infty(M)$ by
\begin{align} \label{eq: 1-cusp vector fields, local}
x_{\oc}^3\partial_{x_{\oc}}, x_{\oc}\partial_{y_j}, j=1,2,...,m-1.
\end{align}
$\mathcal{V}_{\oc}$ give rise to a vector bundle with \eqref{eq: 1-cusp vector fields, local} being its local frame, which is called 1-cusp tangent bundle, and denoted by $^{\oc}TM$.

The 1-cusp differential operators of order (at most) $k$, denoted by $\mathrm{Diff}_{\oc}^k(M)$, consists of polynomials of these vector fields of degree at most $k$:
\begin{align}  \label{eq: Diff 1c definition}
\begin{split}
\mathrm{Diff}_{\oc}^k(M) 
= \{ \sum_{\alpha+|\beta| \leq k}  a_{\alpha\beta}(x_{\oc},y)(x_{\oc}^3\partial_{x_{\oc}})^\alpha(x_{\oc}\partial_{y})^\beta : 
\\ a_{\alpha \beta} \in C^\infty(M), \alpha \in \N, \beta \in \N^{m-1} \}.
\end{split}
\end{align}

The 1-cusp cotangent bundle, denoted by $^{\oc}T^*M$ is the dual bundle of $^{\oc}TM$. It is locally spanned over $C^\infty(M)$ by 
\begin{align}
\frac{dx_{\oc}}{x_{\oc}^3}, \frac{dy_j}{x_{\oc}}, j=1,2,...,m-1.
\end{align}
Then $T^*M$ embeds into $^{\oc}T^*M$ canonically in the interior of $M$, giving it a symplectic structure naturally. In particular, one may write the canonical one form as
\begin{align}  \label{eq: 1c- canonical form}
\xi_{\oc} \frac{dx_{\oc}}{x_{\oc}^3} + \eta_{\oc} \cdot \frac{dy}{x_{\oc}},
\end{align}
where $(\xi_{\oc},\eta_{\oc})$ are coordinates of fibers on $^{\oc}T^*M$. 
We use $\overline{^{\oc}T^*M}$ to denote the compactification of $^{\oc}T^*M$, obtained by compactifying each fiber radially to be a ball
with boundary defining function
\begin{align}
\rho_{\oc} = (1+\xi_{\oc}^2+|\eta_{\oc}|^2)^{-1/2}.
\end{align}

Notice the minor difference with $\overline{^{\ps}T^*\R^{n+1}}$, which is compactified on both base and fiber level, is because the 1-cusp construction is happening on the manifold with boundary $M$, which is already `compactified in a priori'.

The 1-cusp symbol class of differential order $m$ and decay (in fact growth) order $l$, denoted by $S_{\oc}^{m,l}(M)$ ,is defined to be smooth functions on ${}^{\oc}T^*M$ satisfying for any $j,\alpha,k,\beta$ there exists a constant $C_{j\alpha k \beta}$ such that
\begin{align}
 |x_{\oc}^l \rho_{\oc}^{m-k-|\beta|} (x_{\oc}\partial_{x_{\oc}})^j\partial_y^\alpha\partial_{\xi_{\oc}}^k\partial_{\eta_{\oc}}^\beta a(x_{\oc},y,\xi_{\oc},\eta_{\oc})|
 \leq C_{j\alpha k \beta}.
\end{align}
And locally they quantize to be operators acting by
\begin{align}
(q_{L,\oc}(a)u)(x_{\oc},y) = (2\pi)^{-n}
\int e^{ i \xi_{\oc}\frac{x_{\oc}-x_{\oc}'}{x_{\oc}^3}+\eta_{\oc} \cdot \frac{y-y'}{x_{\oc}}}
a(x_{\oc},y,\xi_{\oc},\eta_{\oc}) u(x_{\oc}',y') d\xi_{\oc}d\eta_{\oc} \frac{dx_{\oc}'dy'}{(x_{\oc}')^{n+2}}.
\end{align}
The collection of all such operators, are called 1-cusp pseudodifferential operators with differential order $m$ and decay order $l$, and denoted by $\Psi_{\oc}^{m,l}(M)$.

Next we define the ellipticity of symbols and operators.
\begin{definition}
A symbol $a \in S^{m,l}_{\oc}(M)$ is called elliptic if
\begin{align*}
|a(x_{\oc},y,\xi_{\oc},\eta_{\oc})| \geq cx_{\oc}^{-l} \la (\xi_{\oc},\eta_{\mathrm{1c}}) \ra^m, \quad c>0 \text{ when }  |(\xi_{\oc},\eta_{\oc})|^{-1} \leq \epsilon \text{ or } x_{\oc} \leq \epsilon,
\end{align*}
for some $\epsilon>0,C>0$, and its quantization $A$ is also called elliptic in this case.
\end{definition}
Under this condition, see \cite[Section~2.5]{zachos2022inverting}, its quantization $A$ has a parametrix $B \in \Psi_{\oc}^{-m,-l}(M)$ such that
\begin{align*}
AB-\Id, \; BA-\Id \in \Psi_{\oc}^{-\infty,-\infty}(M).
\end{align*}

One can now define 1-cusp Sobolev spaces
$H^{s,r}_{\oc}(M)$ (see \cite[Section~2.5]{zachos2022inverting}) for $s \geq 0$ by choosing $A \in \Psi_{\oc}^{s,0}(M)$ elliptic, and demanding
\begin{equation} \label{eq:1c-Sobolev-def}
u\in H^{s,r}_{\oc}(M) \Leftrightarrow u\in x_{\oc}^r L_{\oc}^2(M)\ \text{and}\ Au\in x_{\oc}^r L_{\oc}^2(M);
\end{equation}
here $L_{\oc}^2(M)$ is the $L^2$ space relative to the 1-cusp density $\frac{dx\,dy}{x_{\oc}^{n+2}}$.
They are equipped with norms:
$$
\|u\|^2_{H^{s,r}_{\mathrm{1c},h}}=\|x_{\oc}^{-r}u\|_{L_{\oc}^2}^2+\sum_{j+|\alpha|\leq s}\|(hx_{\oc}^3D_{x_{\oc}})^j(h^{1/2}x_{\oc}D_y)^\alpha\|_{L_{\oc}^2}^2.
$$
The 1-cusp Sobolev spaces for other $s$ are defined via interpolation and duality.

1-cusp pseudodifferential operators are bounded on these Sobolev spaces, namely for $A \in \Psi_{\oc}^{m,l}(M)$ and all $s,r$, $A$ is bounded linear operator from $H^{s,r}_{\oc}$ to $H^{s-m,r-l}_{\oc}$.

The 1-cusp wavefront set, which characterizes the location of singularities in the phase space, is defined by:
\begin{definition} \label{defn: 1c WF, both decay and smoothness}
\begin{align} \label{eq: definition of 1c WF, both decay and smoothness}
\begin{split}
\WF_{\oc}(u) =&  \Big(\{ q \in \partial \overline{{}^\oc{T^*X}}: \text{There exists } A \in \Psi_{\oc}^{0,0}(X) \text{ such that } 
 \\& A \text{ is elliptic at } q, \,Au \in \dot{C}^\infty(X) \} \Big)^c.
\end{split}
\end{align}
\end{definition}
And we have a refined version which captures the order of the singularities:
\begin{definition} \label{defn: 1c WF, both decay and smoothness, with order}
\begin{align} \label{eq: definition of 1c WF, both decay and smoothness, with order}
\begin{split}
\WF_{\oc}^{s,r}(u) =&  \Big(\{ q \in \partial \overline{{}^\oc{T^*X}}: \text{There exists } A \in \Psi_{\oc}^{s,r}(X) \text{ such that } 
 \\& A \text{ is elliptic at } q, \,Au \in L^2(X)  \}\Big)^c.
\end{split}
\end{align}
\end{definition}

The residual one, for example $\WF_{\oc}^{-\infty,r}(u)$ is defined by
\begin{align}
\WF_{\oc}^{-\infty,r}(u) = \bigcap_{s \in \Z} \WF_{\oc}^{s,r}(u). 
\end{align}
On the opposite direction, we define 
\begin{align}
\WF_{\oc}^{\infty,r}(u) = \overline{\bigcup_{s \in \Z} \WF_{\oc}^{s,r}(u)}. 
\end{align}

\section{\texorpdfstring{The $\oc-\ps$ geometry and analysis}{The 1c-ps geometry and analysis}} 
\label{sec:1c-ps-geometry-distribution}

In this section, we recall main geometric and analytic ingredients needed for the theory $\oc-\ps$ Fourier integral operators from \cite[Section~3]{gell2022propagation} and \cite[Section~4, Section~5]{HJ2026-scattering-map}. This is used to characterize the Poisson operator sending the asymptotic data $f_\pm$ in \eqref{eq:u expansion} to the solution $u$.

\subsection{The bulk-boundary duality for the Schr\"odinger equation}
\label{subsec:bulk-boundary}

In this subsection, we first recall the dynamical structure of the flow of the Hamilton vector field of our symbol and then explain the bulk-boundary duality that we mentioned in Section~\ref{subsec:strategy-sketch} in this subsection. 

Let $p$ be the left symbol of $P$, the interesting part is where the Schr\"odinger operator is non-elliptic, i.e. the characteristic variety:
\begin{equation}
    \Sigma(P):= p^{-1}(0) \subset \overline{{}^{\ps}T^*\R^{n+1}}.  
\end{equation}
The set 
\begin{equation} \label{eq:CharP-def}
    \Char(P) = \Sigma(P) \cap \partial(\overline{{}^{\ps}T^*\R^{n+1}})
\end{equation}
is where the microlocal propagation takes place.
We denote the boundary defining function of fiber infinity of $\overline{{}^{\ps}T^*\R^{n+1}}$ by $\rhops$.

One of the key ingredient of our analysis is the global source-sink structure of the rescaled Hamilton flow associated to $P$, which we introduce below.
Using the symplectic structure on $T^*\R^{n+1}$ we have a local expression for the Hamilton vector field
\begin{align}
H_a = \partial_\tau a \partial_t - \partial_ta \partial_\tau
+ \partial_\zeta a \cdot \partial_z - \partial_z a \cdot \partial_\zeta,
\end{align}
whose integral curves (more precisely, their closures) are called bicharacteristic lines.
As the discussion in \cite[Section~3]{gell2022propagation} shows, 
the bicharacteristic lines starting on the boundary with finite frequency will remain on the boundary for all time and only those bicharacteristics at fiber infinity of $\overline{{}^{\ps}T^*\R^{n+1}}$ will leave $\partial \overline{\R^{n+1}}$ and meet the metric perturbation, so we will restrict ourselves to be near the characteristic set in such a region.
In addition, such a bicharacteristic line will stay in the same time slice, or in other words, it reaches the endpoint `instantly' in terms of $t$.
In this region, we know $\tau^{1/2}$ and $|\zeta|_g$ are comparable and we can take
\begin{equation} \label{eq:rho-ps-def}
    \rhops = \big( \sum_{j,k} g^{jk}(z, t) \zeta_j \zeta_k \big)^{-1/2}.
\end{equation}

Then let $\hat{\tau} = \rho_{\ps}^2\tau \in \R$, and $\hat{\zeta} \in \R^{n-1}$ be coordinates parametrizing $\rho_{\ps} \zeta$ in the sphere with respect to $g$, a valid coordinate system with $(t,z)$ in a compact region is
\begin{equation} \label{eq:ps-coordinates}
    (t,z,\rho_{\ps},\hat{\tau},\hat{\zeta}).
\end{equation}

Let $\rho_{\base} = (1+t^2+|z|^2)^{-1/2}$ be the defining function of the spacetime infinity, then the rescaled Hamilton vector field of $p$
\begin{equation} \label{eq:Hp20-def}
    H_p^{2,0} = \rhops \rho_{\base}^{-1}H_p
\end{equation}
is a smooth vector field on $\overline{{}^{\ps}T^*\R^{n+1}}$ that is tangent to its boundary. 
The flow of $H_p^{2,0}$ have a global source-sink structure with following `radial sets' being the source and sink.
\begin{definition}
 \label{defn: radial sets}
The radial set (of $P$) $\mathcal{R}$ is defined to be 
\begin{align}
\mathcal{R} = \{ q \in \Char(P): H^{2,0}_p \text{ vanishes at } q \}.
\end{align}
\end{definition}
Since the flow of $H^{2,0}_p$ has a source-sink structure , we have the decomposition of $\mathcal{R}$:
\begin{equation}
    \mathcal{R} = \Radm \cup \Radp \subset \overline{{}^{\ps}T_{\partial \overline{\R^{n+1}}}^*\R^{n+1}},
\end{equation}
where $\Radm$ is the source and $\Radp$ is the sink. 
In the region $|t/ z | \leq C$ for a constant $C>0$, they are given by 
\begin{equation}\label{eq:rad.corner.pm}
\mathcal{R}_\pm \supset \{ x_{\ps} = 0, \ t/|z| = \pm \rho_{\ps}/2, \ \hat \zeta/|\zeta| = \pm z/|z|, \ \tau/|\zeta|^2 = -1 \} \cap \{  \frac{|t|}{|z|} \leq C \}.
\end{equation}

So one can see that $\Radm$ is a graph over the southern hemisphere (i.e the part of $\partial \overline{\R^{n+1}}$ with $t/\la z \ra \leq 0$) and $\Radp$ is a graph over the northern hemisphere (i.e. the part of $\partial \overline{\R^{n+1}}$ with $t/\la z \ra \geq 0$).
Both of them turn vertical (in the sense of tending to fiber infinity) when approaching the equator $\{ t/|z| = 0, 1/|z|=0 \}$.
Though they overlap at the equator on the base level, but they tend to the opposite direction in terms of the frequency, hence they remain disjoint on the phase space level.

As aforementioned, for each point on $\partial \overline{\R^{n+1}}$, there is an $n$-dimensional family of geodesic tending to it. On a phase space level, this corresponds to different bicharacteristic lines tending to or emanating from the same point in $\mathcal{R}_\pm$. Then they are distinguished via blowing up $\mathcal{R}_\pm$ within $\Sigma$:
\begin{equation}
    [\Sigma; \mathcal{R}_\pm].
\end{equation}
We denote the front face created by the blow up by $W_\pm$, then $W_\pm$ precisely parametrizes all bicharacteristic lines. Using $W_\pm$, we have the following characterization of our bulk-boundary duality happens between a point on a bicharacteristic line and the `endpoint' of this bicharacteristic line. 
\begin{prop}\cite[Lemma~3.4]{HJ2026-scattering-map}
\label{prop:Wpm-1c-identification}
$W_\pm$ is canonically diffeomorphic to $\overline{{}^{\oc}T^*\mathcal{R}_\pm}$, the radially compactified 1-cusp cotangent bundle over $\mathcal{R}_\pm$.
\end{prop}
Since we have a canonical identification between $\mathcal{R}_\pm$ with $\overline{\R^n}$ via projection to the base, $\overline{{}^{\oc}T^*\mathcal{R}_\pm}$ above can be replaced by $\overline{{}^{\oc}T^*\R^n}$.
In addition, more precise characterization of the diffeomorphism in Proposition~\ref{prop:Wpm-1c-identification} is available.
For a bicharacteristics, the `location' of its corresponding point is just the  .
The fibers of $W_\pm$ over $\mathcal{R}_\pm$, which corresponds to different bicharacteristics tending to the same point in $\mathcal{R}_\pm$, is identified with the fiber of $\overline{{}^{\oc}T^*\R^n}$. Hence those 1-cusp frequencies are parametrizing such family of bicharacteristic lines, which are originally parametrized by position variables in the bulk $\R^{n+1}$.
Under this identification, bicharacteristics at fiber infinity in the bulk corresponds to $\overline{{}^{\oc}T_{\partial \overline{\R^n}}^*\R^n}$, i.e. the part of the 1-cusp cotangent bundle over the boundary of the `base'.
More explicitly, consider the case for $\mathcal{R}_+$ for definiteness, let $(t_0,z_0)$ be the position part of a point on a bicharacteristic $\tilde{\gamma}$ after escaping the metric perturbation, then in terms of frequencies in \eqref{eq: 1c- canonical form}, the point in $\overline{{}^{\oc}T_{\partial \overline{\R^n}}^*\R^n}$ corresponds to this line has frequency 
\begin{equation} \label{eq:xioc-etaoc-t-z-relation}
\xioc = -2t_0 + O(\xoc), \quad \etaoc = -z_0' - \frac{(z_0' \cdot \yoc)\yoc}{1 - |\yoc|^2},
\end{equation}
where we assumed that the asymptotic direction is near $(1,0,..,0)$ and $z_0' = ((z_0)_2, \dots, (z_0)_n)$  parametrizes part of $z_0$ that is perpendicular to the asymptotic direction.


\subsection{Geometry of the 1c-ps phase space}
The phase space for 1c-ps Lagrangian distributions (a class of operators that includes suitably microlocalized Poisson operators) is obtained by blowing up the corner of
\begin{align} \label{eq:calM0-def-product-bundle}
\mathcal{M}_0 =  \overline{^\ps{T^*\RR^{n+1}}} \times \overline{^\oc{T^*\RR^n}} 
\end{align}
at base infinity of $\overline{^\oc{T^*\RR^n}}$ and fiber-infinity of $\overline{^\ps{T^*Y}}$. Here the $\RR^n$ factor represents the interior of either $\Rp$ or $\Rm$.  We refer readers to \cite{melrose1993atiyah}\cite{Melrose1994} for more details about blow ups. Concretely, we define
\begin{align} \label{eq: 1c-ps cotangent bundle definition}
\mathcal{M}:= [ \mathcal{M}_0 ; \{ \rho_{\ps} = 0, x_{\oc} =0 \} ], \quad \rhops = \big( \sum_{j,k} g^{jk}(z, t) \zeta_j \zeta_k \big)^{-1/2}, 
\end{align}
and denote the blow down map by 
\begin{align}
\beta_{\oc-\ps}: \mathcal{M} \rightarrow \mathcal{M}_0.
\end{align}
The front face created by this blow-up shall be denoted $\ffocps$. 
The new smooth coordinate on $\ffocps$ introduced by the blow up is 
\begin{equation}\label{eq:sigma}
\varsigma = \frac{x_{\oc}}{\rho_{\ps}},
\end{equation}
or its reciprocal. In the interior of $\ffocps$, either $\xoc$ or $\rhops$ can be taken as a boundary defining function. 

The manifolds $\SM_0$ and $\SM$ have codimension 4 corners. However, we shall only be interested in a neighbourhood of a compact subset $K$ of the interior of $\ffocps$; in particular, we shall stay away from all the other boundary hypersurfaces. So, in effect, we are dealing with a manifold with boundary. 

The manifold $\mathcal{M}$ is endowed with a canonical symplectic structure\footnote{In this article, we allow symplectic structures to blow up or degenerate at the boundary} $\omega$ from $\mathcal{M}_0$ by lifting the symplectic form on $\mathcal{M}_0$, which in turn is equipped with the product symplectic structure from its two factors. We are particularly interested in the symplectic/contact structures on $\ffocps$ induced by this symplectic structure in the interior. To prepare for this, we shall specify coordinates to use in a neighbourhood of $K \subset \ffocps$. These will be $\xoc$ (a boundary defining function) and $\yoc$, which are base coordinates on $\RR^n$ near base infinity;  $\xioc$, $\etaoc$, their one-cusp dual coordinates as defined in Section~\ref{sec:the_1_cusp_pseudodifferential_algebra}; $z$, $t$, Euclidean space and time coordinates;  $\varsigma$ as defined in \eqref{eq:sigma}; and fibre coordinates near fibre-infinity, which we take to be $\tilde \tau = \tau \xoc^2$, $\hat \zeta = \zeta / |\zeta|$. We remark that we are also mostly interested in $(z, t)$ near the perturbation of the metric, which by assumption is a compact set in spacetime. On the other hand, $(\zeta, \tau)$ will be near infinity, since $\rhops = 0$ at $\ffocps$. To summarize, our coordinates are
\begin{equation}\label{eq:coordinates near ffocps}
\xoc, \quad \yoc, \quad \xioc, \quad \etaoc, \quad z, \quad t, \quad \varsigma, \quad \tilde \tau, \quad \hat \zeta.
\end{equation}


\begin{definition}[Admissible 1c-ps Lagrangian submanifold and 1c-ps fibred-Legendre submanifold]  \label{defn: admissible 1c-ps Lagrangian submanifold and 1c-ps fibred-Legendre submanifold}
We define an admissible 1c-ps Lagrangian submanifold of $\SM$ to be a $2n+1$-dimensional submanifold $\Lambda$ that is Lagrangian in the interior (the canonical symplectic form $\omega$ vanishes on it), such that 
\begin{itemize}
\item $\Lambda$ meets $\ffocps$ transversally, 
\item the differential $dt$ is non-vanishing on $\Lambda \cap \ffocps$, and 
\item  its closure is disjoint from all other boundary hypersurfaces of $\SM$ (other than $\ffocps$). 
\end{itemize}
We define a fibred-Legendre submanifold $L$ of $\ffocps$ to be the boundary of an admissible 1c-ps Lagrangian submanifold. In other words, there exists $\Lambda$ as above such that $L = \Lambda \cap \ffocps$. 
\end{definition}
See \cite[Proposition~4.4]{HJ2026-scattering-map} for the reason of the term `fibred' from a symplectic fibration.
The 1c-ps Lagrangian submanifolds that are used for our analysis are the twisted forward and backward sojourn relations which arise from the bulk-boundary duality we discussed in Section~\ref{subsec:bulk-boundary}.
As the calculation in \cite[Section~4.4]{HJ2026-scattering-map} shows, the rescaled Hamilton vector field $H_p^{2,0}$ is a smooth vector field tangent to the boundary of $[\Sigma; \SR]$ except at $W_\pm$, where it is transverse: inward-pointing at $W_-$ and outward-pointing at $W_+$. Using the non-trapping assumption we see that each point of $W_\pm$ gives rise to a smooth integral curve of  $H_p^{2,0}$, that travels from $W_-$ to $W_+$ in finite parameter time $s$. Let $q_-$ be a point of $W_-$, and let  $\gamma_{q_-}(s)$ be the integral curve of $H_p^{2,0}$ emanating from $q_-$ at time $s=0$, and arriving at $W_+$ at time $T(q_-) > 0$. Similarly, let $q_+$ be a point of $W_+$, and let $\mu_{q_+}(s)$ be the integral curve of $H_p^{2,0}$ emanating from $q_+$ at time $s=0$, and arriving at time $T'(q_+) < 0$. 
The forward sojourn relation is the subset of $\SM_0 =  \overline{^\ps{T^*\RR^{n+1}}} \times \overline{^\oc T^*{\Rm}} $ defined by 
\begin{equation}\label{eq:FSR def}
\FSR = \{ (\gamma_{q_-}(s), q_-) \mid q_- \in W_- , \ s \in [0, T(q_-)] \}. 
\end{equation}
Similarly, the backward sojourn relation is defined by 
\begin{equation}\label{eq:BSR def}
\BSR = \{ (\mu_{q_+}(s), q_+) \mid q_+ \in W_+ , \ s \in [T'(q_+), 0] \}. 
\end{equation}

Then our Lagrangian will be lifted and microlocalized version of them, after twisting the sign of the 1-cusp frequencies.
Let $U_\pm \subset W_\pm \sim \overline{ ^{\oc}T^* \R^n} $ be open sets disjoint from fibre-infinity, and let $G_\pm \subset \overline{ ^{\ps}T^* \R^{n+1}}$ be open sets disjoint from spacetime infinity. Consider the microlocalized Lagrangians 
\begin{equation}\label{eq:microlocalized sojourn relns}
\Lambda_\pm = \beta_{\oc-\ps}^*\big(\{ (\gamma_{q_\pm}(s), -q_{\pm}) \in \overline{\Lambda_\pm} \mid q_{\pm} \in U_\pm, \ \gamma_{q_\pm}(s) \in G_\pm  \}\big)
\end{equation}
where $-q_\pm$ means changing the sign of the fiber part and we now view it as being in $\SM$ rather than $\SM_0$. That is, we take these sets to be the closure, in $\SM$, of their interiors lifted to $\SM$ from $\SM_0$ via the blowdown map $\beta_{\oc-\ps}$. 
This $\Lambda_\pm$ depend on the choice of $U_\pm$ and $G_\pm$ but we do not indicate this in the notation, regarding these choices of open sets as fixed. 
As shown in \cite[Proposition~4.12]{HJ2026-scattering-map}, this $\Lambda_\pm$ is admissible in the sense of Definition~\ref{definition:admissible-1c1cLag}.

Discussion above also gives rise to the following map $\msf{I}_\pm$ that we will use later.
For $\tilde{q} \in \Char(P)$, there is a unique $q_\pm \in W_\pm$ such that
\begin{equation}
(\tilde{q}, q_\pm) \in \overline{\Lambda'_\pm}.
\end{equation}
That is, $q_\pm$ is the endpoint of the $H_p^{2,0}$-flow starting from $\tilde{q}$ in the forward/backward direction.
Then we view $q_\pm$ as a point in $\overline{{}^{\oc}T^*\R^n}$ and define
\begin{equation} \label{eq:msf-I-def}
    \msf{I}_\pm(\tilde{q}) = q_\pm \in \overline{{}^{\oc}T^*\R^n}.
\end{equation}
In addition, for $\tilde{q}$ in fiber infinity of $\psphase$, we have $\msf{I}(\tilde{q}) \in \overline{{}^{\oc}T_{\partial\overline{\R^n} }^*\R^n}$.

Next we recall the definition of parameterizations of a 1c-ps fibred Legendre submanifold and its Lagrangian extension, which means giving representations of them (locally) as the graph of the differential of a phase function. Our definition is slightly different with \cite[Definition~4.5]{HJ2026-scattering-map} in that we include the parametrization of the Lagrangian extension as well since this will encode the geometric information we need.

\begin{definition} \label{def:1c-ps-parametrization}
	Suppose $\Legps$ is an 1c-ps fibered-Legendre submanifold of $\ffocps \subset \mathcal{M}$, $q \in \Legps$ and $(\theta_0, \theta_1)$ is in a non-empty open subset $U$ of $\RR^{k_1+k_2}$, then we say that 
\begin{align}
\Phi_{\oc-\ps} =  \frac{ \varphi_0(t,\theta_0) }{x_{\oc}^2} + \frac{\varphi_1(\msf{K},\theta_0,\theta_1)}{x_{\oc}} ,\; \mathsf{K} = (\xoc,\yoc,z,t), 
\end{align}
gives a non-degenerate parametrization of $\Legps$ near $q$ if there is a point $q' \in \RR^{2n}_{\yoc, z, t} \times U_{\theta_0, \theta_1}$ such that the differentials
\begin{equation}\label{eq:varphi0}
d_{t,\theta_0} \frac{\partial \varphi_0}{\partial \theta_{0j}}, \quad 1 \leq j \leq k_1,
\end{equation}
and
\begin{equation}\label{eq:varphi1}
d_{\yoc, z, \theta_1} \frac{\partial \varphi_1}{\partial \theta_{1j}}, 1 \leq j \leq k_2,
\end{equation}
are linearly independent at $q' \in \RR^{2n}_{\yoc, z, t} \times U$, such that $\Legps$ is given locally by 
\begin{align}\label{eq:1c-ps param}
\Legps
=\{ (\msf{K},d_{\msf{K}}\Phi_{\oc-\ps}) \mid \; \xoc = 0, (\yoc, z, t,\theta_0,\theta_1) \in C_{\Phi_{\oc-\ps}}\},
\end{align}
where 
\begin{align}  \label{eq: 1c-ps critical set}
C_{\Phi_{\oc-\ps}} =  \{ (\yoc, z, t,\theta_0,\theta_1) \mid d_{\theta_0}\varphi_0 = 0, d_{\theta_1}\varphi_1 = 0 \}
\end{align}
with the point $q' \in C_{\Phi_{\oc-\ps}}$ corresponding under \eqref{eq:1c-ps param} to $q \in L$. 
Let $\Lagps$ be a Lagrangian extension of $\Legps$, then we say such $\Phi_{\oc-\ps}$ satisfying the independence condition of its differentials above gives a non-degenerate parametrization of $\Lagps$ if 
$\Lagps$ is given locally by 
\begin{align}\label{eq:1c-ps param-Lag}
\Lagps
=\{ (\msf{K},d_{\msf{K}}\Phi_{\oc-\ps}) \mid \;  (x_{\oc},\yoc, z, t,\theta_0,\theta_1) \in \tilde{C}_{\Phi_{\oc-\ps}}\},
\end{align}
where 
\begin{align}  \label{eq: 1c-ps critical set-Lag}
\tilde{C}_{\Phi_{\oc-\ps}} =  \{ (\xoc, \yoc, z, t,\theta_0,\theta_1) \mid d_{\theta_0}\varphi_0 + x_{\oc}d_{\theta_0}\varphi_1 = 0, d_{\theta_1}\varphi_1 = 0 \}.
\end{align}
\end{definition}
As shown in \cite[Proposition~4.6]{HJ2026-scattering-map}, a non-degenerate parametrization of $\Legps$ always exists. In addition, the same proof, except for that allowing functions constructed in the proof to depend on $\xoc$ to represent the Lagrangian extension, proves the existence of non-degenerate parametrization of the Lagrangian extension $\Lagps$ as well. See also the discussion in \cite[Remark~4.8]{HJ2026-scattering-map}.

In addition, using \cite[Lemma~B.1]{HJ2026-scattering-map}, after a linear change of coordinates in $z$ (with $\tilde \zeta$ transforming correspondingly) we have a splitting of coordinates so that in a neighborhood of $q_0 \in \Legps$ the projection
 \begin{align}\label{eq:fullrank}
    \Lambda_\pm \ni (t,z,\tilde{{\tau}},\tilde{\zeta},x_{\oc},y_{\oc},\xi_{\oc},\eta_{\oc}) \rightarrow      (t,z',\tilde{{\zeta}}'',x_{\oc},y_{\oc})
 \end{align}
has full rank $2n+1$, where $z'=(z_1,...,z_k),\ \tilde{\zeta}''=(\tilde{\zeta}_{k+1},...,\tilde{\zeta}_n)$ for some $k \in \{ 1, \dots, n\}$. Under this splitting of coordinates, using \cite[Proposition~4.6, Lemma~B.2]{HJ2026-scattering-map} the phase function used to parametrize $\Legps$ (together with its Lagrangian extension $\Lambda_\pm$) can be taken as the following normal form:
\begin{align} \label{eq:1c-ps-phase-normal-form}
\Phi_{\oc-\ps} = -\frac{t}{x_{\oc}^2}+\frac{z'' \cdot \tilde{\zeta}''-\tilde{\varphi}_1(t,z',\tilde{{\zeta}}'',x_{\oc},y_{\oc}) }{x_{\oc}}.
\end{align}

\subsection{The calculus of 1c-ps Fourier Integral Operators}
\label{subsec:1c-ps-calculus}
In this subsection, we recall the calculus of 1c-ps Fourier Integral Operators developed in \cite[Section~5]{HJ2026-scattering-map}.

\begin{definition} \label{defn: 1c-ps FIO}
Let $\Legps$ be a fibered-Legendre submanifold of $\SM$ and $\Lagps$ be its Lagrangian extension in Definition~\ref{defn: admissible 1c-ps Lagrangian submanifold and 1c-ps fibred-Legendre submanifold}. 
A 1c-ps Legendre distribution associated to  $\Legps$
of order $m$ is a distributional half-density that can be written (modulo $\mathcal{S}(\R^{n+1} \times \R^n)$) as a finite sum of oscillatory integrals of the form 
\begin{multline} \label{eq: 1c-ps FIO definition}
u(\msf{K}) =  (2\pi)^{- ( \frac{2n+1}{4} ) - \frac{k_0+k_1}{2} }
\Big(\int e^{ i( \frac{ \varphi_0(t,\theta_0) }{x_{\oc}^2} + \frac{\varphi_1(\msf{K}',\theta_0,\theta_1)}{x_{\oc}} )} 
x_{\oc}^{-(m+\frac{2k_0+k_1}{2})-\frac{1}{4}}
\\  \times a(\msf{K},\theta_0,\theta_1) d\theta_0d\theta_1 \Big)
|dtdz|^{1/2}|\frac{dx_{\oc}d\yoc}{x_{\oc}^{n+2}}|^{1/2}, \quad \msf{K} = (\xoc, \yoc, t, z), \quad \msf{K}' = (\yoc, t, z),
\end{multline}
where $\Phi_{\oc-\ps} = \frac{ \varphi_0(t,\theta_0) }{x_{\oc}^2} + \frac{\varphi_1(\msf{K},\theta_0,\theta_1)}{x_{\oc}}$ is a local non-degenerate parametrization of $\Legps$ (see Definition~\ref{def:1c-ps-parametrization}), with $\theta_0 \in \R^{k_0}$, $\theta_1 \in \R^{k_1}$, and 
 $a \in C_c^\infty([0,\epsilon)_{\xoc} \times \R^{n-1}_{\yoc} \times \R^{n+1}_{t,z} \times \R^{k_1 + k_2}_{\theta_0, \theta_1})$ is assumed to be supported in the region where $\Phi_{\oc-\ps}$ parametrizes $\Legps$. 
 The set of such Legendre distributions is denoted $I_{\oc-\ps}^m(\R^{n+1} \times \R^n, \Lagps; \Omega^{1/2})$, where $\Omega^{1/2}$ is the half-density bundle whose typical section is as in the expression above.
 A linear operator $A$, mapping half-densities on $\R^n$ to half-densities on $\R^{n+1}$ is called a 1c-ps Fourier Integral Operator of order $m$ associated to $\Legps$ if its Schwartz kernel is a Legendre distribution of order $m$ associated to $\Legps$. 
\end{definition}

Then we recall some key facts we need about 1c-ps Fourier integral operators.
First of all, the oscillatory integral in the definition above can be written as, modulo a Schwartz error, an oscillatory integral using another parametrization, on the region where both parametrizations are valid.

\begin{prop}\cite[Proposition~5.4]{HJ2026-scattering-map}
\label{prop:invariance under parametrization change} Assume that $u$ in \eqref{eq: 1c-ps FIO definition} has essential support in an open set $U \subset \Legps$, and that $\tilde \Phi_{\oc-\ps}$ is another parametrization of $\Legps$ in $U$. Then $u$ can be written, modulo $\mathcal{S}(\R^{n+1} \times \R^n)$, as as oscillatory integral with respect to the parametrizing function $\tilde \Phi_{\oc-\ps}$. 
\end{prop}

Now we discuss the calculus of 1c-ps Fourier integral operators. We first recall the definition of its principal symbol.
Suppose $\lambda = (\lambda_1,...,\lambda_{2n})$ are local coordinates of $\Legps \cap \partial \mathcal{M}$, such that $(x_{\oc},\lambda,\partial_{{\theta}_0}\varphi_0,\partial_{\theta_1}\varphi_1)$ gives a coordinate system of 
$\mathcal{M} \times \R^{N_0}_{\theta_0} \times \R^{N_1}_{\theta_1}$ near the critical set $C_{\Phi_{\oc-\ps}}$ in (\ref{eq: 1c-ps critical set}).
Then the symbol of $A \in I^{m}_{\oc-\ps}(\R^{n+1} \times \R^n, \Lagps)$ written as in (\ref{eq: 1c-ps FIO definition}) with respect to this coordinate system and this parametrization is the half density
\begin{align} \label{eq: 1c-ps principal symbol, half density form}
a(0,\lambda,0,0) |\det \frac{\partial (\lambda,\partial_{\theta_0} \varphi_0,\partial_{\theta_1}\varphi_1)}{\partial (\msf{K}',\theta_0,\theta_1)}|^{-1/2}
|d\lambda|^{1/2} \otimes |d\xoc|^{-m-n/2-5/4}.
\end{align}
This expression is designed to be invariant under changes of parametrization and scaling of the boundary defining function $\xoc$ (hence the inclusion of the  factor $|d\xoc|^{-m-n/2-5/4}$).
To make it invariant under changes of parametrization with different signatures and changes between different choices of boundary defining functions, we define the bundle 
\begin{align} \label{eq: S[m]bundle, 1c-ps case}
S^{[m]}_{\oc-\ps}(\Legps):=  |N^*(\ffocps)|^{-m-\frac{2n+5}{4} } \otimes M(\Legps) \otimes E(\Legps),
\end{align}
where $M(\Legps)$ is the Maslov bundle used in \cite[Section~3.2,3.3]{FIO1} and originates from \cite{maslov1972asymptotic};
and $|N^*(\ffocps)|^{-m-\frac{2n+5}{4}}$ models a section that is homogeneous of degree $-m-\frac{2n+5}{4}$ in $x_{\oc}$. \footnote{This degree can be computed in the case without any extra parameters $\theta_0,\theta_1$, in which case \eqref{eq: 1c-ps FIO definition} has the same homogeneity in $\xoc$ as
$x_{\oc}^{-m-\frac{3}{4}} |\frac{dx_{\oc}d\yoc}{x_{\oc}^{n+2}}|^{1/2}$.}
Here $E(\Legps)$ is the line bundle arises in \cite[Section~3.1]{hassell2001resolvent}, whose transition coefficient depends on the choice of the boundary defining function (or more precisely, the zeroth and first order jet of it).

Then \eqref{eq: 1c-ps principal symbol, half density form} is a section of $\Omega^{1/2}(\Legps) \otimes S_{\oc-\ps}^{[m]}(\Legps)$ associated to the 1c-ps Fourier integral operator $A$ that is independent of the choice of local coordinates and parametrizations.
Then we have the following principal symbol map 
\begin{align} \label{eq: invariant, symbol map, 1c-ps}
\begin{split}
\sigma^m_{\oc-\ps}: \quad   I^m_{\oc-\ps}(\R^{n+1} \times \R^n ,\Lagps; \Omega^{1/2}) 
\rightarrow 
 C^\infty(\Legps \cap \partial \mathcal{M}; 
\Omega^{1/2}(\Legps) \otimes S_{\oc-\ps}^{[m]}(\Legps)).
\end{split}
\end{align}

Then this principal symbol map gives the following short exact sequence in \cite[Proposition~5.5]{HJ2026-scattering-map}, which is the key ingredient to establish the calculus of 1c-ps Fourier integral operators.

\begin{align} \label{eq:1c-ps-exact-sequence}
\begin{split}
0 & \rightarrow I^{m-1}_{\oc-\ps}(\R^{n+1} \times \R^n,\Lagps)
\rightarrow I^{m}_{\oc-\ps}(\R^{n+1} \times \R^n, \Lagps)
\\ & \xrightarrow{\sigma^m_{\oc-\ps}} C^\infty(\Legps \cap \partial \mathcal{M}; 
\Omega^{1/2}(L) \otimes S_{\oc-\ps}^{[m]}(\Legps)) \rightarrow 0.
\end{split}
\end{align}

Next we recall how pseudodifferential operators act on 1c-ps Fourier integral operators.
First we consider the case when we compose a ps-pseudodifferential operator from the left.
\begin{prop}{\cite[Theorem~5.7]{HJ2026-scattering-map}} 
\label{prop: PsiDO- 1c-ps FIO composition}
Suppose $Q \in \Psi_{\ps}^{m',0}(\R^{n+1})$ has parabolically homogeneous principal symbol at fibre-infinity. If $A \in I_{\oc-\ps}^{m}(\R^{n+1} \times \R^n, \Lagps; \Omega^{1/2})$, then we have  
\begin{equation}
QA \in I^{m+m'}_{\oc-\ps}(\R^{n+1} \times \R^n,\Lagps),
\end{equation}
with principal symbol 
\begin{align}  \label{eq: QA principal symbol, non-vanishing q, density version}
\sigma^{m+m'}_{\oc-\ps}(QA) = \sigma^{m'}_{\ps}(Q)|_{\Legps} \otimes \sigma^m_{\oc-\ps}(A),
\end{align}
where we lift the principal symbol of $Q$ to $\SM$, and view as a section of $|N^*(\ffocps)|^{-m'}$ over $\Legps$. 
\end{prop}

When the pseudodifferential operator has vanishing principal symbol on $\Lagps$, then we have the following refined characterization.
\begin{prop}{\cite[Theorem~5.8]{HJ2026-scattering-map}} \label{prop: vanishing principal symbol product}
Suppose $P \in \Psi_{\ps}^{m',0}(\R^{n+1})$, and its parabolically homogeneous principal symbol $p_{\hom}$ vanishes identically on the projection of $\Legps$ in $\psphase$. In addition, let $p$ be its left full symbol and assume that $\tilde{p}(t,z,\tilde{\tau},\tilde{\zeta})=x_{\oc}^{m'}p(t,z,\tau,\zeta)$ is smooth.\footnote{This in particular is satisfied by all differential operators.} 
If $A \in I_{\oc-\ps}^{m}(\R^{n+1} \times \R^n,\Lambda_\pm; \Omega^{1/2})$, then we have 
\begin{align*}
PA \in I^{m+m'-1}_{\oc-\ps}(\R^{n+1} \times \R^n,\Lambda_\pm; \Omega^{1/2}).
\end{align*}
The principal symbol is of $PA$ is as follows: let 
\begin{align*}
 \sigma^m_{\oc-\ps}(A) = \textbf{a} \otimes |dx_{\oc}|^{-m-\frac{2n+5}{4}},
\end{align*}
with $\textbf{a}$ being a section of $\Omega^{1/2}(\Legps) \otimes M(\Legps)$, then
\begin{align} \label{eq: 1c-ps vanishing principal composition, principal symbol} 
\begin{split}
\sigma_{\oc-\ps}^{m+m'-1}(PA) = & (-i\mathscr{L}_{ \rHp} + i (\frac{m'-1}{2} + m + \frac{2n+5}{4} )(x_{\oc}^{-1} \rHp x_{\oc}) + p_{\sub})\textbf{a}
\\& \otimes |dx_{\oc}|^{-m-m'+1-\frac{2n+5}{4}}.
\end{split}
\end{align}
\end{prop}

We can also compose 1c-pseudodifferential operators from the right. As one would expect from the duality in Section~\ref{subsec:bulk-boundary}, the role of differential and decay order are switched compared with the $\ps$-setting.
\begin{prop}{\cite[Theorem~5.9]{HJ2026-scattering-map}} \label{prop: PsiDO- 1c-ps FIO composition-1c-right}
Suppose $Q' \in \Psi_{\oc}^{-\infty,m'}(\R^n)$. If $A \in I_{\oc-\ps}^{m}(\R^{n+1} \times \R^n, \Lagps)$ 
then we have  
\begin{equation}
AQ' \in I^{m+m'}_{\oc-\ps}(\R^{n+1} \times \R^n,\Lagps),
\end{equation}
with principal symbol 
\begin{align}  \label{eq: AQ' principal symbol, non-vanishing q, density version}
\sigma^{m+m'}_{\oc-\ps}(AQ') =   \sigma^m_{\oc-\ps}(A)
\otimes \sigma^{m'}_{\oc}(Q')|_{\Legps}
\end{align}
where we lift the principal symbol of $Q'$ to $\SM$, and view as a section of $|N^*(\ffocps)|^{-m'}$ over $\Legps$. 
\end{prop}

Next we give a characterization of Poisson operators as 1c-ps Fourier integral operators.
The (forward and backward) Poisson operators $\mathcal{P}_\pm$ is the operator sending the `final state data' $f_\pm$ in \eqref{eq:u expansion} to the solution $u$.
With this calculus of 1c-ps Fourier integral operators, we have the following characterization of these two Poisson operators. 
We begin with introducing pseudodifferential operators microlocalizing the Poisson operator to the part associated to $\Lambda_\pm$, which is the part meeting the geometric perturbation.
Let $U_\pm$ and $G_\pm$ be as in \eqref{eq:microlocalized sojourn relns}, we choose 
\begin{equation} \label{eq:Q1c-Qps-def}
    Q_{\oc} \in \Psi_{\oc}^{0,0}(\R^n), \quad Q_{\ps} \in \Psi_{\ps}^{0,0}(\R^{n+1})
\end{equation}
such that $\WF'(Q_{\oc}) \subset U_-$ (hence disjoint from fiber-infinity in $\overline{{}^{\oc} T^* \R^n}$), and $\WF'(Q_{\ps}) \subset G_-$ (hence is away from spacetime infinity). We further assume that 
$Q_{\oc}$ is microlocally equal to the identity on all points in $\overline{{}^{\oc} T^* \R^n}$ whose corresponding (under the identification in Proposition~\ref{prop:Wpm-1c-identification}) bicharacteristics meet the perturbation, i.e. meet $\WF'_{\ps}(P - P_0)$ and
$Q_{\ps} \in \Psi_{\ps}^{0,0}(\R^{n+1})$ is microlocally equal to the identity on $\WF'(P - P_0)$. 
Then our characterization of the microlocalized forward and backward Poisson operators is as follows.
\begin{prop}{\cite[Proposition~5.16]{HJ2026-scattering-map}}
\label{prop:microlocalized Poisson are 1c-ps FIOs} 
Let $Q_{\oc}$ and $Q_{\ps}$ be as above, then 
\begin{equation}
    Q_{\ps} \Poipm Q_{\oc} \in I^{-3/4}_{\oc-\ps}(\R^{n+1} \times \R^n, \Lambda_\pm).
\end{equation}
\end{prop} 

\begin{remark}
As one can see from the definition of our microlocalizers $Q_{\oc},Q_{\ps}$, they are introduced to microlocalize the Poisson operator to the interesting and challenging part meeting the geometric and potential.
Other parts of $\Poipm$ will behave like the Poisson operator associated to the free operator $P_0$, whose kernel has an explicit expression 
\begin{equation} \label{eq:free-Poisson-kernel}
    e^{i(-t|Z|^2+z \cdot Z)},
\end{equation}
where $Z \in \R^n$ is the position variable on the 1-cusp side, and no extra parameter or oscillatory integral representation is needed. To write \eqref{eq:free-Poisson-kernel} in a way that is more compatible with our general expression of the phase function in \eqref{eq: 1c-ps FIO definition}, we set $x_{\oc} = 1/|Z|$, and let $\overline{y_{\oc}} = Z/|Z|$, which can be identified with $y_{\oc} \in \mathbb{S}^{n-1}$. Using those variables, \eqref{eq:free-Poisson-kernel} can be rewritten as 
\begin{equation} \label{eq:free-Poisson-kernel-2}
    e^{i(-\frac{t}{x_{\oc}^2}+\frac{z \cdot \overline{y_{\oc}}}{x_{\oc}})}.
\end{equation}
\end{remark}

We conclude this section by remarking that the reason of investigating forward and backward Poisson operators is that the scattering map can be constructed out of them directly. More precisely, we have:
\begin{equation} \label{eq:S formula} 
 S = i(2\pi)^n \mathcal{P}_+^* [P,Q_+] \mathcal{P}_-,
\end{equation}
where $Q_+ \in \Psi_{\ps}^{0,0}(\R^{n+1})$ is microlocally equal to the identity on a neighborhood of $\mathcal{R}_+$ with $\WF'_{\ps}(Q_+)$ contained in a slightly enlarged neighborhood of $\Radp$. In particular, $\WF'_{\ps}([P,Q_+])$ is away from both of $\mathcal{R}_\pm$. See \cite[Section~7]{gell2022propagation} (with a correction on an overall sign in \cite[Eq~(1.7)]{HJ2026-scattering-map}).

\section{The 1c-1c analysis and the structure of the scattering map}

In this section, we recall main geometric and analytic ingredients needed for the theory $\oc-\oc$ Fourier integral operators from \cite[Section~6, Section~7]{HJ2026-scattering-map}. This is used to characterize our scattering map.

\subsection{Geometry of the 1c-1c phase space}

Let $X$ be a manifold with boundary $\partial X$ (in our case, $X=\overline{\R^n}$), the b-double space introduced by Melrose \cite{melrose1993atiyah}
is defined to be
\begin{align} \label{eq: defn  b-double space}
X_b^2 : = [ X \times X ; \partial X \times \partial X ],
\end{align}
and the blow down map $X_b^2 \rightarrow X^2$ is denoted $\beta_b$.
Then we define the b-lifted 1c-1c cotangent bundle to be
\begin{align} \label{eq:def-lifted-ococ-bundle}
\ococb:= \beta_b^*(\overline{{}^{\oc}{T^*X}} \times \overline{{}^{\oc}{T^*X}}),
\end{align}
where the right hand side is viewed pulling back a bundle over $X^2$ to $X_b^2$.
We denote the corresponding projection map $\ococb \to \ocphase \times \ocphase$ by $\tilde{\beta}_b$.
We denote the lift of $\partial X \times X, X \times \partial X, \partial X \times \partial X$ under $\beta_b$ by $\mathrm{lb}$ (`left boundary'), $\mathrm{rb}$ (`right boundary') and $\mathrm{bf}$ (`b-face') respectively. We denote the part of the bundle $\ococb$ lying over $\mathrm{bf}$ by $\ffococ$. 
Let $(x_{\oc,1}, y_1;x_{\oc,2}, y_2)$ be coordinates on $X^2$, then we use the notation
\begin{align*}
x_{\oc} = x_{\oc,1}, \; \sigma = \frac{x_{\oc,1}}{x_{\oc,2}},
\end{align*}
and $\msf{X}:=(x_{\oc},\sigma, y_1, y_2)$ forms a local coordinate system of $X_b^2$ on the region $\{ C^{-1} \leq \sigma \leq C \}$ for a fixed $C$, which is the interesting part for us.
The canonical 1-form on $\ococb$ is given by the sum of the canonical one form lifted from the left and right factors:
\begin{equation}\label{eq:1c1c-1form--1}
 \alpha_{\oc-\oc} = \xioco \frac{d\xoco}{\xoco^3} + \etaoco \frac{d\yoco}{\xoco} + \xioct \frac{d(\xoco/\sigma)}{(\xoco/\sigma)^3} + \etaoct \frac{d\yoct}{(\xoco/\sigma)},
\end{equation}
where we used  $\xoct = \xoco/\sigma$. This can be rewritten as:
\begin{align}
\label{eq:1c1c-1form-0}
\begin{split}
 \alpha_{\oc-\oc} =
( \xioco + \sigma^2 \xioct) \frac{d\xoco}{\xoco^3} - (\sigma \xioct) \frac{d\sigma}{\xoco^2} + \etaoco \frac{d\yoco}{\xoco} + \sigma \etaoct \frac{d\yoct}{\xoco}.
\end{split}
\end{align}
Then we can define our $\oc-\oc$ frequencies to be coefficients in this form:
\begin{equation} \label{eq:1c1c-1form-1}
 \alpha_{\oc-\oc} = \nu_1 \frac{d\xoco}{\xoco^3} + \nu_2 \frac{d\sigma}{\xoco^2} + \etaoco \frac{d\yoco}{\xoco} +  \tilde{\eta}_{\oc,2} \frac{d\yoct}{\xoco},
\end{equation}
and its differential gives our symplectic form 
\begin{equation}\label{eq:omega 1c1c}
\omega_{\oc-\oc} = d \alpha_{\oc-\oc}. 
\end{equation}
Then the class of Legendre and Lagrangian submanifolds we will use is the following.

\begin{definition}[Admissible 1c-1c Lagrangian submanifold and 1c-1c fibred-Legendre submanifold]  \label{definition:admissible-1c1cLag}
We define an admissible $\oc-\oc$ Lagrangian submanifold of $\ococb$ to be a $2n$-dimensional submanifold $\Uplambda$ that is Lagrangian in the interior (that is, the symplectic form $\omega_{\oc-\oc}$ from \eqref{eq:omega 1c1c} vanishes on it), such that 
\begin{itemize}
\item $\Uplambda$ meets $\ffococ$ transversally, 
\item the differential $d\xioco$ is nonvanishing on $\Uplambda \cap \ffococ$, and 
\item  its closure is disjoint from all other boundary hypersurfaces of $\ococb$ (other than $\ffococ$). 
\end{itemize}
We define a fibered-Legendre submanifold $\SL$ of $\ffococ$ to be the boundary of an admissible 1c-1c Lagrangian submanifold. In other words, there exists $\Uplambda$ as above such that $\SL = \Uplambda \cap \ffococ$. 
\end{definition}

See \cite[Proposition~6.3]{HJ2026-scattering-map} for the symplectic fibration justifying the name.
Similar to Definition~\ref{def:1c-ps-parametrization}, we will define the parametrization of both Legendre and Lagrangian submanifolds in Definition~\ref{definition:admissible-1c1cLag}.

\begin{definition}\label{defn: 1c-1c parametrization clean}
Suppose $\Legps$ is an 1c-1c fibred-Legendre submanifold of $\ffococ \subset \mathcal{M}$, $q \in \SL$ and $(\ococparaone, \ococparatwo)$ is in a non-empty open subset $U$ of $\RR^{k_0+k_1}$. We say that 
\begin{align}\label{eq:Phi 1c-1c param clean}
\Phi_{\oc-\oc}(\msf{K}, \ococparaone, \ococparatwo) =  \frac{ \varphi_0(\sigma,\ococparaone) }{x_{\oc,1}^2} + \frac{\varphi_1(\msf{K},\ococparaone,\ococparatwo)}{x_{\oc,1}} , \quad \msf{K} = (\xoco, \sigma, \yoco, \yoct), 
\end{align} 
gives a \emph{clean parametrization} of $\SL$ near $q$ with excess $e$ if there is a point $q' = (\msf{K}, \ococparaone', \ococparatwo')$ such that the differentials
\begin{equation} \label{eq:varphi0 1c clean}
d_{\sigma,\ococparaone} \frac{\partial \varphi_0}{\partial v_j}, \quad 1 \leq j \leq k_0,
\end{equation}
are linearly independent at $(\sigma, \ococparaone')$, and the differentials 
\begin{equation}\label{eq:varphi1 1c clean}
d_{\yoco, \yoct, \ococparatwo} \frac{\partial \varphi_1}{\partial w_j}, 1 \leq j \leq k_1,
\end{equation}
have a fixed rank $k_1 - e$ near $q'$, such that $\SL$ is given locally by 
\begin{align}\label{eq:1c-1c param clean}
\SL
=\{ (\msf{K},d_{\msf{K}}\Phi_{\oc-\oc}(\msf{K},\ococparaone,\ococparatwo)) \mid \; \xoc = 0, (\msf{K},\ococparaone,\ococparatwo) \in C_{\Phi_{\oc-\oc}}\},
\end{align}
where 
\begin{align}  \label{eq: 1c-1c critical set clean}
C_{\Phi_{\oc-\oc}} =  \{ (\msf{K},\ococparaone,\ococparatwo) \mid d_{\ococparaone}\varphi_0 = 0, d_{\ococparatwo}\varphi_1 = 0 \}.
\end{align}
When $e=0$, we say this parametrization is \emph{non-degenerate}.
As before, either $\ococparaone$ or $\ococparatwo$ may be absent, in which case conditions for the derivatives in \eqref{eq:varphi0 1c clean}, resp. \eqref{eq:varphi1 1c clean} are dispensed with, as well as stationarity with respect to $\ococparaone$, resp. $\ococparatwo$ in \eqref{eq: 1c-1c critical set clean}. 

Let $\Lag$ be a Lagrangian extension of $\Leg$, then we say such $\Phi_{\oc-\oc}$ satisfying the constant rank condition of its differentials above gives a clean parametrization of $\Lag$ if 
$\Lag$ is given locally by 
\begin{align}\label{eq:1c-1c param-Lag}
\Lag
=\{ (\msf{K},d_{\msf{K}}\Phi_{\oc-\oc}) \mid \;  (\msf{K},v,w) \in \tilde{C}_{\Phi_{\oc-\oc}}\},
\end{align}
where 
\begin{align}  \label{eq: 1c-1c critical set-Lag}
\tilde{C}_{\Phi_{\oc-\oc}} =  \{ (\msf{K},v,w) \mid d_{v}\varphi_0 + \xoco d_{v}\varphi_1 = 0, d_{w}\varphi_1 = 0 \}.
\end{align}
When $e=0$ in the constant rank condition, we say this parametrization is non-degenerate.
\end{definition}

From \eqref{eq:S formula}, the 1c-1c Fourier integral operator that we consider will arise from composing an 1c-ps Fourier integral operator and its adjoint.
So if we use the phase function in \eqref{eq:1c-ps-phase-normal-form} in the 1c-ps Fourier integral operators, then the corresponding composition will be the difference of two such functions. Then we have the following normal form for parametrizing 1c-1c Legendre submanifolds (and their Lagrangian extensions):
\begin{align} \label{eq:1c-1c-phase-normal-form}
\Phi_{\oc-\oc}(\msf{K},t, \theta_1) =  \frac{ t(1-\sigma^2) }{\xoco^2} + \frac{\varphi_1(\msf{K},t,\theta_1)}{\xoco} , \quad \msf{K} = (\xoco, \sigma, \yoco, \yoct).
\end{align}

\subsection{The calculus of 1c-1c Fourier integral operators and the scattering map}
\label{subsec:1c-1c-calculus}
In this subsection, we recall the calculus of 1c-1c Fourier integral operators, which is used to characterize our scattering map. 

\begin{definition} \label{def:1c-1c-FIO}
Let $\Uplambda$ be an admissible Lagrangian submanifold of $\ococb$ as in Definition~\ref{definition:admissible-1c1cLag}, with boundary $\SL$. 
We define 
$I^{m}_{\oc-\oc}(X_b^2, \SL; \Omega_{\oc-\oc}^{1/2})$, i.e., the space of 1c-1c fibered-Legendre distributions of order $m$, to be the space of operators with Schwartz kernel given (modulo a Schwartz function) by a finite sum of terms of the form 
\begin{align} \label{eq: 1c-1c FIO, local form}
\begin{split}
 & (2\pi)^{-\frac{n+(k_0+k_1-e)}{2}} \Big(\int 
e^{i\Phi_{\oc-\oc}(\msf{K},v,w)} 
 a(\msf{K},v,w)
\\& x_{\oc,1}^{-m-\frac{2k_0+(k_1-e)}{2}+ \frac{n+1}{2}} dvdw \Big) 
|\frac{d\sigma dy_{\oc,2}}{x_{\oc,1}^{n+1}}|^{1/2}|\frac{dx_{\oc,1}dy_{\oc,1}}{x_{\oc,1}^{n+2}}|^{1/2},
\end{split}
\end{align}
where $\Phi_{\oc-\oc}$ is a phase function of the form \eqref{eq:Phi 1c-1c param clean} locally parametrizing $\mathcal{L}$ non-degenerately in the sense of Definition \ref{defn: 1c-1c parametrization clean}.
Moreover,  $a \in C^\infty_c([0,\infty)_{x_{\oc}} \times [C^{-1},C]_{\sigma} \times \mathbb{S}^{n-1} \times \mathbb{S}^{n-1} \times \R^{k_0} \times \R^{k_1})$, where $v \in \R^{k_0},w \in \R^{k_1}$. 
\end{definition}

In the same way as in Proposition~\ref{prop:invariance under parametrization change} above, 
over a region on which two phase functions are clean parametrizations, the oscillatory integral above using one of them can be written as an oscillatory integral using the other as the phase function, modulo a Schwartz error. In addition, allowing clean phase function in \eqref{eq: 1c-1c FIO, local form} instead of using non-degenerate phase functions only does not enlarge the operator class, see \cite[Remark~7.2, Proposition~7.3]{HJ2026-scattering-map}.

Now we begin the discussion about the calculus of 1c-1c- Fourier integral operators by the principal symbol. 
For $A$ in the form (\ref{eq: 1c-1c FIO, local form}), let $\msf{X}'=(\sigma,\yoco,\yoct)$, and let $\lambda$ be a set of functions of $(\msf{X}',v,w)$ such that $(x_{\oc},\lambda,d_{v}\varphi_0+x_{\oc}d_v\varphi_1,d_w\varphi_1)$ form a local coordinate system near $C_{\Phi}$.
When $e=0$ (that is, when the parametrization is non-degenerate), we define the principal symbol of $A$ to be the half density on $\SL \cap \partial X_b^2$ given by:
\begin{align} \label{defn: 1c-1c principal symbol, non-degenerate}
a(0,\msf{X}',v,w)|_{C_{\Phi_{\oc-\oc}}} | \frac{\partial(d_v\varphi_0,d_{w}\varphi_1,\lambda)}{\partial(\msf{X}',v,w)}|^{-1/2} |d\lambda|^{1/2}.
\end{align}
More generally, let $\Phi_{\oc-\oc}$ be a clean phase function as in \eqref{eq:Phi 1c-1c param clean}, then there is a (local) splitting of $\ococparatwo$:
\begin{equation} \label{eq:ococpara-splitting}
   \ococparatwo = (\ococparatwo',\ococparatwo''), \; \ococparatwo' \in \R^{k_1-e} , \ococparatwo'' \in \R^e
\end{equation}
such that for each fixed $\ococparatwo''$, $\Phi_{\oc-\oc}$ gives a non-degenerate parametrization of $\Uplambda$, which means that there is a fibration of $C_{\Phi_{\oc-\oc}}$ (the critical set of $\Phi_{\oc-\oc}$) over $\Uplambda$ with $e$-dimensional fiber parametrized by $\ococparatwo''$.
One can also view such a clean phase function as a family of non-degenerate phase functions parametrized by $w''$.
In the same way as in \cite[Equation~(7.6), Lemma~7.2]{duistermaat-guillemin1975spectrum}, integrating the amplitude along such fiber parametrized by $w''$ gives our half-density as in \eqref{defn: 1c-1c principal symbol, non-degenerate} when we parametrize using non-degenerate phase functions. More precisely, in this setting we have:
\begin{align} \label{defn: 1c-1c principal symbol, clean}
\int_{C_q} a(0,\msf{X}',v,w)  d\ococparatwo''
| \frac{\partial(d_v\varphi_0,d_{w'}\varphi_1,\lambda)}{\partial(\msf{X}',v,w')}|^{-1/2} |d\lambda|^{1/2}.
\end{align}
Here $q \in \Uplambda$ and $C_q$ is the fiber in $C_{\Phi_{\oc-\oc}}$ over $q$.
Then this is invariant (except for Maslov factors and $E(\SL)$ factors) in the leading order, which can be shown in the same way as in \cite[Section~25.1]{hormander2009analysis}. 
Notice that here we are only concerning the clean parametrization with excess in $\varphi_1$, hence the parabolic scaling between $\varphi_0$ and $\varphi_1$ part is not involved and the proof for invariance under change of coordinates in the homogeneous case, after converting to polar coordinates, almost applies here directly.

In a manner parallel to the 1c‑ps setting, to make this concept of principal symbol invariant under change of parametrizations and choice of boundary defining functions, we introduce
\begin{align} \label{eq:S[m]-bundle}
S_{\oc-\oc}^{[m]}(\SL):= |N^*\partial X_b^2|^{- m - \frac{n+1}{2} }
\otimes M(\SL) \otimes E(\SL),
\end{align}
where $M(\SL)$ is the Maslov bundle used in \cite[Section~3.2,3.3]{FIO1}, which in turn comes from \cite{maslov1972asymptotic}; and $|N^*\partial X_b^2|^{-m-\frac{n+1}{2}}$, models a section that is homogeneous of degree $-m-\frac{n+1}{2}$ in $x_{\oc}$.
The order $-m-\frac{n+1}{2}$ comes from considering the homogeneous degree of:
\begin{align*}
  x_{\oc}^{-m + \frac{n+1}{2} }|\frac{dx_{\oc}dy_{\oc,1}}{x_{\oc}^{n+2}}|^{1/2}
|\frac{d\sigma dy_{\oc,2}}{x_{\oc}^{n+1}}|^{1/2},
\end{align*}
while the factor $x^{- \frac{2k_1+k_2}{2} }$ is encoded in $| \frac{\partial(d_v\varphi_0,d_w\varphi_1,\lambda)}{\partial(\msf{X},v,w)}|^{-1/2}$ in an invariant manner. Similar to the 1c-ps setting, the line bundle $E(\SL)$ here is introduced to make this definition invariant under change of boundary defining functions, see \cite[Section~3.1]{hassell2001resolvent} for more details.
Combining discussions above, we can view the symbol map as
\begin{align} \label{eq: invariant, symbol map, 1c-1c}
\begin{split}
\sigma^m_{\oc-\oc}: \quad  I^m_{\oc-\oc}(X_b^2,\Uplambda;\Omega_{\oc-\oc}^{1/2}) 
\rightarrow 
 C^\infty(\SL \cap \partial X_b^2; 
\Omega^{1/2}(\SL) \otimes S_{\oc-\oc}^{[m]}(\SL)).
\end{split}
\end{align}
We say $A \in I^m_{\oc-\oc}(X_b^2,\Uplambda;\Omega_{\oc-\oc}^{1/2})$ is elliptic at $q \in \SL$ if the amplitude on right hand side of \eqref{defn: 1c-1c principal symbol, non-degenerate} or \eqref{defn: 1c-1c principal symbol, clean} is non-vanishing at the point that is sent to $q$ under the parametrization map in Definition~\ref{defn: 1c-1c parametrization clean}. This is equivalent to requiring the principal symbol to be invertible near $q$ as the section of $\Omega^{1/2}(\SL) \otimes S_{\oc-\oc}^{[m]}(\SL)$.
The fact that this principal symbol map captures leading order singularity can be summarized in the following short exact sequence from \cite[Proposition~7.4]{HJ2026-scattering-map}:
\begin{align} \label{eq:1c-1c-FIO-short-exact-sequence}
\begin{split}
0  \rightarrow I^{m-1}_{\oc-\oc}(X_b^2,\SL)
\rightarrow I^{m}_{\oc-\oc}(X_b^2,\SL)
\xrightarrow{\sigma^m_{\oc-\oc}} C^\infty(\SL \cap \partial X_b^2; 
\Omega^{1/2}(\SL) \otimes S_{\oc-\oc}^{[m]}(\SL)) \rightarrow 0.
\end{split}
\end{align}

Next we recall results concerning the composition of $\oc-\oc$ Fourier integral operators.
Let $\mathcal{L}_i = C_i' \cap \ffococ$, $i=1,2$ be admissible 1c-1c Legendre submanifolds with $C_i'$ being the corresponding Lagrangian submanifolds.
Suppose $C_1,C_2$ are lifts of canonical relations in $\leftidx{^{\oc}}{T^*X} \times \leftidx{^{\oc}}{T^*X}$ to $\ococb$, and $C_2 \times C_1$ 
intersect the lift to $\ococb \times \ococb$ of the diagonal in the second and third components of
\begin{align*}
\leftidx{^{\oc}}{T^*X} \times \leftidx{^{\oc}}{T^*X} \times \leftidx{^{\oc}}{T^*X} \times \leftidx{^{\oc}}{T^*X}
\end{align*}
transversally. 
Recall from \cite[Theorem~4.2.2]{FIO1} 
that there is a natural bilinear map giving the product of density bundles:
\begin{align}  \label{eq: symbol product, bundle maps}
\begin{split}
& \Omega^{1/2}(\mathcal{L}_2) \otimes S^{[m_2]}(\mathcal{L}_2)
 \times \Omega^{1/2}(\mathcal{L}_1) \otimes S^{[m_1]}(\mathcal{L}_1) 
 \\ & \rightarrow  \Omega^{1/2}((C_2 \circ C_1)' \cap \ffococ) \otimes S^{[m_2+m_1]}( (C_2 \circ C_1) ' \cap \ffococ).
\end{split}
\end{align}
Denoting this bilinear map by `$\times$', we have the following composition law in the calculus of $\oc-\oc$ Fourier integral operators.
\begin{prop}{Adapted version of \cite[Proposition~7.5]{HJ2026-scattering-map}}
\label{prop: 1c-1c transversal composition}
Suppose $C_1 \times C_2$ satisfies the transversal intersection condition above, and $A_1 \in I^{m_1}_{\oc-\oc}(X^2_b,C_1'), A_2 \in I^{m_2}_{\oc-\oc}(X^2_b,C_2')$, then we have
\begin{align*}
A_2A_1 \in I^{m_1+m_2}_{\oc-\oc}(X_b^2, (C_2 \circ C_1)').
\end{align*}
And when they have $a_1,a_2$ as their principal symbol respectively, then $A_2A_1$ has principal symbol
\begin{align*}
a_2 \times a_1.
\end{align*}
In addition, when $C_1,C_2$ are both graphs of symplectomorphisms, if $A_1,A_2$ are elliptic (in the sense defined after \eqref{eq:1c-1c-FIO-short-exact-sequence}) at $q_1',q_2'$ such that $(q_2,q_1)$ is sent to $q \in C_2 \circ C_1$, then $A_2A_1$ is elliptic at $q'$.
\end{prop}

Now we define the classical scattering map $\Cl$, whose graph will have a natural correspondence to bicharacteristic lines and will be used to define our 1c-1c Lagrangian submanifold.
Given $q \in \overline{{}^{\oc}T^* \SR_-}$, under the identification with $W_-$ in Proposition~\ref{prop:Wpm-1c-identification}, there is a unique bicharacteristic that tends to it in the backward direction.
This bicharacteristic will tends to a point $q' \in \overline{{}^{\oc}T^*\SR_+}$ in the forward direction, again after identified with $W_+$ as above.
Then we define
\begin{equation} \label{eq:classical-sc-map-def}
    \Cl(q) = q'.
\end{equation}
This is a smooth map since the flow of $H_p^{2,0}$ reaches $W_\pm$ in finite time transversally.

The 1c-1c Lagrangian submanifold that we are going to use for our scattering map is the following.
Let $\tilde{\beta}_b$ be the projection map $\ococb \to \overline{{}^{\oc}{T^*X}} \times \overline{{}^{\oc}{T^*X}}$, and use $\mathrm{Gr}(\Cl)'$ to denote the twisted (i.e. with the sign of the frequency variable of the second component flipped) graph of $\Cl$ in $\overline{{}^{\oc}{T^*X}} \times \overline{{}^{\oc}{T^*X}}$, then the admissible 1c-1c Lagrangian submanifold that we will use is $\tilde{\beta}_b^*(\mathrm{Gr}(\Cl)')$, which is defined to be the closure of the preimage of the interior part of $\mathrm{Gr}(\Cl)'$.
Then the characterization of the scattering map using our calculus of 1c-1c Fourier integral operators is the following.
\begin{thm}{\cite[Theorem~1.1]{HJ2026-scattering-map}}
\label{thm: sc-map-HJ}
The scattering map $S$ is an elliptic 1-cusp Fourier integral operator of order zero, with canonical relation the graph of the classical scattering map. 
\begin{align}\label{eq:S 1c-1c FIO}
S \in I^0_{\oc-\oc}(X_b^2,\tilde{\beta}_b^*(\mathrm{Gr}(\Cl)')), \quad X = \overline{\R^n}. 
\end{align}
The scattering map $S$ acts as the identity microlocally on functions (asymptotic data) supported in a compact subset of $\R^n$, or are supported microlocally near frequency-infinity in the 1-cusp sense. 
\end{thm}

We conclude this section by summarizing how $\Cl$ relates to geodesics of $g(t)$.
In the correspondence of Proposition~\ref{prop:Wpm-1c-identification}, the position variable in $\overline{{}^{\oc}T^*\R^n}$ corresponds to the frequency variable in $\overline{{}^{\ps}T^*\R^{n+1}}$. 
As discussed after \eqref{eq:CharP-def}, the only bicharacteristics meeting the metric and potential perturbations are those ones with infinite frequency.
Under the identification in Proposition~\ref{prop:Wpm-1c-identification}, these bicharacteristics in fiber infinity of $\overline{{}^{\ps}T^*\R^{n+1}}$ are identified with points in 
$\overline{{}^{\oc}T^*_{\partial \overline{\R^n}}\R^n}$.
Then in terms of $\mathrm{Gr}(\Cl)$ above, they will be canonically identified with points in
\begin{equation}
  \tilde{\beta}_b^*(\mathrm{Gr}(\Cl)') \cap  \ffococ.
\end{equation}
The part of $\Char(P)$ at fiber infinity over $B_R(0)$, after projecting out the frequency dual to $t$ and identify the fiber infinity as a sphere, can be identified with
\begin{equation} \label{eq:parametrized-copehre-bundle-def}
    \mk{B}_g = [-T,T] \times S^*_gB_R(0),
\end{equation}
which is a family of the sphere bundle $S^*_g\R^n$ parametrized by time $t \in [-T,T]$.
More concretely, we introduce:
\begin{equation}  \label{eq:iota-def}
\iota: \;  \quad (t,z,v) \in \mk{B}_g   \to  (t,z,0,-1,v) \in \Char(P),
\end{equation}
where the coordinates are as in \eqref{eq:ps-coordinates},
with $\hat{\zeta}$-part parametrized by $v$, the fiber part of $S^*_gB_R(0)$.

When we send a point from $\overline{{}^{\oc}T^*_{\partial \overline{\R^n}}\R^n}$ to a bicharacteristic line, $t$ is determined by $t=-\frac{1}{2}\xi_{\oc}$ using \eqref{eq:xioc-etaoc-t-z-relation} and fixed over the entire bicharacteristic line.
Using $\iota$ in \eqref{eq:iota-def}, each bicharacteristic line, after forgetting the (rescaled) frequency dual to $t$ and the $\rho_{\ps}$ component, is just a geodesic (lifted to the cosphere bundle) of $g(t)$. 
In addition, those geodesics enter $\mk{B}_g$, or the image of $\mk{B}_g$ under $\iota$ above, corresponds to a region of $\overline{{}^{\oc}T^*_{\partial \overline{\R^n}}\R^n}$ on which 1-cusp frequencies are bounded.
So discussion above also gives the correspondence between these geodesics and points in $\tilde{\beta}_b^*(\mathrm{Gr}(\Cl)') \cap  \ffococ$, which plays the role of `end points' of those geodesics.
\begin{prop}  \label{prop:classical-SC-geodesic}
Let $\xi_{\oc}$ be the frequency as in \eqref{eq: 1c- canonical form} lifted from the left factor to $\ococb$, then each point in $\tilde{\beta}_b^*(\mathrm{Gr}(\Cl)') \cap  \ffococ$ corresponds to a geodesic of $g(t)$ with $t=-\frac{1}{2}\xi_{\oc}$, so that its left and right projections to $\overline{{}^{\oc}T^*_{\partial \overline{\R^n}}\R^n}$ corresponds to the initial and ending point of the geodesic as above.
In addition, those geodesics enter $\mk{B}_g$ corresponds to points $\overline{{}^{\oc}T^*_{\partial \overline{\R^n}}\R^n}$ with 1-cusp frequency in a bounded region.
\end{prop}

Next we recall \cite[Theorem~8.3]{HJ2026-scattering-map}, which is motivated by \eqref{eq:S formula} and will be used to give the expression of the principal symbol of the scattering map.
\begin{prop}{\cite[Theorem~8.3]{HJ2026-scattering-map}} \label{prop: 1c-ps composition to 1c-1c}
Suppose $A_- \in I_{\oc-\ps}^{m_-}(\R^{n+1} \times \R^n,\Lambda_-)$, $A_+ \in I_{\oc-\ps}^{m_+}(\R^{n+1} \times \R^n,\Lambda_+)$, we have
\begin{align}
A_+^*A_- \in I_{\oc-\oc}^{m_++m_-+\frac{1}{2}}(X_b^2,\beta_b^*(\mathrm{Gr}(\Cl))).
\end{align}
The principal symbol of $A_+^*A_-$ at $(q-, q_+)$, where $q_- \in W_-$ and $q_+ = \Cl(q_-)$, is 
\begin{align} \label{eq: principal symbol, 1c-ps composition to 1c-1c}
\int_{-\infty}^\infty \overline{\sigma_{1c-ps}^{m_+}(A_+)} \sigma_{1c-ps}^{m_-}(A_-) (q_-, \gamma_{q_-}(s)) ds ,
\end{align}
where the integral is a global section of 
\begin{equation}
\Omega^{1/2}(\beta_b^*(\mathrm{Gr}(\Cl))) \otimes S_{\oc-\oc}^{[m_-+m_++\frac{1}{2}]}(\beta_b^*(\mathrm{Gr}(\Cl))).    
\end{equation}
\end{prop}



\section{The analysis of phase functions}

\subsection{Some further analysis on the forward and backward sojourn relations and the classical scattering map}

In this subsection, we examine in greater detail how $\Lambda_\pm$ and $\tilde{\beta}_b^*(\mathrm{Gr}(\Cl))$ intersect the front face in the 1c-ps and 1c-1c setting respectively. 
We will show that $\Lambda_\pm$ (resp. $\tilde{\beta}_b^*$) only intersect $\ffocps$ (resp. $\ffococ$) at the `middle' of the front face, i.e. at $\varsigma = 1$ for $\ffocps$ (resp. at $\sigma = 1$ in $\ffococ$). We consider $\Lambda_\pm$ first and the claim for $\tilde{\beta}_b^*(\mathrm{Gr}(\Cl))$ will follow as a corollary. 

\begin{prop} \label{prop:1cps-sigma-equals-1}
Let $\Lambda_\pm$ be the microlocalized sojourn relation defined in \eqref{eq:microlocalized sojourn relns}, then we have
\begin{equation} \label{eq:1cps-sigma-equals-1}
\Legps = \partial \Lambda_\pm \subset \{ \varsigma = 1 \},
\end{equation}
where $\varsigma = x_{\oc}/\rho_{\ps}$ with the choice $\rho_{\ps} = (\sum_{i,j=1}^n g^{ij}\zeta_i\zeta_j)^{-1/2}$ and $x_{\oc} = \frac{2|t|}{|z|}$ near $\partial \mathcal{R}_\pm$.
\end{prop}

\begin{proof}
Consider the case of $\Lambda_-$ for definiteness. 
Recall that for the part of $\Lambda_\pm$ with $(t,z)$ before entering the metric perturbation, it coincide with the forward sojourn relation associated to the free operator $P_0$ and \eqref{eq:1cps-sigma-equals-1} is satisfied. See the proof of \cite[Proposition~4.12]{HJ2026-scattering-map} for more details.
Now we extend this property to the entire $\partial \Lambda_\pm$ even after meeting the metric perturbation.

Recall that at finite frequency, the Hamilton vector field is given by (ignoring the potential, which won't contribute to the leading order behavior at fiber infinity):
\begin{equation}
H_p = \partial_t + 2 \sum_{\ell,k=1}^n g^{\ell k}\zeta_i \partial_{z_k}
- \sum_{k,i,j=1}^n\partial_{z_k}g^{ij} \zeta_i\zeta_j \partial_{\zeta_k}.
\end{equation}
Next we show that the coefficient of $\partial_{\varsigma}$ of its extension lifted to $\SM$ vanishes when restricted to $\Legps = \partial \Lambda_- \subset \{ x_{\oc} = 0 = \rho_{\ps} \}$.
Since we are away from $\partial \overline{\R^{n+1}}$, we can take the boundary defining function of the spacetime boundary to be $x_{\ps}=1$ and we have $H_p^{2,0} = \rho_{\ps}H_p$.
A direct computation shows
\begin{equation}
    \frac{\partial \rho_{\ps}}{\partial \zeta_k}
    = \frac{\partial(\sum_{i,j=1}^n g^{ij}\zeta_i\zeta_j)^{-1/2}}{\partial \zeta_k} = 
    (\sum_{i,j=1}^n g^{ij}\zeta_i\zeta_j)^{-3/2}
    \sum_{i=1}^n g^{ik}\zeta_i = \rho_{\ps}^2\varsigma^{-1}\sum_{i=1}^n g^{ik}\tilde{\zeta}_i,
\end{equation}
where we used $\varsigma = x_{\oc}/\rho_{\ps}$ and $\tilde{\zeta}_i=x_{\oc}\zeta_i$. Consequently, we have
\begin{equation}
    \frac{\partial \varsigma}{\partial \zeta_k}
    = -\frac{x_{\oc}}{\rho_{\ps}^2} \frac{\partial \rho_{\ps}}{\partial \zeta_k}
    = -\rho_{\ps}\sum_{i=1}^n g^{ik}\tilde{\zeta}_i.
\end{equation}
Now we investigate the $\partial_{\varsigma}$-component of $H_p^{2,0}$.
Similarly, we have
\begin{equation}
  \frac{\partial \varsigma}{\partial z_k} =  
  -\frac{x_{\oc}}{\rho_{\ps}^2} \frac{\partial \rho_{\ps}}{\partial z_k} = - \frac{1}{2} \varsigma^{-1}  \sum_{i,j=1}^n \partial_{z_k}g^{ij}\tilde{\zeta}_i\tilde{\zeta}_j.
\end{equation}

So the coefficient of $\partial_{\varsigma}$ in $H_p^{2,0} = \rho_{\ps}H_p$ from terms other than $\partial_t$ is
\begin{align}
\begin{split}
& 2 \varsigma^{-1} \sum_{\ell,k=1}^n g^{\ell k}\tilde{\zeta}_\ell \frac{\partial \varsigma}{\partial z_k}
- \rho_{\ps} \sum_{k,i,j=1}^n\partial_{z_k}g^{ij} \zeta_i\zeta_j  \frac{\partial \varsigma}{\partial \zeta_k} 
\\ = &  \varsigma^{-1} \sum_{\ell,k=1}^n g^{\ell k}\tilde{\zeta}_\ell \big( -\varsigma^{-1} \sum_{i,j=1}^n \partial_{z_k}g^{ij}\tilde{\zeta}_i\tilde{\zeta}_j \big)
-  \rho_{\ps} \sum_{k,i,j=1}^n\partial_{z_k}g^{ij} \zeta_i\zeta_j  \big( -\rho_{\ps}\sum_{\ell=1}^n g^{\ell k}\tilde{\zeta}_\ell \big) = 0.
\end{split}
\end{align}
The contribution from $\partial_t$ is
\begin{equation}
    \rho_{\ps}\frac{\partial \varsigma}{\partial t} =
    \rho_{\ps} \Big(- \frac{1}{2} \varsigma^{-1}  \sum_{i,j=1}^n \partial_{t}g^{ij}\tilde{\zeta}_i\tilde{\zeta}_j\Big) = O(\rho_{\ps}),
\end{equation}
hence vanishes on $\Legps$.
In sum, $H_p^{2,0}$ restricted to $\Legps$ has vanishing coefficient in $\partial_{\varsigma}$, hence $\varsigma$ will remain constant $1$ by the flow-out definition of $\Lambda_-$.
\end{proof}

The similar property of $\tilde{\beta}_b^*(\mathrm{Gr}(\Cl))$ follows from the proposition above.
\begin{coro} \label{coro:1c1c-sigma-equals-1}
The boundary $\SL = \partial \tilde{\beta}_b^*(\mathrm{Gr}(\Cl)')$ satisfy:
\begin{equation}
 \SL \subset  \{ \sigma = 1 \}.
\end{equation}
\end{coro}

\begin{proof}
We first recall that $\tilde{\beta}_b^*(\mathrm{Gr}(\Cl))$ is the composition of $\Lambda_-'$ and $(\Lambda_+^*)' \subset \SM^* $, where $\Lambda_+^*$ and $\SM^*$ are manifolds obtained by switching the ps and 1c-variables of $\Lambda_+$ and $\SM$. See \cite[Lemma~8.1]{HJ2026-scattering-map} for more details about dealing with extra blowups here.
Then the composition of canonical relations happens in the space
\begin{equation}\label{eq:SM2*}
\SM^{2, *} = [ \SM^* \times \SM; \{ \rhops = \rhops' = 0 \}].
\end{equation}
Then on (the lift of) the composed canonical relation, $\sigma$ can be recovered by:
    \begin{equation}
        \sigma = \frac{\xoco}{\xoct}  = \frac{\xoco}{\rho_{\ps}} \frac{\rho_{\ps}}{\xoct},
    \end{equation}
where we used the same $\rho_{\ps}$ for two $\overline{{}^{\ps}T^*\R^{n+1}}$-factors since such composition is formed by the intersection of the product of $\Lambda_+^*$ and $\Lambda_-$ with the lifted diagonal in the factor $\overline{{}^{\ps}T^*\R^{n+1}} \times \overline{{}^{\ps}T^*\R^{n+1}}$.
In addition, $\Leg$, the boundary of $\tilde{\beta}_b^*(\mathrm{GR}(\Cl)')$ (i.e. at $\xoco=\xoct=0$), is formed by the composition of the boundary of $\Lambda_+^*$ and $\Lambda_-$.
Applying Proposition~\ref{prop:1cps-sigma-equals-1}, we know $\sigma = 1$ on $\Leg$.
\end{proof}

\subsection{The analysis of the phase functions}

In this subsection, we analyze phase functions parametrizing $\Lambda_\pm$ and show that they change with constant speed under the $H_p^{2,0}$. 
Recall from \eqref{eq:fullrank} that near each $q_0 \in \Lambda_\pm$, we have a splitting $z=(z',z'')$ and $\tilde{\zeta}=(\tilde{\zeta}',\tilde{\zeta}'')$ of coordinates so that 
\begin{equation} \label{eq:coordinates-Lamda-pm}
    (t,z',\tilde{{\zeta}}'',x_{\oc},y_{\oc})
\end{equation}
becomes a coordinate system on $\Lambda_\pm$. 
We view $H_p^{2,0}$ as a vector field lifted to $\SM$ and still denote it by $H_p^{2,0}$.
Recall the expression of $\varphi_1$, which is the second part of the phase function in \eqref{eq:1c-ps-phase-normal-form}:
\begin{align}
\varphi_1 = z'' \cdot \tilde{\zeta}''-\tilde{\varphi}_1(t,z',\tilde{{\zeta}}'',x_{\oc},y_{\oc}).
\end{align}
From now on in this proof we use $\frac{\partial \bullet}{\partial \bullet}$ to denote the partial derivative without any extra dependence (i.e. not viewing $z''$ and $\tilde{\zeta}'$ and other components as functions of $z'$ and $\tilde{\zeta}''$).
Then using the same computation in the proof of Proposition~\ref{prop:1cps-sigma-equals-1} above, we compute the expression of $H_p^{2,0}$ in terms of coordinates in \eqref{eq:coordinates near ffocps} and have
\begin{align} \label{eq:Hp-varphi1}
H_p^{2,0}(\varphi_1) = 
\sum_{j' \in I'} \sum_{i=1}^n 2g^{ij}\tilde{\zeta}_i\frac{\partial (-\tilde{\varphi}_1)}{\partial z'_{j'}}
+ \sum_{j'' \in I''} \sum_{i=1}^n  2g^{ij}\tilde{\zeta}_i\tilde{\zeta}_{j''}
- \sum_{j' \in I'} \frac{\partial\tilde{p}}{\partial z'_{j'}}\frac{\partial \varphi_1}{\partial \tilde{\zeta'}_{j'}} 
- \sum_{j'' \in I''} \frac{\partial\tilde{p}}{\partial z''_{j''}}\frac{\partial \varphi_1}{\partial \tilde{\zeta''}_{j''}} + O(x_{\oc}).
\end{align}
In our case, all terms in the third sum vanishes since $\tilde{\varphi}_1$ does not involve $\tilde{\zeta}_{j'}$ for $j' \in I'$.
For the last sum, it also vanishes because we are at the critical set of the phase function when restricted to $\partial \Lambda_\pm$, hence we have
\begin{align*}
\partial_{\tilde{\zeta}_{j''}}\varphi_1 = 0, \; \forall j'' \in I''.
\end{align*}
Also from the definition of parametrization in Definition~\ref{def:1c-ps-parametrization}, for all $j' \in I'$ we have
\begin{align*}
\partial_{z_{j'}}(-\tilde{\varphi}_1) = 
\partial_{z_{j'}}\varphi_1 = \tilde{\zeta}_{j'}.
\end{align*}
Substituting quantities above into \eqref{eq:Hp-varphi1}, with the choice $\rhops = \big( \sum_{j,k} g^{jk}(z, t) \zeta_j \zeta_k \big)^{-1/2}$ and using $x_{\ps}/\rho_{\ps} =  1$ on $\Legps$ by Proposition~\ref{eq:1cps-sigma-equals-1}, we have the following property. 
\begin{prop}
Suppose $\Phi_{\oc-\ps} = -\frac{t}{x_{\oc}^2}+\frac{\varphi_1}{x_{\oc}}$ parametrizes $\Legps$, then when restricted to $C_{\Phi_{\oc-\ps}}$ (and identified with $\Legps$), we have
\begin{align} \label{eq:varphi1-const-speed}
H_p^{2,0}(\varphi_1) = 2 \sum_{i,j=1}^n g^{ij} \tilde{\zeta_i}\tilde{\zeta}_j = 2.
\end{align}
\end{prop}

Next we analyze the value of phase functions parametrizing 1c-ps Lagrangian submanifolds a geometric meaning. As a corollary, this  shows that the value of phase function at their critical points does not depend on the choice of parametrizations and depends only on the underlying Legendre submanifold and its Lagrangian extension.

First, we introduce a further resolved space
\begin{equation}
\mathfrak{M} =  [ \SM ; \{\xioc+2t=0, \, \xoc = 0 \} ],
\end{equation}
and denote the blow-down map $\mk{M} \to \SM$ by $\beta_{\mk{M}}$. Then 
\begin{equation} \label{eq:Ksi-1c-def}
    \Xi_{\oc} = \frac{\xioc+2t}{\xoc}
\end{equation}
and its reciprocal parametrizes the new front face $\ff_{\mk{M}}$. 
Since $\xi_{\oc} = -2t$ on $\Lambda_\pm \cap \{x_{\oc} = 0 \}$ (this can be seen from taking $\xoc$-derivative of $\Phi_{\oc-\ps}$ in \eqref{eq:1c-ps-phase-normal-form}, see \cite[Remark~4.13]{HJ2026-scattering-map} for full details), we know that the lift of $\partial \Lambda_\pm$ is included in $\ff_{\mk{M}}$.   
In addition, since $\Lambda_\pm$ intersect $\ffocps$ transversally in the interior, by considering the Taylor expansion of  $\xi_{\oc}$ restricted to $\Lambda_\pm$, we know that $\Xi_{\oc}$ can be viewed a smooth function of other variables on $\Lambda_\pm$.
In particular, $\beta_{\mk{M}}^*(\Lambda_\pm)$, which is the lift (i.e. the closure of the preimage of its interior) of $\Lambda_\pm$ intersects $\ff_{\mk{M}}$ in the region $|\Xi_{\oc}| \lesssim 1$, and 
\begin{equation}
    \beta_{\mk{M}}|_{\partial \beta^*(\Lambda_\pm)}
    \to \partial \Lambda_\pm
\end{equation}
is a diffeomorphism. Nevertheless, we still introduce this $\mk{M}$ to give $\Xi_{\oc}$ a clear geometric meaning of this subleading part of the $\xioc$-frequency, which helps us to understand its 1c-1c analogue that we will use later.
See also \cite[Proposition~4.4]{HJ2026-scattering-map} for connections between our Taylor expansion above and  symplectic reduction.
This $\Xi_{\oc}$ is related to our phase function by the following proposition.
\begin{prop} \label{prop:Xioc-equals-varphi1}
    Suppose $\Phi_{\oc-\ps,\pm} = -\frac{t}{x_{\oc}^2} + \frac{\varphi_\pm(\msf{K}_\pm,\theta_1)}{x_{\oc}}$ parametrizes $\Lambda_\pm$ non-degenerately in the sense of Definition~\ref{def:1c-ps-parametrization}, 
    and $q' \in \tilde{C}_{\Phi_{\oc-\ps}}$ is sent to $q \in \Lambda_\pm$ under the parametrization map, 
    then we have
\begin{equation} \label{eq:Xioc-equals-varphi1}
    \Xi_{\oc}(q) = \varphi_\pm(q'),
\end{equation}
where $\Xi_{\oc}(q)$ is the value of $\Xi_{\oc}$ at $q$ when we view it as a smooth function on $\Lambda_\pm$.
Alternatively, it can also be viewed as the $\Xi_{\oc}$-component of $\beta_{\mk{M}}^{-1}(q) \cap \beta^*(\Lambda_\pm)$.
\end{prop}

\begin{proof}
On the one hand, by discussion above, we have
\begin{equation} \label{eq:xioc-1}
    \xioc =  -2t + x_{\oc}\Xi_{\oc}
\end{equation}
on $\Lambda_\pm$. On the other hand, by the definition of parametrization, for $q' \in \tilde{C}_{\Phi_{\oc-\ps}}$, we have
\begin{equation} \label{eq:xioc-2}
   \xi_{\oc} \frac{d\xoc}{\xoc^3} = \big(-2t+\xoc \varphi_\pm(q') \big)\frac{d\xoc}{\xoc^3}.
\end{equation}
Comparing \eqref{eq:xioc-1} and \eqref{eq:xioc-2}, we obtain \eqref{eq:Xioc-equals-varphi1}.
\end{proof}

As a corollary, the value of the phase function does not depend on the choice of parametrization.
\begin{coro} \label{coro:1cps-phase-value-equal}
 Suppose both $\Phi_{\oc-\ps,\pm} = -\frac{t}{x_{\oc}^2} + \frac{\varphi_\pm(\msf{K},\theta_1)}{x_{\oc}}$ and $\tilde{\Phi}_{\oc-\ps,\pm} = -\frac{t}{x_{\oc}^2} + \frac{\tilde{\varphi}_\pm(\msf{K},\tilde{\theta}_1)}{x_{\oc}}$
 parametrizes $\Lambda_\pm$ non-degenerately, and both $q' \in \tilde{C}_{\Phi_{\oc-\ps}}$ and $\tilde{q'} \in \tilde{C}_{\tilde{\Phi}_{\oc-\ps}}$ corresponds to the same point $q \in \Lambda_\pm$ under the parametrization, then 
 \begin{equation} \label{eq:phi1-equals=tildephi1}
 \varphi_\pm(q') = \tilde{\varphi}_\pm(\tilde{q'}). 
 \end{equation}
\end{coro}

\begin{proof}
    This follows from \eqref{eq:Xioc-equals-varphi1} since both sides of \eqref{eq:phi1-equals=tildephi1} equals to $\Xi_{\oc}$ at the corresponding point $q \in \Lambda_\pm$, or $\beta_{\mk{M}}^{-1}(q) \cap \beta^*(\Lambda_\pm)$ in the lift of $\Lambda_\pm$, which is independent of the choice of parametrization.
\end{proof}

\begin{remark}
    What essentially matters in our analysis is the 1-jet of $\Lambda_\pm$ as $x_{\oc} \to 0$. If we use another Lagrangian extension of $\Legps$ with the same 1-jet, then in terms of $\xioc$, this corresponds to modifying it by a $O(x_{\oc}^2)$-term, which in turn corresponds to modifying $\Xi_{\oc}$ and the phase function $\varphi_\pm$ by a $O(x_{\oc})$ term.
    Since $\varphi_\pm$ enters in the form $e^{i\frac{\varphi_\pm}{x_{\oc}}}$, such change only introduces a smooth factor and does not change the main part of our oscillatory factor.
\end{remark}

Next we consider the 1c-1c setting. 
As shown in Corollary~\ref{coro:1c1c-sigma-equals-1}, we only need to concern the region with $\sigma$ close to $1$. 
Let coordinates be as in \eqref{eq:1c1c-1form--1}, and we define
\begin{equation} \label{eq:J-def}
    J = \{  \xioco+\xioct = 0, \, \xoco = 0 \}.
\end{equation}
Then we consider
\begin{equation} \label{eq:J-resolved-1c1c-phase-space}
    [\ococb;J]
\end{equation}
obtained by blowing up $J$ and we will show that the parameter parametrizing the front face $\ff_{SJ}$ created by blowing up encodes the sojourn time. 
Here $SJ$ stands for `sojourn' and the reason will become clear in Section~\ref{subsec:total-sj-relation}.
This allows us to recover the sojourn time geometrically from the lift of $\mathrm{Gr}(\Cli)'$. 
We denote the blow-down map by
\begin{equation} \label{eq:beta-SJ-def}
	\beta_{SJ}: \;  [\ococb;J] \to \ococb.
\end{equation}
The parameter parametrizing the front face is
\begin{equation} \label{eq:N1-def}
    N_1 =  \frac{\xioco+\xioct}{\xoco},
\end{equation}
and a coordinate system near a point with $\sigma=1$ in the interior of $\ff_{SJ}$ is 
\begin{equation} \label{eq:coordinate-near-ffSJ}
    (\xoco,\sigma,\yoco,\yoct, N_1,\xioct ,\etaoco,\etaoct).
\end{equation}

Our next goal is to investigate the relationship between $N_1$ and the value of our 1c-1c phase functions at their critical points.
In terms of coordinates in \eqref{eq:1c1c-1form-1}, we have
\begin{equation} \label{eq:N1-def-2}
    N_1 = \frac{\nu_1+\nu_2\sigma^{-1}(1-\sigma^2)}{\xoco}.
\end{equation}

Before stating the general result, let us consider the case where we take $\Phi_{\oc-\oc}$ parametrizing the lift of $\mathrm{Gr}(\Cli)'$ to be $-\Phi_{\oc-\ps,+}+\Phi_{\oc-\ps,-}$, with $\Phi_{\oc-\ps,\pm}$ parametrizing $\Lambda_\pm$.
Let $\Xi_{\oc,1}$, $\Xi_{\oc,2}$ be as in \eqref{eq:Ksi-1c-def} but defined for the left and right factors for $\Lambda_+$ and $\Lambda_-$ respectively,
then we know that restricted to $\Lambda_+^* \times \Lambda_-$ intersecting the diagonal of $\psphase \times \psphase$ in two intermediate $\ps$-variables,
we have
\begin{equation}
    \xioco = -2t + \xoco\Xi_{\oc,1}, \; \xioct = 2t - \xoct \Xi_{\oc,2},
\end{equation}
which gives
\begin{equation} \label{eq:N1-smooth-expression}
  N_1 = \Xi_{\oc,1} - \sigma^{-1} \Xi_{\oc,2}= \Xi_{\oc,1} - \sigma \Xi_{\oc,2} + O(\xoco).
\end{equation}
Using Proposition~\ref{prop:Xioc-equals-varphi1}, take $\Phi_{\oc-\ps,\pm} = -\frac{t}{x_{\oc}^2} + \frac{\varphi_\pm(\msf{K}_\pm,\theta_\pm)}{x_{\oc}}$ parametrizing $\Lambda_\pm$ non-degenerately we know that $\Xi_{\oc,1} - \sigma \Xi_{\oc,2}$ equals to $\varphi_+-\sigma\varphi_-$ at the critical point of $\Phi_{\oc-\oc}$.
On the other hand $-\varphi_++\sigma\varphi_-$ equals to the $\varphi_1$-part of $-\Phi_{\oc-\ps,+}+\Phi_{\oc-\ps,-}$ after rewriting it in the form \eqref{eq:1c-1c-phase-normal-form}

In addition, \eqref{eq:N1-smooth-expression} also shows if we further lift $\tilde{\beta}_b^*(\mathrm{Gr}(\Cl)')$ to $[\ococb; J]$, then each point in $\tilde{\beta}_b^*(\mathrm{Gr}(\Cl)')$ is lifted to a single point \footnote{The interesting part is $\tilde{\beta}_b^*(\mathrm{Gr}(\Cl)') \cap J$ since $\beta_{SJ}$ is a diffeomorphism automatically away from $J$.} and $\beta_{SJ}$ restricted to $\beta_{SJ}^*(\tilde{\beta}_b^*(\mathrm{Gr}(\Cl)'))$ is a diffeomorphism between $\beta_{SJ}^*(\tilde{\beta}_b^*(\mathrm{Gr}(\Cl)'))$ and $\tilde{\beta}_b^*(\mathrm{Gr}(\Cl)')$.
In other words, $N_1$ can effectively be viewed as a smooth function on $\tilde{\beta}_b^*(\mathrm{Gr}(\Cl)')$, while it is still convenient introduce this blow-up when we talk about more than one metrics.


Next we consider a more general version of the relation above between $N_1$ and the function value of phase functions parametrizing 1c-1c Lagrangian submanifolds.
\begin{prop} \label{prop:N1-equals-1c1c-phase}
Suppose
\begin{align}  
\Phi_{\oc-\oc}(\msf{K},t, \theta_1) =  \frac{ t(1-\sigma^2) }{\xoco^2} + \frac{\varphi_1(\msf{K},t,\theta_1)}{\xoco} , \quad \msf{K} = (\xoco, \sigma, \yoco, \yoct)
\end{align}
parametrizes $\tilde{\beta}_b^*(\mathrm{Gr}(\Cl)')$ cleanly in the sense of Definition~\ref{defn: 1c-1c parametrization clean}. For $q' \in \tilde{C}_{{\Phi}_{\oc-\oc}}$ that is sent to $q \in \tilde{\beta}_b^*(\mathrm{Gr}(\Cl)')$, 
and let $q'' = \beta_{SJ}^{-1} \cap \beta_{SJ}^*(\tilde{\beta}_b^*(\mathrm{Gr}(\Cl)'))$ be the lift of $q$, then 
\begin{equation} \label{eq:varphi1-equals--N1}
    \varphi_1(q')  = - N_1(q'') + O(\xoco).
\end{equation}
In particular, when $\xoco = 0$, we have an equality and the value of $\varphi_1$ on $C_{\Phi_{\oc-\oc}}$ is independent of the choice of the parametrization; it only depends on the corresponding point in $\tilde{\beta}_b^*(\mathrm{Gr}(\Cl)')$.
\end{prop}

\begin{proof}
    By the definition of parametrization, 
    taking the $\xoco$ and $\sigma$ derivatives of $\Phi_{\oc-\oc}$, we know that at with $q,q'$ in the proposition, we have
\begin{equation}
    \nu_1 = -2t(1-\sigma^2) - \xoco \varphi_1(q') + \xoco^2 \partial_{\xoco}\varphi_1(q'),
\end{equation}
and 
\begin{equation}
    \nu_2 = -2t\sigma + \xoco \partial_\sigma\varphi_1(q').
\end{equation}
Then using \eqref{eq:N1-def-2} and the fact that $\sigma = 1 + O(\xoco)$ when restricted to $\tilde{\beta}_b^*(\mathrm{Gr}(\Cl)')$, which is proved in Corollary~\ref{coro:1c1c-sigma-equals-1}, we obtain \eqref{eq:varphi1-equals--N1}.
\end{proof}

\section{Second microlocalization: composition with scattering pseudodifferential operators}
\label{sec:second-micro}

\subsection{The 1c,sc-1c,sc phase space and Fourier integral operators}

We introduce a further resolved phase space based on $\ococb$, a class of Lagrangian submanifolds in it, and a class of Fourier integral operators associated to them in this section.
This is used to conduct second microlocalization on our 1c-1c Fourier integral operators. Geometrically, this corresponds to distinguishing different 1-jets at $\ffococ$ of 1c-1c Lagrangian submanifolds having the same 1c-1c Legendre submanifold as their boundaries.
We will only introduce minimal amount of concepts and objects needed to second-microlocalize our scattering map in detail and give a sketch of how the general framework should be developed in Appendix~\ref{appendix:global-1csc-phase-space}. 

The phase space we will use is obtained by blowing up $\ococb$ at fixed $\xioco$ and $\xioct$ frequency:
\begin{equation} \label{eq:secm-def}
\secm  =  [\ococb; \{ \xioco = \xiocl, \, \xioct = - \xiocr, \xoco = 0\} ].
\end{equation}
We denote the blow-down map by
\begin{equation} \label{eq:beta-ocsc-def}
    \beta_{\oc,\sct-\oc,\sct}: \; \secm \to \ococb,
\end{equation}
and denote the front face created by blowing up by $\ffsec$.
Then $\ffsec$ can be parametrized by
\begin{equation} \label{eq:xisc-12-def}
    \xi_{\sct,1} = \frac{\xioco-\xi_{\oc,0}}{\xoco}, \;
    \xi_{\sct,2} = \frac{\xioct+\xi_{\oc,0}}{\xoct},
\end{equation}
and their reciprocals or ratio. Only the region of $\ffsec$ with $\xi_{\sct,1},\xi_{\sct,2}$ in a bounded region, hence indeed parametrized by them, will be of interest for us and a coordinate system there is given by
\begin{equation} \label{eq:coordinates-near-ffsec}
(\xoco,\sigma,\yoco,\yoct,\xi_{\sct,1},\xi_{\sct,2},\etaoco,\etaoct).
\end{equation}

Then $\secm$ inherits a canonical 1-form 
\begin{equation}
    \alpha_{\oc,\sct-\oc,\sct} = \beta_{\oc,\sct-\oc,\sct}^* \alpha_{\oc-\oc},
\end{equation}
and the corresponding symplectic form 
\begin{equation} \label{eq:omega 1csc-1csc}
    \omega_{\oc,\sct-\oc,\sct} = d\alpha_{\oc,\sct-\oc,\sct}.
\end{equation}
Near a point in the interior of $\ffsec$, with other coordinates not involved in this further blow-up as in \eqref{eq:1c1c-1form-0}, the canonical 1-form is
\begin{align}
\label{eq:1form-sec}
\begin{split}
 \alpha_{\oc,\sct-\oc,\sct} = &
( \xi_{\oc,0}(1-\sigma^2) + (\xoco \xi_{\sct,1} + \xoct\sigma^2 \xi_{\sct,2}) \frac{d\xoco}{\xoco^3} - (-\sigma\xiocr + \xoco \xi_{\sct,2} ) \frac{d\sigma}{\xoco^2} \\& + \etaoco \frac{d\yoco}{\xoco} + \sigma \etaoct \frac{d\yoct}{\xoco}
\\ = &(\xi_{\oc,0}(1-\sigma^2))\frac{d\xoco}{\xoco^3} 
+ \sigma \xiocr \frac{d\sigma}{\xoco^2}
+ (\xi_{\sct,1} + \sigma \xi_{\sct,2}) \frac{d\xoco}{\xoco^2} -  \xi_{\sct,2}  \frac{d\sigma}{\xoco} \\& + \etaoco \frac{d\yoco}{\xoco} + \sigma \etaoct \frac{d\yoct}{\xoco} .
\end{split}
\end{align}

In $\secm$, the class of Lagrangian and Legendre submanifolds second-microlocalized at $\xi_{\oc,0}$ are defined as follows.

\begin{definition}[Second-microlocalized 1c,sc-1c,sc Lagrangian submanifolds]  \label{definition:second-microlocalized-Lag-manifold}
We define an admissible $\oc,\sct-\oc,\sct$ Lagrangian submanifold (second-microlocalized at $\xi_{\oc,0}$) of $\secm$ to be a $2n$-dimensional submanifold  $\Lagsec$ that is Lagrangian in the interior (that is, the symplectic form $\omega_{\oc,\sct-\oc,\sct}$ from \eqref{eq:omega 1csc-1csc} vanishes on it), such that 
\begin{itemize}
\item $\Lagsec$ meets $\ff_{(\oc,\sct,\xiocl)}$ transversally, 
\item the differential $d \xi_{\sct,1}$ is non-vanishing on $\Lagsec \cap \ffsec$, and 
\item  its closure is disjoint from all other boundary hypersurfaces of $\secm$ other than $\ffsec$. 
\end{itemize}
We define a $(\oc,\sct)-(\oc,\sct)$ Legendre submanifold $\Legsec$ of $\ffococ$ to be the boundary of an admissible 1c-1c Lagrangian submanifold. In other words, there exists $\Lagsec$ as above such that $\Legsec = \Lagsec \cap \ffsec$. 
\end{definition}
We emphasize that the last condition in the definition above ensures that the entire $\Legsec$ `lives within' a single 1-cusp frequency in the radial direction: $\xioco = \xiocl$, $\xioct=-\xiocr$.

We didn't include `fibered' in the name, but indeed this class of Lagrangian submanifold also have a fibration structure similar to \cite[Proposition~6.3]{HJ2026-scattering-map}, except for that now with $\xi_{\sct,1},\xi_{\sct,2}$ obtained by second-microlocalization playing the role of the original $\xioco,\xioct$.


The current $\xi_{\sct,1},\xi_{\sct,2}$ in \eqref{eq:xisc-12-def} is directly related to $N_1$ defined in \eqref{eq:N1-def}.
\begin{lmm} \label{lemma:N1-equals-xisc1-plus-xisc2}
Let $\Lag$ be an admissible $\oc-\oc$ Lagrangian submanifold in the sense of Definition~\ref{definition:admissible-1c1cLag}, and suppose that both $\beta_{SJ}^*(\Lag)$ (with $\beta_{SJ}$ in \eqref{eq:beta-SJ-def})
and $\beta_{\oc,\sct-\oc,\sct}^*(\Lag)$ (for a fixed $\xi_{\oc,0}$ in \eqref{eq:secm-def}) are diffeomorphic to $\Lag$ via $\beta_{SJ}$ and $\beta_{\oc,\sct-\oc,\sct}$,
then for $q \in \Lag$ with $\xioco = \xi_{\oc,0}, \xioct = -\xi_{\oc,0}$, we have
\begin{equation}
        N_1(\beta_{SJ}^{-1}(q)) = \xi_{\sct,1}(\beta_{\oc,\sct-\oc,\sct}^{-1}(q)) + \xi_{\sct,2}(\beta_{\oc,\sct-\oc,\sct}^{-1}(q)).
\end{equation}
\end{lmm}

\begin{proof}
This follows from 
\begin{equation}
    N_1 = \frac{\xioco+\xioct}{\xoco}
    = \frac{\xioco-\xi_{\oc,0}}{\xoco} + \frac{\xioct+\xi_{\oc,0}}{\xoco} = \xi_{\sct,1} + \xi_{\sct,2},
\end{equation}
which extends down to the boundary after we introduce blow-ups on both sides respectively.
Or equivalently, one can view both sides as functions on $\Lag$ under the diffeomorphic assumption in the lemma.
\end{proof}

Next we discuss parametrizations of this class of Lagrangian submanifolds.
\begin{definition} \label{def:parametrization-sec}
Suppose $\Lagsec$ is an admissible $\oc,\sct-\oc,\sct$ Lagrangian submanifold second-microlocalized at $\xi_{\oc,0}$, and $(\ococparaone, \ococparatwo)$ is in a non-empty open subset $U$ of $\RR^{k_0+k_1}$. We say that 
\begin{align}\label{eq:Phi 1c-1c-sec param clean}
\Phi_{\oc,\sct-\oc,\sct}(\msf{K}, \ococparaone, \ococparatwo) =  \frac{ \varphi_0(\sigma,\ococparaone) }{x_{\oc,1}^2} + \frac{\varphi_1(\msf{K},\ococparaone,\ococparatwo)}{x_{\oc,1}} , \quad \msf{K} = (\xoco, \sigma, \yoco, \yoct), 
\end{align} 
gives a \emph{clean parametrization} of $\SL$ near $q$ with excess $e$ if there is a point $q' = (\msf{K}', \ococparaone', \ococparatwo')$ such that the differentials
\begin{equation} \label{eq:varphi0 1csc clean}
d_{\sigma,\ococparaone} \frac{\partial \varphi_0}{\partial v_j}, \quad 1 \leq j \leq k_0,
\end{equation}
are linearly independent at $(\sigma', \ococparaone')$, and the differentials 
\begin{equation}\label{eq:varphi1 1csc clean}
d_{\yoco, \yoct, \ococparatwo} \frac{\partial \varphi_1}{\partial w_j}, 1 \leq j \leq k_1,
\end{equation}
have a fixed rank $k_1 - e$ near $q'$, such that $\SL$ is given locally by  
\begin{align}\label{eq:1c-1c param-Lagsec}
\Lagsec
=\{ (\msf{K},d_{\msf{K}}\Phi_{\oc-\ps}) \mid \;  (\msf{K},v,w) \in \tilde{C}_{\Phi_{\oc-\oc}}\},
\end{align}
where 
\begin{align}  \label{eq: 1c-1c critical set-Lagsec}
\tilde{C}_{\Phi_{\oc-\oc}} =  \{ (\msf{K},v,w) \mid d_{v}\varphi_0 + \xoco d_{v}\varphi_1 = 0, d_{w}\varphi_1 = 0 \}.
\end{align}
When $e=0$, we say this parametrization is \emph{non-degenerate}.

As before, either $\ococparaone$ or $\ococparatwo$ may be absent, in which case conditions for the derivatives in \eqref{eq:varphi0 1csc clean}, resp. \eqref{eq:varphi1 1csc clean} are dispensed with, as well as stationarity with respect to $\ococparaone$, resp. $\ococparatwo$ in \eqref{eq: 1c-1c critical set clean}. 
\end{definition}
A typical phase function for such a parametrization takes the form
\begin{equation} \label{eq:Phisec-typical}
    \Phi = -\frac{\xi_{\oc,0}(1-\sigma^2)}{2\xoco^2}+\frac{\varphi_1(\msf{K},w)}{\xoco},
\end{equation}
which means the $v$-parameters are absent. But we are still including those $v$-parameters to handle phase functions arise in compositions. In fact, the second part $\frac{\varphi_1(\msf{K},w)}{\xoco}$ parametrizes a sc-Lagrangian submanifold. See the identification in Proposition~\ref{prop:1c-sc-phase-space-relation} below.

Now we define the class of second-microlocalized Fourier integral operators in this $\oc,\sct-\oc,\sct$ setting.
\begin{definition} \label{def:1c-1c-FIOsec}
Let $\Lagsec$ be an admissible Lagrangian submanifold of $\secm$ as in Definition~\ref{definition:second-microlocalized-Lag-manifold}. 
We define $I^{m}_{\oc,\sct-\oc,\sct}(X_b^2, \Lagsec; \Omega_{\oc-\oc}^{1/2})$, i.e., the space of $\oc,\sct-\oc,\sct$ Fourier integral operators of order $m$, to be the space of operators with Schwartz kernel given (modulo a Schwartz function) by a finite sum of terms of the form 
\begin{align} \label{eq: 1c-1c FIOsec, local form}
\begin{split}
 & (2\pi)^{-\frac{n+(k_0+k_1-e)}{2}} \Big(\int 
e^{i\Phi_{\oc,\sct-\oc,\sct}(\msf{K},v,w)} 
 a(\msf{K},v,w)
\\& x_{\oc,1}^{-m-\frac{2k_0+(k_1-e)}{2}+ \frac{n+1}{2}} dvdw \Big) 
|\frac{d\sigma dy_{\oc,2}}{x_{\oc,1}^{n+1}}|^{1/2}|\frac{dx_{\oc,1}dy_{\oc,1}}{x_{\oc,1}^{n+2}}|^{1/2},
\end{split}
\end{align}
where $\Phi_{\oc,\sct-\oc,\sct}$ is a phase function of the form \eqref{eq:Phi 1c-1c-sec param clean} locally parametrizing $\Lagsec$ non-degenerately in the sense of Definition \ref{def:parametrization-sec}.
Moreover,  $a \in C^\infty_c([0,\infty)_{x_{\oc}} \times [C^{-1},C]_{\sigma} \times \mathbb{S}^{n-1} \times \mathbb{S}^{n-1} \times \R^{k_0} \times \R^{k_1})$, where $v \in \R^{k_0},w \in \R^{k_1}$. 
\end{definition}

The proof of the phase equivalence (i.e. rewriting the oscillatory integral in \eqref{eq: 1c-1c FIOsec, local form} using another phase function parametrizing the same Lagrangian submanifold) proved in \cite[Proposition~7.3]{HJ2026-scattering-map} (with details in Appendix~A there) still applies in this setting, since the only difference compared with the standard phase equivalence (i.e., for H\"ormander's Fourier integral operator class or the class of Legendre distributions by Melrose-Zworski \cite{melrose1996scattering}) remain the same: the parabolic scaling between two parts of the phase function.
So the operator class defined using a clean parametrization can always be rewritten as an operator with non-degenerate parametrization.
In summary, we have:
\begin{prop} \label{prop: phase-equivlaence-sec-microlocal}
Let $\Phi_{\oc,\sct-\oc,\sct},\tilde{\Phi}_{\oc,\sct-\oc,\sct}$ be two clean parametrizations of $\Lagsec$ near $q \in \partial \Lagsec$ with excesses $e,\tilde{e}$ respectively, then any local expression
\begin{align*}
 \int e^{i \Phi_{\oc,\sct-\oc,\sct} } 
x_{\oc,1}^{-m - \frac{2k_0+(k_1-e)}{2} +\frac{n+1}{2} }
a(x_{\oc,1},\sigma,y_{\oc},y_{\oc}'',v,w)dvdw,
\end{align*}
with $a$ smooth, supported in the region where the parametrizations of $\Phi_{\oc,\sct-\oc,\sct},\tilde{\Phi}_{\oc,\sct-\oc,\sct}$ are both valid, can also be written as
\begin{align*}
A = u_0 + \int e^{i \tilde{\Phi}_{\oc,\sct-\oc,\sct} } 
x_{\oc,1}^{-m - \frac{2\tilde{k}_0+(\tilde{k}_1-\tilde{e})}{2} +\frac{n+1}{2} }
\tilde{a}(x_{\oc,1},\sigma,y_{\oc},y_{\oc}'',\tilde{v},\tilde{w})d\tilde{v}d\tilde{w},
\end{align*}
with $u_0$ being Schwartz, $\tilde{a}$ smooth.
\end{prop}

Then the principal symbol map $\sigma^m_{\oc,\sct-\oc,\sct}$ is defined to be a half density bundle on $\Legsec$ in the same way as in \eqref{defn: 1c-1c principal symbol, non-degenerate} in case of non-degenerate parametrization and as in \eqref{defn: 1c-1c principal symbol, clean} in the more general case of clean parametrizations.
The bundle $S_{\oc,\sct-\oc,\sct}^{[m]}(\Legsec)$ is defined as in \eqref{eq:S[m]-bundle} except for with $\Leg$ replaced by $\Legsec$.
This is guaranteed to be invariant under change of parametrizations by the phase equivalence above.

Using this principal symbol map, in the same way as \eqref{eq:1c-1c-FIO-short-exact-sequence}, we have the following exact sequence:
\begin{align} \label{eq:1c-1c-FIOsec-short-exact-sequence}
\begin{split}
0  \rightarrow I^{m-1}_{\oc,\sct-\oc,\sct}(X_b^2,\Lagsec)
\rightarrow I^{m}_{\oc,\sct-\oc,\sct}(X_b^2,\Lagsec)
\xrightarrow{\sigma^m_{\oc-\oc}} C^\infty(\Legsec; 
\Omega^{1/2}(\Legsec) \otimes S_{\oc,\sct-\oc,\sct}^{[m]}(\Legsec)) \rightarrow 0.
\end{split}
\end{align}
The exactness in the middle, i.e. when a Fourier integral operator has vanishing principal symbol then it is one order lower, is proved in the same way as in \cite[Lemma~5.6]{HJ2026-scattering-map}, i.e. expand the amplitude as derivatives of the phase function and then integrate by parts.

Now we explore the relationship between the objects we are using to refine the analysis at a fixed $\oc$-frequency with objects used to analysis objects on the $\sct$-frequency level. More precisely, we will show that the part of $\ffsec \subset \secm$ we are concerning can be identified with the lift of $\scphase \times \scphase$ to $X_b^2$, with $\scphase$ defined in Section~\ref{sec: parabolic sc PsiDO}.

We consider the relationship between a single cotangent bundle first.
We will consider the resolved phase space 
\begin{equation} \label{eq:1csc-cotangent-bundle-def}
{}^{\oc,\sct}T^*\R^n:= [{}^{\oc}T^*\R^n ; \{\xioc = 0, \, \xoc = 0 \}].
\end{equation}
We call denote the front face by $\ff_{\sct}$, and the lift of $\{\xioc \neq 0, \xoc = 0\}$ by $\mathrm{mf}$.
Then we have the following identification of the interior of $\ff_{\sct}$ and ${}^{\sct}T^*\R^n$.
See also \cite[Lemma~5.1]{Vasy-second-microlocal-1} for a similar relationship between the sc-cotangent bundle and the b-cotangent bundle.

\begin{prop} \label{prop:1c-sc-phase-space-relation}
For fixed $\epsilon \ll 1$ and $C>0$, the region $|\etaoc|, |\xioc/\xoc| < C, \xoc < \epsilon$
is canonically diffeomorphic to the region $|\eta_{\sct}|<C,|\xi_{\sct}|<C,\xoc<\epsilon$ in ${}^{\sct}T^*\overline{\R^n}$ via the extension of the identity map in the interior $\{\xoc>0\}$. 
\end{prop}
\begin{proof}
Comparing \eqref{eq: 1c- canonical form} and \eqref{eq:sc-1-form}, we know
\begin{equation}
    \frac{\xioc}{\xoc} = \xi_{\sct}
\end{equation}
in the interior. The blow up \eqref{eq:1csc-cotangent-bundle-def} is exactly the operation that extends the left hand side to the boundary, or more precisely into the interior of $\ff_{\sct}$, which shows that the identity map in the interior extends to a diffeomorphism as in the proposition.
\end{proof}

Returning to our double space picture. 
First observe that $\secm$ for general $\xi_{\oc,0}$ can be identified with ${}^{\oc,\sct,0}T^*X_b^2$ by translation in $\xioco$ and $\xioct$.
\footnote{This corresponds to the time translation in the Schr\"odinger equation. Since $\xioco=-2t$ in the composition of canonical relations.}
In addition, in the region $\sigma \sim 1$, the operation of blow-up in \eqref{eq:secm-def} commutes with lifting from two factors.
That is, if we lift resolved phase spaces $[\ocphase;\{\xioco = \xiocl , \xoco = 0 \}] \times [\ocphase;\{ \xioct = -\xiocr , \xoct = 0\}]$ to $X_b^2$, then it is the same as $\secm$.
In combination with Proposition~\ref{prop:1c-sc-phase-space-relation}, we know that the part of $\secm$ near the interior of $\ffsec$ with $|\xi_{\sct,1}|, |\etaoco|,|\xi_{\sct,2}|, |\etaoct|<C$ can be identified with the part of 
\begin{equation} \label{eq:lifted-scsc-phase}
    \beta_b^*({}^{\sct}T^* \overline{\R^n}  \times {}^{\sct}T^* \overline{\R^n})
\end{equation}
with $|\xi_{\sct}|,|\eta_{\sct}|$ lifted from the left and right factor being less than $C$.

Let $\secmz$ be $\secm$ constructed with $\xiocl = 0$ and let $\Lagsecz$ be the translated copy of $\Lagsec$ in $\Lagsecz$ obtained by $\xioco \to \xioco - \xiocl$, $\xioct \to \xioct + \xiocr$.
Then by the identification given in Proposition~\ref{prop:1c-sc-phase-space-relation},
$\Lagsecz$ can be identified with a Lagrangian submanifold of $\beta_b^*({}^{\sct}T^* \overline{\R^n}  \times {}^{\sct}T^* \overline{\R^n})$, which has been studied in \cite{melrose1996scattering}\cite{hassell1999spectral}\cite{hassell2001resolvent}.
In particular, as shown in \cite[Proposition~6]{melrose1996scattering}, there always exist a phase function $\frac{\varphi_1(\msf{K},w)}{\xoco}$ that parametrizes $\Lagsecz$ non-degenerately.
With such a parametrization of $\Lagsecz$, then by definition we know
\begin{equation} \label{eq:phase-second-micro-2}
    \Phi = -\frac{\xi_{\oc,0}(1-\sigma^2)}{2\xoco^2}+\frac{\varphi_1(\msf{K},w)}{\xoco},
\end{equation}
parametrizes $\Lagsec$. By the phase equivalence in Proposition~\ref{prop: phase-equivlaence-sec-microlocal}, we can always assume that an operator in $I_{\oc,\sct-\oc,\sct}^m(X_b^2,\Lagsec)$ is written using such a phase function, which in turn can be rewritten as 
\begin{equation} \label{eq:A=conjugatedA0}
    e^{-i\frac{\xiocl}{2\xoco^2}} A_0 e^{i \frac{\xiocl}{2\xoco^2} },
\end{equation}
with $A_0$ being a Legendre distribution in \cite[Section~11]{melrose1996scattering}.
This is correspondence between the conjugation on the leading order oscillation and the translation of the front face in the phase space (i.e. where the second-microlocalization is happening) is very similar to \cite{NRL-I}\cite{NRL-II}\cite{Vasy-second-microlocal-3}.


\subsection{Compositions of Fourier integral operators of various types}

We consider the composition of $\oc,\sct-\oc,\sct$ Fourier integral operators with $\oc-\oc$ Fourier integral operators first.
This will arise when we (second-)microlocalize our scattering map on one side.
In order to treat this, we discuss composition of such canonical relations in this context first.
We use $'$ to denote switching the sign of frequencies lifted from the right component.
Let $C_L' \subset \secm$ be an admissible $\oc,\sct-\oc,\sct$ Lagrangian submanifold in the sense of Definition~\ref{definition:second-microlocalized-Lag-manifold} and $C'$ be an admissible Lagrangian submanifold of $\ococb$ in the sense of Definition~\ref{definition:admissible-1c1cLag}, 
the composition of the interior of corresponding canonical relations $C_L,C$ is defined in the usual way, i.e. 
\begin{align} \label{eq:Cl-C-composition-interior}
    \begin{split}
 C_L^\circ \circ C^\circ
 = & C_L^\circ \times C^\circ \cap \{ \text{The partial diagonal of } 
 {}^\oc T^*\R^n
 \times  {}^\oc T^*\R^n \times  {}^\oc T^*\R^n \times  {}^\oc T^*\R^n 
  \\& \text{ in the second and third component }  \}.    
 \end{split}
\end{align}
Then the entire composition of canonical relations is defined by
\begin{equation} \label{eq:CL-C-composition-sec}
    C_L \circ C : = (\beta_{\oc,\sct-\oc,\sct}
    \circ \tilde{\beta}_b)^*( C_L^\circ \circ C^\circ) \subset \secm,
\end{equation}
where $\tilde{\beta}_b$ is the blow-down map $\ococb \to \ocphase \times \ocphase$.
As usual, here the lift operation means taking closure in $\secm$ of the composition of canonical relations in the interior.
In general, the corresponding Lagrangian submanifold $(C_L \circ C)'$ will not be admissible in the sense of Definition~\ref{definition:second-microlocalized-Lag-manifold} anymore. However, if $C'$ satisfies:
\begin{equation} \label{eq:C'-condition}
    \text{the lift of } C' \text{ in } [\ococb,J] \text{ defined in \eqref{eq:J-resolved-1c1c-phase-space} 
    only hits} \ff_{SJ} \text{ in the interior,}   
\end{equation}
then indeed $(C_L \circ C)'$ is admissible in the sense of Definition~\ref{definition:second-microlocalized-Lag-manifold}.
This condition means that $\xioco$ equals to $-\xioct$ on the boundary of $C'$, and they are comparable on the subleading order.
This is exactly the condition that when left variables are restriction to a region on which $\xi_{\sct,1} = \frac{\xi_{\oc,1}-\xi_{\oc,0}}{\xoco}$ is bounded, then so does $\xi_{\sct,2} = \frac{\xioct+\xiocr}{\xoco}$,
where we have used the fact that $\sigma = \xoco/\xoct$ is comparable to 1 on $C$ since $C$ is admissible in the sense of Definition~\ref{definition:admissible-1c1cLag}.

We also need the transversal intersection condition of this composition. More precisely, we assume that $C_L \times C$ intersects the lift of the partial diagonal of ${}^\oc T^*\R^n \times  {}^\oc T^*\R^n \times  {}^\oc T^*\R^n \times  {}^\oc T^*\R^n$ in the second and third component transversally.
Then in the same way as \eqref{eq: symbol product, bundle maps}, there is a natural bilinear map
\begin{align}  \label{eq: symbol product, bundle maps-sec}
\begin{split}
& \Omega^{1/2}(\partial (C_L')) \otimes S^{[m_2]}(\partial (C_L'))
 \times \Omega^{1/2}(\partial C') \otimes S^{[m_1]}(\partial C') 
 \\ & \rightarrow  \Omega^{1/2}((C_L \circ C)' \cap \ffsec) \otimes S^{[m_2+m_1]}( (C_L \circ C) ' \cap \ffsec),
\end{split}
\end{align}
which we again denote by $\times$.

Similarly, we will also consider composition with an second-microlocalized object on the right.
Let $C_R' \subset \secm$ be an admissible $\oc,\sct-\oc,\sct$ Lagrangian submanifold in the sense of Definition~\ref{definition:second-microlocalized-Lag-manifold}, then we can define
\begin{equation} \label{eq:C-CR-composition-sec}
     C \circ C_R : = (\beta_{\oc,\sct-\oc,\sct}
    \circ \tilde{\beta}_b)^*(C^\circ \circ C_R^\circ) \subset \secm,
\end{equation}
and we have the bilinear map on half-densities as in \eqref{eq: symbol product, bundle maps-sec} with $C_L'$ replaced by $C_R$ and moved to the right.

Now we discuss phase functions parametrizing such composition.
\begin{lmm} \label{lemma:phase-composition-sec}
Let $C_L',C'$ be as above, in particular suppose $C'$ satisfies \eqref{eq:C'-condition}, we denote variables used for $C_L'$ and $C'$ by $\sigma_L = x_{\oc,L}/\xoco$, $\msf{K}_L = (x_{\oc,L},\sigma_L,y_{\oc,L},\yoco)$, $\sigma = \xoco/\xoct$, $\msf{K} = (\xoco,\sigma,\yoco,\yoct)$.
Suppose $\Phi_L = \frac{\varphi_0(\sigma_L,v_L)}{\xoco^2}+\frac{\varphi_1(\msf{K}_L,v_L,w_L)}{\xoco}$ parametrizes $C_L'$ cleanly in the sense of Definition~\ref{def:parametrization-sec}, 
$\Phi_{\oc-\oc} = \frac{\varphi_0(\sigma,v)}{\xoco^2}+\frac{\varphi_1(\msf{K},v,w)}{\xoco}$ parametrizes $C'$ cleanly in the sense of Definition~\ref{eq:1c-1c param clean}, then
\begin{equation} \label{eq:phase-CL-C}
   \Phi_{\oc,\sct-\oc,\sct} = \Phi_L + \Phi_{\oc-\oc},
\end{equation}
where we take $\tilde{\sigma} = x_{\oc,L}/\xoct = \sigma_L \sigma$ as the coordinate parametrizing the b-front face, 
$(\sigma_L,\xoco,\yoco)$ as parameters and substitute in $\sigma  = \tilde{\sigma} \sigma_L^{-1}$ for all dependence on $\sigma$, parametrizes the lift of $(C_L \circ C)'$ in the sense of Definition~\ref{def:parametrization-sec} cleanly.

Similarly, if $\Phi_R$ parametrizes $C_R'$ cleanly, then $\Phi_{\oc-\oc} + \Phi_R$ parametrizes $(C \circ C_R)'$ cleanly.
\end{lmm}

As we explained, after imposing the condition \eqref{eq:CL-C-composition-sec} on $C'$, the composition is admissible in the sense of Definition~\ref{definition:second-microlocalized-Lag-manifold}.
The proof follows from the same proof in the standard setting in the interior (i.e. applying the proof of \cite[Theorem~4.2.2]{FIO1} to the $\xoco^{-2})$ level part first and then to the $\xoco^{-1}$-level part), and by the way we defined $C_L \circ C$, the closure of this part in $\secm$ is this composition.
On the other hand, since it coincide with the graph of $\Phi_{\oc,\sct-\oc,\sct}$ restricted to the critical set in the interior and both are smooth closed submanifolds of $\secm$, this fact extends to the boundary. 
Then we have the following composition law.

\begin{prop} \label{prop:composition-ALA-AAR-FIOsec}
    For $A_L \in I^{m_L}_{\oc,\sct-\oc,\sct}(X_b^2,C_L')$, $A \in I^m_{\oc-\oc}(X_b^2,C')$, we have
\begin{equation}
    A_LA \in I^{m_L+m}_{\oc,\sct-\oc,\sct}(X_b^2, (C_L \circ C)').
\end{equation}
When they have $a_L,a$ as their principal symbol respectively, $A_LA$ has principal symbol
\begin{align*}
a_L \times a.
\end{align*}
Similarly, if $A_R \in I^{m_R}_{\oc,\sct-\oc,\sct}(X_b^2,C_R')$, then  
\begin{equation}
    AA_R \in I^{m+m_R}_{\oc,\sct-\oc,\sct}(X_b^2, (C \circ C_R)').
\end{equation}
\end{prop}
The kernel of $A_LA$ is an oscillatory integral with phase function $\Phi_L+\Phi_{\oc-\oc}$ with $\Phi_L$, $\Phi_{\oc-\oc}$ parametrizing $C_L'$ and $C'$ respectively.
Then using Lemma~\ref{lemma:phase-composition-sec}, the proof follows in the same way as \cite[Proposition~7.5]{HJ2026-scattering-map}, except for that the type of Lagrangian submanifolds that it parametrizes is different.


Now we discuss composition between operators in $I_{\oc,\sct-\oc,\sct}^*$ classes.
We start with considering compositions of canonical relations.
Let $C_L',C_R'$ be admissible $\oc,\sct-\oc,\sct$ Lagrangian submanifold second-microlocalized at $\xi_{\oc,0}$ in the sense of Definition~\ref{definition:second-microlocalized-Lag-manifold}.
Then similar to \eqref{eq:C-CR-composition-sec}, their composition is still defined to be
\begin{equation} \label{eq:CL-CR-composition-sec}
     C_L \circ C_R : = (\beta_{\oc,\sct-\oc,\sct}
    \circ \tilde{\beta}_b)^*(C_L^\circ \circ C_R^\circ) \subset \secm.
\end{equation}
Suppose $\Phi_L$ parametrizes $C_L'$ cleanly and $\Phi_R$ parametrizes $C_R'$ cleanly in the sense of Definition~\ref{def:parametrization-sec}, then the analogue of Lemma~\ref{lemma:phase-composition-sec} holds: $\Phi_L+\Phi_R$ parametrizes $(C_L \circ C_R)'$.
This can still be shown by applying the proof of \cite[Theorem~4.2.2]{FIO1} to two parts of the phase function respectively.
But in this case we can further simplify. We can first apply Proposition~\ref{prop: phase-equivlaence-sec-microlocal} to reduce to phase functions taking the form \eqref{eq:phase-second-micro-2}, then the $\xoco^{-2}$-level part in $\Phi_L+\Phi_R$ takes the desired form directly and we only need to apply the proof of \cite[Theorem~4.2.2]{FIO1} to the $O(\xoco^{-1})$-level part.
Also, we have a bilinear map sending half-densities on $\partial(C_L')$ and $\partial(C_R')$ to a half-density on $\partial(C_L \circ C_R)'$, which we still denote by $\times$. 
In summary, we have the following composition law.

\begin{prop} \label{prop:composition-AL-AR-FIOsec}
    For $A_L \in I^{m_L}_{\oc,\sct-\oc,\sct}(X_b^2,C_L')$, $A_R \in I^{m_R}_{\oc,\sct-\oc,\sct}(X_b^2,C_R')$, we have
\begin{equation}
    A_LA_R \in I^{m_L+m_R}_{\oc,\sct-\oc,\sct}(X_b^2, (C_L \circ C_R)').
\end{equation}
When they have $a_L,a_R$ as their principal symbol respectively, the principal symbol of $A_LA_R$ is
\begin{align*}
a_L \times a_R.
\end{align*}
In addition, when $C_L,C_R$ are both graphs of symplectomorphisms, if $A_L,A_R$ are elliptic at $q_L',q_R'$ (in the sense that the principal symbol is non-vanishing) such that $(q_L,q_R)$ is sent to $q \in C_2 \circ C_1$, then $A_2A_1$ is elliptic at $q'$.
\end{prop}

The following two special cases are of our interest. 
\begin{coro} \label{coro:ALAR-trivial}
    For $A_L \in I^{m_L}_{\oc,\sct-\oc,\sct}(X_b^2,C_L')$, $A_R \in I^{m_R}_{\oc,\sct-\oc,\sct}(X_b^2,C_R')$ with $C_L \circ C_R = \emptyset$, then
    \begin{equation}
        A_LA_R \in \mathcal{S}(\R^n \times \R^n).
    \end{equation}
\end{coro}
\begin{proof}
    In this case, the phase $\Phi_L+\Phi_R$ does not admit any critical point, and one can integrate by parts (i.e. apply a non-stationary phase argument) to show that the kernel of $A_LA_R$ has infinite order decay in $\xoco$ (hence in $\xoct$ as well since we are in the region $\sigma \sim 1$).
    Since differentiation only causes at most $\xoco^{-3}$ order loss, which can be absorbed by the infinite order decay, hence the Schwartz kernel of $A_LA_R$ is in $\mathcal{S}(\R^n \times \R^n)$ as claimed.
\end{proof}

\begin{coro} \label{coro:sec-micro-noncompact}
    Suppose $A \in I^0_{\oc,\sct-\oc,\sct}(X_b^2,\Lagsec)$ with $(\Lagsec)'$ being the lift of a graph of a symplectomorphism, 
    if it is elliptic at certain point $q \in \Legsec = \partial \Lagsec$, then it is not compact from $L^2(\R^n)$ to itself.
\end{coro}

\begin{proof}
Writing $A$ in the form \eqref{eq:A=conjugatedA0}, we only need to show that $A_0 \in I^0_{\oc,\sct-\oc,\sct}(X_b^2,\Lagsecz)$ is not compact from $L^2(\R^n)$ to itself.

As pointed out after \eqref{eq:lifted-scsc-phase}, $\Lagsecz$ can be identified with Lagrangian submanifolds of $\beta_b^*({}^{\sct}T^* \overline{\R^n}  \times {}^{\sct}T^* \overline{\R^n})$.

Now we apply Proposition~\ref{prop:composition-AL-AR-FIOsec} to $A_0^*A$ and obtain that it is in 
\begin{equation}
    I_{\oc,\sct-\oc,\sct}^0(X_b^2, \big((\Lagsecz)^{*'} \circ (\Lagsecz)'\big)' )
\end{equation}
where $*$ means switching left and right variables and the sign of frequencies.
One can verify that this is the conormal bundle of the lifted diagonal in $X_b^2$, which is parametrized by
\begin{equation} \label{eq:sc-phase}
    \Phi = \xi_{\sct} \frac{(1-\sigma)}{\xoco} + \eta_{\sct} \cdot \frac{(y-y')}{\xoco},
\end{equation}
    and is in the class of sc-pseudodifferential operators defined in \cite{Melrose1994}.
In addition, by the assumption and the formula for the principal symbol in Proposition~\ref{prop:composition-AL-AR-FIOsec}, $A_0^*A_0$ is a scattering pseudodifferential operator that is elliptic at some point $q_1 \in {}^{\sct}T^*_{\partial \overline{\R^n}} \overline{\R^n}$, hence it is not compact from $L^2(\R^n)$ to itself.
\end{proof}
\section{Determining the metric} \label{sec:determine-metric}

We prove that the leading order behavior of the scattering matrix determines the metric in this section.
First we will show that the Lagrangian submanifold associated to the scattering map determines the scattering relation, sending initial position and velocity to the escaping position and velocity, truncated to a ball.
Then we will show that if two metrics have the same sojourn time, then their geodesics have the same length as well.
 Finally we will show that the 1c-1c Lagrangian submanifold associated to the scattering map, in particular its $N_1$-component defined in \eqref{eq:N1-def}, determines the sojourn time.
All these ingredients together allow us to determine the metric from the scattering map.

\subsection{Determining the truncated scattering relation}
In this subsection, we show that the information in our geometric scattering relation $\Cli, \, i=1,2$ (defined using \eqref{eq:classical-sc-map-def} with $g$ replaced by $g_i$)
determines its `truncated' analogue: the geometric scattering relation restricted to $B_R(0)$. 
In order to `read' length of geodesic truncated to a finite region, we define $\mk{B}_{g_i}$ to be the time-parametrized cosphere bundle as in \eqref{eq:parametrized-copehre-bundle-def} with $g$ replaced by $g_i$:
\begin{equation} \label{eq:parametrized-cosphere-def}
\mk{B}_{g_i} = [-T,T] \times S^*_{g_i}B_R(0).
\end{equation}
Since $g_1,g_2$ coincide over $[-T,T] \times \partial B_R(0)$, we know that 
\begin{equation}
\mk{B}_{g_1}|_{\partial B_R(0)} = \mk{B}_{g_2}|_{\partial B_R(0)},
\end{equation}
and we will omit the subscript $i$ for objects restricted to this boundary.
Note that $S^* \partial B_R(0)$ (since $g_1$ and $g_2$ are the same here, we ignore the index $1,2$, the same below) divides $\mk{B}_{g_i}|_{\partial B_R(0)}$ into two components: the inward pointing one and the outward pointing one, which we denote by
\begin{equation}
\partial_{\mp}\mk{B},
\end{equation}
where $-$ corresponds to the inward component and $+$ corresponds to the outward component.
Then we can define the ($R$-truncated) scattering relation of $g_i(t)$ to be the map from $\partial_- \mk{B}$ to $\partial_+ \mk{B}$ by
\begin{equation} \label{eq:Cli-truncated-def}
\Cli^{R}: \quad (t,z,v) \to (t,z',v'),
\end{equation}
where $(z,v),  (z',v') \in S^*_{\partial B_R(0)} B_R(0)$ and $(z',v')$ is the location and direction exiting $B_R(0)$ (in the sense reaching $\partial B_R(0)$) for the first time along the unit speed geodesic with respect to $g_i(t)$, and we denote its length by
\begin{equation} \label{eq:definition-length-map}
\ell_{g_i}:  \partial_- \mk{B} \to [0,\infty).
\end{equation}
Here $(z',v')$ always exists and depends on $(z,v)$ smoothly due to our non-trapping condition and compactness\footnote{One would also have such nice properties of geometric scattering relations even if the region is not compact if one has certain assumption on the metric or introduces certain resolutions. For example, see \cite{melrose1996scattering}\cite{jia2022tensorial} for the case of asymptotically conic manifolds.} of the region we are considering. Also, $\ell_{g_i}$ always has a finite value for the same reason.

Next we discuss how to recover the scattering relation truncated to a large ball from the scattering relation at infinity.
Let $\iota_i$ be the map defined in \eqref{eq:iota-def} with $g$ replaced by $g_i$. 
For any $q \in \overline{{}^{\oc}T^*_{\partial \overline{\R^n} }\R^n}$ with $(\xi_{\oc}, \eta_{\oc})$ in a bounded region as in the last part of Proposition~\ref{prop:classical-SC-geodesic}, then the identification in Proposition~\ref{prop:Wpm-1c-identification} above sends it to a point $q' \in W_+$, which in turn is identified with a bicharacteristic line within the time slice $\{ t = -\frac{1}{2}\xi_{\oc} \}$. 
Then this bicharacteristic line passes through $\partial_+\mk{B}$ at a point $(t,z',v') \in \partial_+\mk{B}$, which is the place where such a (lifted) geodesic leaves $\mk{B}_{g_i}$, and we define $\mathsf{I}_{+}^{*}$ by
\begin{equation} \label{eq:I_+-star-def}
\mathsf{I}_{+}^{*}(q) = (t, z', v') \in \partial_+\mk{B}.
\end{equation}
This is smooth in $q$ because $\iota_i((t,z',v'))$ can be obtained from $q'$ through a finite time flow of $H_p^{2,0}$, and the map sending $q$ to $q'$ and $\iota_i^{-1}$ are smooth as well.
Here $\mathsf{I}_{+}^{*}$ is the map sending the endpoint of a lifted geodesic to the place where it leaves $B_R(0)$. 
In addition, this only concerns the part of the bicharacteristic line after escaping the metric perturbation, hence this definition is independent of the choice of using either $g_1$ or $g_2$, which justifies the omission of the subscript $i$.
In the same way, we can define $\msf{I}_-^*$ by 
\begin{equation} \label{eq:I_--star-def}
\msf{I}_-^*: \; 
    q \in \overline{{}^{\oc}T^*_{\partial \overline{\R^n} }\R^n}  \to  (t,z,v) \in \partial_-\mk{B},
\end{equation}
which is the point that the bicharacteristic corresponding to $q$ first enters $\mk{B}$. 
Notice that $g_1(t)$ and $g_2(t)$ coincide outside $B_R(0)$, so $(t,z',v')$ above does not depend on using $\iota_{1}$ or $\iota_{2}$, and so do $\mathsf{I}_{\pm}^{*}$. Now we are ready to state the main result of this part: the classical scattering map $\Cli$ determines their truncated version $\Cli^R$ in \eqref{eq:Cli-truncated-def}.

\begin{prop} \label{prop:determine-truncated-SC-map}
Let $g_1,g_2$ be as in \eqref{eq:g-definition-perturbation}, suppose that $\Clo$ and $\Clt$ have the same restriction on $\overline{{}^{\oc}T^*_{\partial \overline{\R^n}}\R^n}$, then we have $\Clo^R=\Clt^R$. 
\end{prop}

\begin{proof}
Starting from $(t,z,v) \in \partial_- \mk{B}$, we first identify such point with a point in $\Char(P_i)$ (the part in fiber infinity) using $\iota_i$ defined as in \eqref{eq:iota-def}.
From this point, we go backward (using $\mathsf{I}_-$ in \eqref{eq:msf-I-def}) to $W_-$, then apply $\Clo=\Clt$, and finally go backward along the geodesic flow again to go back to $\partial_+ \mk{B}$ (using $\mathsf{I}_+^*$). This final image is  $\Cli^R(t,z,v)$ by definition.
Here we used the fact that $g_1,g_2$ coincide outside $B_R(0)$ and on its boundary, hence $\msf{I}_-$ defined in \eqref{eq:msf-I-def} using $g_1$ and $g_2$ coincide on the image of $\partial_-\mk{B}$ under $\iota_1$ and $\iota_2$.
Summarizing the process above, we have the following formula for $\Cli^R$ showing that it can be determined by $\Cli$:
\begin{equation} \label{eq:Cli-R-from-Cli}
    \mathsf{I}_{+}^* \circ \Cli \circ \mathsf{I}_-\circ \iota_i|_{\partial_-\mk{B}} = \Cli^R.
\end{equation}

Using again the fact that $g_1$ and $g_2$ coincide on the boundary of and outside $B_R(0)$, we know $\iota_1|_{\partial_-\mk{B}} = \iota_2|_{\partial_-\mk{B}}$ and \eqref{eq:Cli-R-from-Cli} shows that $\Cli^R$, $i=1,2$ will be the same if $\Cli$ are the same.
\end{proof}

\subsection{The total sojourn time and the lens equivalence}
\label{subsec:total-sj-relation}
For $i=1,2$, let $\Lambda_\pm^i$ be $\Lambda_\pm$ defined in \eqref{eq:microlocalized sojourn relns} using bicharacteristics of $P_i$, $\Leg^i = \partial \Lambda_\pm^i$, and $\Cli$ be $\Cl$ defined in \eqref{eq:classical-sc-map-def} with $g$ replaced by $g_i$. 
In this subsection, we analyze phase functions parametrizing $\Lambda_\pm^i$ and $\tilde{\beta}_b^*(\mathrm{Gr}(\Cli)')$ to show that they encode the information of lengths of geodesics.
More precisely, our calculus of Legendrian distributions allows us to record the sojourn time defined in \eqref{eq:def-totoal-sojourn-time} below,
which is analogous to the Buseman function in more general contexts.  
This sojourn time, in turn, determines the length of geodesics restricted to $B_R(0)$. 
The sojourn time was considered by Guillemin \cite{Guillemin-sojourn} in the setting of scattering by an obstacle. 
Let $\ell_{g_i(t)}$ be the length of geodesic defined as in \eqref{eq:def-length-g} with $g$ being $g_i$.

We begin with the following concept of microlocal lens equivalence, which reveals how we capture the lens data defined in \eqref{eq:scattering-relation}\eqref{eq:def-length-g} using Lagrangian submanifolds arising in microlocal analysis.
\begin{definition} \label{definition:microlocal-lens-equivalence}
Let $[\ococb,J]$ be as in \eqref{eq:J-resolved-1c1c-phase-space} and $\beta_{SJ}$ being the blow-down map there, we say that $\Clo$ is lens equivalent to $\Clt$ if
\begin{equation}
    \partial \beta_{SJ}^*\tilde{\beta}_b^*(\mathrm{Gr}(\Clo)')
    = \partial \beta_{SJ}^*\tilde{\beta}_b^*(\mathrm{Gr}(\Clt)').
\end{equation}
This is equivalent to requiring that $\partial \tilde{\beta}_b^*(\mathrm{Gr}(\Clo)') = \partial \tilde{\beta}_b^*(\mathrm{Gr}(\Clt)')$ and $N_1$ coincide as a function restricted to them.
When this holds, we will also say that $g_1$ and $g_2$ are microlocally lens equivalent.
\end{definition}

The goal of this subsections is to prove the following:
\begin{thm} \label{thm:Cl-determines-length}
Suppose $g_1,g_2$ are as in \eqref{eq:g-definition-perturbation} with $\Clo$ microlocally lens equivalent to $\Clt$, then we have $\ell_{g_1(t)} = \ell_{g_2(t)}$.
\end{thm}

We will reduce this to showing that $\Cli$ determines the total sojourn time, which encodes the difference of the length of geodesic compared with the Euclidean case.
Then since the sojourn time does not change outside $[-T,T] \times B_R(0)$, this means the sojourn time of $g_1$ and $g_2$ restricted to $[-T,T] \times B_R(0)$ are the same, which implies $\ell_{g_1(t)} = \ell_{g_2(t)}$.
We begin by giving the definition of the total sojourn time. For a geodesic $\gamma(s)$ with respect to $g_i(t)$ with arc-length parametrization, we denote the $x_{\ps}$-component of $\gamma(s)$ by $x_{\ps}(\gamma(s))$ and define its \emph{total sojourn time} to be
\begin{equation} \label{eq:def-totoal-sojourn-time}
	\TSJ_{g_i(t)}(\gamma) =  \lim_{s \to \infty, s' \to -\infty} (s  -s' - \frac{1}{x_{\ps}(\gamma(s))} -s' - \frac{1}{x_{\ps}(\gamma(s'))}).
\end{equation}
Suppose we can show that two metrics $g_1(t)$, $g_2(t)$ with $\Clo$ microlocally lens equivalent to $\Clt$ have the same total sojourn time for their geodesics $\gamma_i(s)$ starting at $(z,v) \in \partial_- SB_R(0)$, we will show $\ell_{g_1(t)}(z,v) = \ell_{g_2(t)}(z,v)$.
Let $(z',v') \in \partial_+ SB_R(0)$ be the point and direction at which $\gamma_i(s)$ exit $B_R(0)$, then $\TSJ_{g_1(t)}(\gamma_1) = \TSJ_{g_i(t)}(\gamma_2)$ gives
\begin{equation} \label{eq:SJ-equal-3}
	\lim_{s_1 \to \infty, s_1' \to -\infty} (s_1  -s_1' - \frac{1}{x_{\ps}(\gamma_1(s_1))} - \frac{1}{x_{\ps}(\gamma_1(s_1'))})
	= \lim_{s_2 \to \infty, s_2' \to -\infty} (s_2  -s_2' - \frac{1}{x_{\ps}(\gamma_2(s_2))} - \frac{1}{x_{\ps}(\gamma_2(s_2'))}).
\end{equation}
Let $s_{i,\pm}$ be the time at which $\gamma_i$ enters and leaves $B_R(0)$ in the chosen parametrization, which means 
\begin{align}
	(z,v) = (\gamma_i(s_{i,-}),\dot{\gamma}_i(s_{i,-})), \quad (z',v') = (\gamma_i(s_{i,+}),\dot{\gamma}_i(s_{i,+})).
\end{align}
Since the part of $\gamma_1$ with $s>s_{1,+}$ (resp. $s<s_{1,-}$) coincide with the part of $\gamma_2$ with $s>s_{2,+}$ (resp. $s<s_{2,-}$), we know the increment of $s_i - \frac{1}{x_{\ps}(\gamma_i(s_i))}$ from $s_{i,+}$ to $ \infty$ are the same for $i=1,2$. Also, the increment of $-s_i'-\frac{1}{x_{\ps}(\gamma_i(s'_i))}$ from $s_i'=s_{i,-}$ to $-\infty$ are the same. In combination with \eqref{eq:SJ-equal-3}, we have
\begin{equation} \label{eq:SJ-equal-4}
	 (s_{1,+}  -s_{1,-} - \frac{1}{x_{\ps}(\gamma_1(s_{1,+}))} - \frac{1}{x_{\ps}(\gamma_1(s_{1,-}))})
	= (s_{2,+}  -s_{2,-} - \frac{1}{x_{\ps}(\gamma_2(s_{2,+}))} - \frac{1}{x_{\ps}(\gamma_2(s_{,2-}))}),
\end{equation}
which gives 
\begin{equation}
	s_{1,+}  -s_{1,-} = s_{2,+}  -s_{2,-},
\end{equation}
and this is equivalent to $\ell_{g_1(t)} = \ell_{g_2(t)}$.
So we are reduced to showing 
\begin{align} \label{eq:mark-sojourn-equal}
   \TSJ_{g_1(t)}(\gamma_1) = \TSJ_{g_i(t)}(\gamma_2).
\end{align}

We use upper indices to indicate to which metric those phase functions and parameters are associated to.
Concretely, let 
\begin{equation}
\Phi_{\oc-\oc}^i(\msf{K}, \theta_0^i, \theta_1^i) =	\frac{ t(1-\sigma^2) }{x_{\oc}^2} + \frac{\varphi^i_1(\msf{K},\theta_0^i,\theta_1^i)}{x_{\oc}}, \quad \msf{K} = (\xoco, \sigma, \yoco, \yoct), 
\end{equation}
be as in \eqref{eq:1c1c-phase-normal-form-concrete} with $g$ being $g_i$. Recall the formula for $S$ in \eqref{eq:S formula}, we know that $\Phi_{\oc-\oc}^i$ parametrizing $(\tilde{\beta}_b)^*(\mathrm{Gr}(\Cli)')$ can be chosen to be
\begin{equation}
\Phi^i_{\oc-\oc} = 	-\Phi^i_{\oc-\ps,+} + \Phi^i_{\oc-\ps,-},
\end{equation}
where 
\begin{equation}
\Phi^i_{\oc-\ps,\pm} = -\frac{t}{x_{\oc,\pm}}+\frac{\varphi^i_\pm(x_{\oc,\pm},\msf{K}_\pm',\tilde{\zeta}''_\pm)}{x_{\oc,\pm}},
\; \msf{K}'_\pm = (t,z,y_{\oc,\pm})	
\end{equation}
 parametrizes $\Lambda^i_\pm$ non-degenerately in the sense of Definition~\ref{def:1c-ps-parametrization}. 
In addition, we can choose $\varphi_\pm^i$ in the form of \eqref{eq:1c-ps-phase-normal-form}, so that it satisfies \eqref{eq:varphi1-const-speed}.\footnote{In fact, the property in \eqref{eq:varphi1-const-speed} is satisfied by other phase functions not necessarily in this normal form as well, since the phase equivalence shows that function value of phase functions at their critical points parametrizing the same point on the Legendre submanifold are the same.}  
So we have, with $\varphi_\pm^i$ satisfying \eqref{eq:varphi1-const-speed},
\begin{equation} \label{eq:1c1c-phase-normal-form-concrete}
	\Phi^i_{\oc-\oc} =   \frac{t(1 - \sigma^2)}{x_{\oc,+}^2}
+ \frac{-\varphi^i_{+}(x_{\oc,+},\msf{K}'_+, \tilde{\zeta}''_+) + \sigma \varphi^i_-(\sigma^{-1} x_{\oc, _+}, \msf{K}'_-,\tilde{\zeta}''_-)}{x_{\oc,+}},
\end{equation}
where $\sigma = x_{\oc,+}/x_{\oc,-}$.

\begin{prop} \label{prop:varphi1-equals-sojourn}
 Let $\Phi_{\oc-\oc}^i$ be as in \eqref{eq:1c1c-phase-normal-form-concrete}, and $q' \in C_{\Phi}$ corresponds to $q \in \Leg^i$, then 
 \begin{equation} \label{eq:varphi1-equals-sojourn}
     \varphi_1^i(q') = \TSJ_{g_i(t)}(\gamma_i)
 \end{equation}
 where $\gamma_i$ is the geodesic with respect to $g_i$ associated to $q$ as in Proposition~\ref{prop:classical-SC-geodesic}.
\end{prop}

\begin{proof}

We denote the subleading part of $\Phi^i_{\oc-\oc}$ above by
\begin{equation} \label{eq:phi1-difference}
T^i =  -\varphi^i_{+}(x_{\oc,+},t,z,y_{\oc,+},\tilde{\zeta}^{'',i}_+) +\sigma \varphi^i_{-}(\sigma^{-1}x_{\oc,+},t,z,y_{\oc,-},\tilde{\zeta}^{'',i}_-),
\end{equation}
the starting point at $\partial_- \mk{B}$ by $(t,z,v)$, the endpoint at $\partial_+ \mk{B}$ by $(t,z',v')$, and denote the geodesic connecting them by $\gamma$. 

Recall that using $\iota_i$ in \eqref{eq:iota-def} (defined using $g_i$ instead), we can identify the part of $\Char(P_i)$ at fiber infinity with $\mk{B}_{g_i}$ defined in \eqref{eq:parametrized-cosphere-def}.
Under this identification, our bicharacteristics are identified with lifted geodesics of $g_i$ in the cosphere bundle. In addition, when taking the parametrization into account, the integral curve of $H_p^{2,0} = \rho_{\ps}H_p$\footnote{Again, we are only concerning $(t,z)$ away from spacetime infinity, hence we can take $x_{\ps}=1$.} corresponds to a geodesic with constant speed $2$ by the expression of $H_p^{2,0}$ in \eqref{eq:Hp-varphi1}.

On the other hand, points on a bicharacteristic line corresponds to a point in $\Lambda_\pm$ by its definition in \eqref{eq:microlocalized sojourn relns}, which in turn corresponds to a point in the critical set of $\varphi_\pm^i$ via the inverse of our parametrization map in Definition~\ref{def:1c-ps-parametrization} since our parametrization is non-degenerate. 

Composing those two identifications, we obtain a local diffeomorphism
\begin{equation}
\mathscr{F}^i_\pm: \quad  \mk{B}_{g_i} \to C_{\Phi_{\oc-\ps,\pm}^i}
\end{equation}
so that for $(t,z,v) \in \mk{B}_{g_i}$, 
let $q \in \Lambda_\pm^i$ be the point corresponds to $(t,z,v)$ as above, then $\mathscr{F}^i_\pm(t,z,v)$ is the point that is sent to $q$ under the parametrization map.

Let $\overline{y_{\oc,\pm}} \in \R^n$ be the image of $y_{\oc,\pm} \in \partial \overline{\R^n} = \mathbb{S}^{n-1}$ under the embedding $\mathbb{S}^{n-1} \hookrightarrow \R^n$.
Since the part of $\Lambda_\pm^i$ corresponds to $\partial_\pm \mk{B}$ coincide with the sojourn relation $\Lambda_0$ associated to the free operator $P_0$,  
applying Corollary~\ref{coro:1cps-phase-value-equal}, the value of the phase function on the image of $\partial_-\mk{B}$ coincide with the phase function parametrizing $\Lambda_0$ as well.
Using \eqref{eq:free-Poisson-kernel-2}, the function value of $\varphi_{-}^i$ at the starting point $\mathscr{F}^i_-(t,z,v)$ with $(t,z,v) \in \partial_-\mk{B}$ 
is determined by $z$ and $\overline{y_{\oc,-}}$ directly: 
\begin{equation}
    \varphi_{-}^i(\mathscr{F}^i_-(t,z,v)) = z \cdot \overline{y_{\oc,-}}.
\end{equation}

As aforementioned, the integral curve of $H_p^{2,0}$ is identified with a lifted geodesic in the cosphere bundle with constant speed 2, then in combination with \eqref{eq:varphi1-const-speed}, we know that the increment of $\varphi_-^i$ equals to the length of the geodesic. 
Though \eqref{eq:varphi1-const-speed} is only valid over a region of local parametrization, but using Corollary~\ref{coro:1cps-phase-value-equal} or Proposition~\ref{prop:1cps-sigma-equals-1}, we know that the phase function restricted to their critical sets can be viewed as a function on $\Lambda_\pm$ globally and the fact that the increment of $\varphi^i_-$ equals to the length of the corresponding part of the geodesic continues to holds globally.
So let $(t,z',v') \in \partial_+\mk{B}$ be the ending point of this geodesic, $\varphi_{-}^i$ at $\mathscr{F}(t,z',v')$ 
is $z \cdot \overline{y_{\oc,-}} + \ell_{g_i}(t,z,v)$ (where $\ell_{g_i}$ is defined in \eqref{eq:definition-length-map}).
On the other hand, at the ending point we have $\varphi_{+}(\mathscr{F}^i_+(t,z',v')) = z' \cdot \overline{y_{\oc,+}}$ since $\Lambda_+^i$ coincide with $\Lambda_0$ there.
So on the critical set of the phase function, we have
\begin{equation}  \label{eq:sojourn-expression-1}
T^i =  - z' \cdot \overline{y_{\oc,+}} + \ell_{g_i}(t,z,v)  + z \cdot \overline{y_{\oc,-}}.
\end{equation}

\begin{figure}
\includegraphics[scale=0.5]{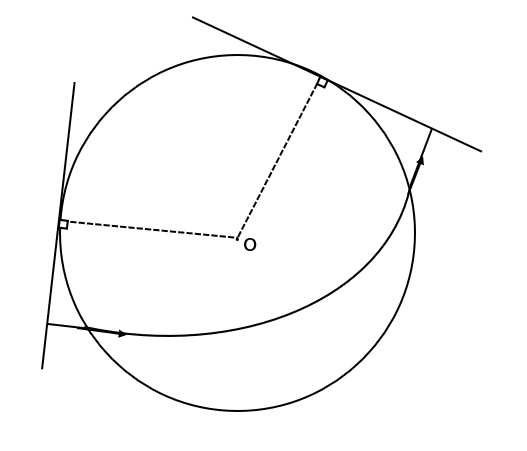}
\caption{Illustration of the sojourn time, with the $t$ direction compressed and only one of $\gamma_i$ is drawn as the solid curve. The momentum variables are indicated as arrows. The length of the solid curve subtract the sum of the length of two dashed lines is the sojourn time.}
\end{figure}

Now we use this expression of $T^i$ to show that $T^i$ equals to the sojourn time $\mk{T}_{g_i}$ defined in \eqref{eq:def-totoal-sojourn-time}.
We decompose $\gamma_i$ into three parts: we denote the part before entering $\mk{B}_{g_i}$ by $\gamma_{i,-}$, the part inside $\mk{B}_{g_i}$ by $\gamma_{i,0}$, and the part after exiting $\mk{B}_{g_i}$ by $\gamma_{i,+}$.
Notice that if we allow $z,z'$ in \eqref{eq:sojourn-expression-1} to be other points on $\gamma_i$ as long as they are on the side before and after $\gamma_i$ encounters $\mk{B}_{g_i}$, and replace $\ell_{g_i}(t,z,y_{\oc,-})$ by the length of the part of $\gamma_i$ connecting them, which we denote by $\tilde{\ell}_{g_i}(t,z,y_{\oc,-})$ below, then this $T$ will remain unchanged. 
This is because, if we move $z$ backward along the geodesic, then $z \cdot \overline{y_{\oc,-}}$ is decreasing with unit speed and $\tilde{\ell}_{g_i}$ is increasing with unit speed.
On the other hand, if we move $z'$ forward along the geodesic, then $-z' \cdot \overline{y_{\oc,+}}$ is decreasing with unit speed while $\tilde{\ell}_{g_i}$ is increasing with unit speed.

Now we have shown that we can extend either ends of the part we count this length while keeping $T^i$ unchanged, and the next step is to show
that $T^i$ equals to the sojourn time in Guillemin \cite[Section~3]{Guillemin-sojourn} (in particular, the discussion after \cite[Figure~2]{Guillemin-sojourn}), which is to extend the part of the geodesic to the `plane of incidence' and the `plane of refraction'. More concretely, we consider the tangent plane of $\partial_-\mk{B}$ that is perpendicular to $\overline{y_{\oc,-}}$. There are two such planes and we choose the one intersect $\gamma_{i,-}$ and set $(t,\tilde{z})$ to be this point of intersection.
Similarly, we consider the tangent plane of $\partial_+\mk{B}$ that is perpendicular to $\overline{y_{\oc,+}}$ and intersects $\gamma_{i,+}$, and let $(t,\tilde{z}')$ be the point where $\gamma_{i,+}$ intersects this plane. Then we have
\begin{equation}
T^i = - \tilde{z}' \cdot \overline{y_{\oc,+}} + \tilde{\ell}_{g_i}(t,\tilde{z},y_{\oc,-})  + \tilde{z} \cdot \overline{y_{\oc,-}},
\end{equation}
which equals to the sojourn time of $g_i(t)$ used in \cite[Section~3]{Guillemin-sojourn}.

Next we show that this equals to the total sojourn time introduced in \eqref{eq:def-totoal-sojourn-time}, which is also used by Hassell and Wunsch \cite{hassell-wunsch2005schrodinger}\cite[Section~15]{Hassell-Wunsch-semiclassical-resolvent}. 
Let $z^{\perp}$ be the component of $z$ that is perpendicular to $\overline{y_{\oc,-}}$, which remain the same for all $z$ on $\gamma_{i,-}$.
Then we know, with $|\cdot|$ denoting the Euclidean length,
\begin{equation}
|z^{\perp}|^2 + |z \cdot \overline{y_{\oc,-}}|^2 = |z|^2,
\end{equation}
which shows $|z \cdot \overline{y_{\oc,-}}| - |z| \to 0$ as $|z| \to \infty$ along $\gamma_{i,-}$ backward. And similarly we know $|z' \cdot \overline{y_{\oc,+}}| - |z'| \to 0$ when $|z'| \to \infty$ along $\gamma_{i,+}$. So it does not matter to use $|z|$ and $|z'|$ or $z \cdot \overline{y_{\oc,-}}$ and $z' \cdot \overline{y_{\oc,+}}$ to renormalize the length when we take this limit and our $T^i$ thus coincide with the sojourn time defined in \cite[Section~15]{Hassell-Wunsch-semiclassical-resolvent}, which applies to our setting for each fixed $t$.
See also \cite[Equation~(1.4)]{hassell-wunsch2005schrodinger} for the version that take `half' of it, that is, only sending one end to infinity in the process above.
\end{proof}

Now we return to the proof of Theorem~\ref{thm:Cl-determines-length}.
\begin{proof}[Proof of Theorem~\ref{thm:Cl-determines-length} resumed]
Let $\Phi_{\oc-\oc}^1$ and $\Phi_{\oc-\oc}^2$ be two phase functions in the form \eqref{eq:1c-1c-phase-normal-form}  parametrizing $\tilde{\beta}_b^*(\Clo)$ and $\tilde{\beta}_b^*(\Clt)$ cleanly respectively.
Since $\Clo$ is microlocally lens equivalent to $\Clt$, if both $q'_i \in C_{\Phi^i_{\oc-\oc}}$, $i=1,2$ are sent to $q \in \partial \tilde{\beta}_b^*(\mathrm{Gr}(\Cli)')$ under the parametrization map, then by Proposition~\ref{prop:N1-equals-1c1c-phase}, both $\varphi_1^1(q'_1)$ and $\varphi_1^2(q'_2)$ equals to the $N_1$ component of the point in the further lift of $q$ in $[\ococb;J]$, hence
\begin{equation}
    \varphi_1^1(q'_1) = \varphi_1^2(q'_2).
\end{equation}

On the other hand, Proposition~\ref{prop:varphi1-equals-sojourn} shows that the function value of $\varphi_1^i$ at $q'_i$ equals to the sojourn time of $g_1(t)$ and $g_2(t)$ of the corresponding geodesic in Proposition~\ref{prop:classical-SC-geodesic}, hence we have $\mk{T}_{g_1}(\gamma_1) = \mk{T}_{g_2}(\gamma_2)$ when they correspond to the same point in $\beta_b^*(\mathrm{Gr}(\Cli)) \cap \ffococ$. 
Combining with discussion before \eqref{eq:mark-sojourn-equal}, we know this gives $\ell_{g_1(t)} = \ell_{g_2(t)}$ and completes the proof of Theorem~\ref{thm:Cl-determines-length}. 
\end{proof}

We conclude this subsection by the following analogue of \cite[Lemma~4.1.3]{Duistermaat-FIO-book}. It shows that, just as in discussion above where we glue the phase function `globally', the Maslov bundle can also be trivialized globally if we relax the requirement on the phase shift, allowing them to be not necessarily integer multiples of $\pi/2$. 
We should emphasize that this has different nature with the phase equivalence like Proposition~\ref{prop: phase-equivlaence-sec-microlocal} even though they are similar. The value of the phase function is addressing the high-frequency oscillation. If we can't glue them together, then it causes $O(x_{\oc}^{-1})$ level phase shift between different parametrizations. In contrast, the Maslov factor only causes $O(1)$ level phase shift which does not affect the overall smoothness or oscillation behavior, and it is usually incorporated with the amplitude or the symbol.
\begin{prop} \label{prop:Maslov-global-trivialization}
The Maslov bundle over any $\oc-\ps$ Legendrian submanifold $\Legps$ can be trivialized globally, as a complex line bundle\footnote{It is generally non-trivial as a $\mathbb{Z}/4\mathbb{Z}$-bundle, where $\mathbb{Z}/4\mathbb{Z}$ is the multiple of $\frac{\pi}{2} i$ in the phase. It is because we relaxed the condition to allow the $1$-chain to be real valued instead of integer (modulo a constant factor) valued, so that the closed $2$-chain becomes exact.}. 
That is, there are local trivializations $\{(U_j,\phi_j)\}$:
\begin{equation}
  \phi_j: \; M(\Legps)|_{U_j} \to U_j  \times \mathbb{C},
\end{equation}
and smooth (real-valued) functions $\chi_j \in C^\infty(U_j)$ such that
\begin{equation}
e^{-i\chi_k} \cdot \phi_k = e^{-i\chi_j} \cdot \psi_j,
\end{equation}
where $\cdot$ is the multiplication action on the fiber part.
\end{prop}

\begin{proof}
The proof is similar to that of \cite[Lemma~4.1.3]{Duistermaat-FIO-book}, which we sketch here for the completeness.
Let $e^{i\psi_{kj}}$ be the transition factor from $j$-th patch to the $k$-th patch:
\begin{equation}
\phi_k = e^{i \psi_{kj}} \cdot \phi_j.
\end{equation}
Then the compatibility condition of this transition factor is:
\begin{equation}
\psi_{\ell j} = \psi_{\ell k} + \psi_{kj},
\end{equation}
which means $\{ \psi_{kj} \}$ form a closed $2$-chain. But the $C^\infty$-sheaf over $\Legps$ is a  fine sheaf\footnote{One can see this by using a partition of unity.}, so $H^2(\Legps,C^\infty) = 0$ and this closed 2-chain is exact, which is equivalent to the existence of smooth $\{\chi_j\}$ such that $\psi_{kj} = \chi_k - \chi_j$. 

\end{proof}

\subsection{Determining the 1-jet}
\label{subsec:determine-1-jet}

In this subsection we show that if the assumption of Theorem~\ref{thm:main-intro} is satisfied, then $g_1$ and $g_2$ are lens equivalent in the sense of Definition~\ref{definition:microlocal-lens-equivalence}.

\begin{prop} \label{prop:Cl1=Cl2}
 Suppose that either $g_1$ or $g_2$ admits a convex function in the sense of Definition~\ref{definition:metric-classes} and $S_{g_1}-S_{g_2}$ is a compact operator from $L^2(\R^n)$ 
 to itself, then $\Clo$ is microlocally lens equivalent to $\Clt$ in the sense of Definition~\ref{definition:microlocal-lens-equivalence}.
\end{prop}

We have the following criterion on the compactness of $\oc-\oc$ Fourier integral operators. 
\begin{lmm} \label{lemma:1c1c-FIO-compact-criterion}
    Suppose $A \in I_{\oc-\oc}^m(X_b^2,\Lag)$ for an admissible 1c-1c Lagrangian submanifold $\Lag$ with $m<0$, then $A$ acts as a compact operator from $H_{\oc}^{s,r}(\R^n)$ to itself.

    On the other hand, if $A \in I_{\oc-\oc}^0(X_b^2,\Lag)$ with $\Lag'$ being a graph of a symplectomorphism,
    and $A$ is elliptic at $q \in \Leg = \partial \Lag$, then $A$ is not compact from any $H_{\oc}^{s,r}(\R^n)$ to itself.
\end{lmm}
\begin{proof}
First of all, $A^* \in I_{\oc-\oc}^m(X_b^2,\Lag^*)$, where $\Lag^*$ obtained by switching the left hand right variables of $\Lag$ and the sign of all frequencies. Then we know that the composition of the corresponding canonical relation
\begin{equation}
    (\Lag^*)' \circ \Lag',
\end{equation}
where the prime stands for switching the sign of frequencies lifted from the right variable,
is the identity map, which means that the corresponding Lagrangian submanifold is the conormal bundle of the lifted diagonal ${}^{\oc}N^*\Delta_b$ in $\ococb$. 
Then by the composition law in Proposition~\ref{prop: 1c-1c transversal composition}, we have
\begin{align} \label{eq:A*A-1c1c}
A^*A \in I^{2m}_{\oc-\oc}(X^2_b;{}^{\oc}N^*\Delta_b).
\end{align}
On the other hand, the class $I^{2m}_{\oc-\oc}(X^2_b;{}^{\oc}N^*\Delta_b)$ is precisely the class of $\oc$-pseudodifferential operators
\begin{align*}
\Psi_{\oc}^{-\infty,2m}(\R^n),
\end{align*}
defined in \cite[Section~2.3]{zachos2022inverting} (see also \cite[Section~2.1]{jia2022tensorial}), where $-\infty$ is the differential order and $2m$ is the decay order.
In particular, $A^*A$ acts as a bounded map
\begin{equation}
    H_{\oc}^{s,r}(\R^n) \to H_{\oc}^{s+s',r-2m}(\R^n), \; \forall s' \in \R.
\end{equation}
When $m<0$, taking $s'>0$, we know that it is compact operator from $H_{\oc}^{s,r}(\R^n)$ to itself.

For the second part, we take $m=0$ in \eqref{eq:A*A-1c1c}, and also use the last part of Proposition~\ref{prop: 1c-1c transversal composition} now.
Let $q_r \in \ocphase$ be the right projection of $q$ and $q_r'$ be the point with the sign of its frequency changed, then $A^*A$ is elliptic at the lift of $(q_r',q_r)$ as an element of $I^{0}_{\oc-\oc}(X^2_b;{}^{\oc}N^*\Delta_b)$, which is equivalent to that $A^*A \in \Psi_{\oc}^{-\infty,0}(\R^n)$ is elliptic at $q_r'$. Such operator is not compact from any $H^{s,r}_{\oc}(\R^n)$ to itself, so does $A$, since otherwise $A^*A$ would be a composition of a compact operator with a bounded operator, which is compact.
\end{proof}

\begin{remark} \label{remark:coincide-leading-order-vs-compact}
Lemma~\ref{lemma:1c1c-FIO-compact-criterion} shows that, if $S_{g_1}$ and $S_{g_2}$ coincide in the sense that their principal symbols are the same, or equivalently
    \begin{equation} \label{eq:Sg1-Sg2-leading-same-remark}
    S_{g_1} - S_{g_2} \in  I^{-1}_{\oc-\oc}(X_b^2,\tilde{\beta}_b^*(\mathrm{Gr}(\Clo)')),
\end{equation}
then the condition ``$S_{g_1}-S_{g_2}$ acts as a compact operator on $L^2(\R^n)$ or other $\oc$-Sobolev spaces'' in Theorem~\ref{thm:main-intro} is satisfied. So our main theorem is formulated to state the stronger version (i.e. with weaker condition) among two questions under conditions \eqref{eq:S1-S2-condition-1}\eqref{eq:S1-S2-condition-2}.
\end{remark}

Now we prove Proposition~\ref{prop:Cl1=Cl2}.
\begin{proof}[Proof of Proposition~\ref{prop:Cl1=Cl2}]
  Suppose $\Clo$ is not microlocally lens equivalent to $\Clt$, then either the restriction to $\overline{{}^{\oc}T^*_{\partial\overline{\R^n}}\R^n}$ are different or the lift of $\mathrm{Gr}(\Clo)$ and $\mathrm{Gr}(\Clt)$ in the further lift of $\ffococ$ $[\ococb;J]$ have different $N_1$-component defined in \eqref{eq:N1-def} at some point with the same projection in $\ffococ \subset \ococb$.

For the first case, $\mathrm{Gr}(\Clo)$ and $\mathrm{Gr}(\Clt)$ will be different on the level of 1-cusp frequencies and it is sufficient to use 1-cusp pseudodifferential operators to microlocalize $S_{g_i}$ to produce a contradiction.
More precisely, we only need to show that we can take pseudodifferential operators $Q,Q' \in \Psi_{\oc}^{0,0}$, which are bounded from $L^2(\R^n)$ to $L^2(\R^n)$, so that
\begin{equation}
    Q (S_{g_1}-S_{g_2}) Q_{\oc}Q'
\end{equation}
is not compact, where $Q_{\oc}$ satisfy conditions after \eqref{eq:Q1c-Qps-def}. In fact, we will choose $Q,Q'$ so that $QS_{g_2}Q_{\oc}Q'$ is compact from $L^2(\R^n)$ to itself while $QS_{g_1}Q_{\oc}Q'$ is not.

By the condition in this case, there exists $(q_0,\tilde{q}_0) \in \overline{{}^{\oc}T^*_{\partial\overline{\R^n}}\R^n} \times \overline{{}^{\oc}T^*_{\partial\overline{\R^n}}\R^n}$
such that 
\begin{equation}
    (q_0,\tilde{q}_0) \in \mathrm{Gr}(\Clo) \subset \ocphase \times \ocphase, \; (q_0,\tilde{q}_0) \notin \mathrm{Gr}(\Clt).
\end{equation}
Then we choose $Q_0,Q_0' \in \Psi_{\oc}^{0,0}(\R^n)$ such that $Q_0$ (resp. $Q_0'$) is elliptic at $q_0$ (resp. $\tilde{q}_0$) and with $\WF'_{\oc}(Q_0)$ (resp. $\WF'(Q_0')$) contained in a small neighborhood of $q_0$ (resp. $\tilde{q}_0$). In particular, we can assume that
\begin{equation}
    \text{There is no point } q \in \WF'_{\oc}(Q_0), \, q' \in \WF'_{\oc}(Q_0') \text{ such that }
    (q,q') \in \mathrm{Gr}(\Clt).
\end{equation}

Then by \cite[Proposition~7.6]{HJ2026-scattering-map} (and the definition of $\oc-\oc$ wavefront set there), we know
\begin{equation}
    QS_{g_2}Q_{\oc}Q' \in \mathcal{S}(\R^n \times \R^n).
\end{equation}

On the other hand, by the definition of $Q_{\oc}$, both $\Clo$ and $\Clt$ coincide with identity outside the region on which $Q_{\oc}$ is microlocally identity, so we know $q_0$ is in such a region.
Applying Proposition~\ref{prop: 1c-1c transversal composition} repeatedly, we know that 
\begin{equation}
    QS_{g_1}Q_{\oc}Q' \in I_{\oc-\oc}^0(X_b^2, \tilde{\beta}_b^*(\mathrm{Gr}(\Clo)')),
\end{equation}
and is elliptic at the lift of $(q_0,\tilde{q}_0')$ in $\tilde{\beta}_b^*(\mathrm{Gr}(\Clo)')$.
Applying the second part of Lemma~\ref{lemma:1c1c-FIO-compact-criterion}, we know this is not compact from $L^2(\R^n)$ to itself and finishes the proof in this case.

Now we consider the second case, where $\Clo$ and $\Clt$ have the same restriction to the boundary, but lifts of their twisted graphs in $[\ococb;J]$ have different $N_1$-components. In this case, we need to use tools developed in Section~\ref{sec:second-micro} to further microlocalize to the scattering frequency level.
Let coordinates be as in \eqref{eq:coordinate-near-ffSJ} with $1,2$ replaced by $+,-$, then there is a point $q_0 = (\sigma_0 = 1,x_{\oc,+0}=0,y_{\oc,+0},y_{\oc,-0}, N_{10},\xi_{\oc,+0} ,\eta_{\oc,+0},\eta_{\oc,-0})$ such that
\begin{align*}
q_0 \in \beta_{SJ}^*\tilde{\beta}_b^*(\mathrm{Gr}(\Clo)),
\quad q_0 \notin \beta_{SJ}^*\tilde{\beta}_b^*(\mathrm{Gr}(\Clt)),
\end{align*}
but $\beta_{SJ}(q_0) = (1,0,y_{\oc,+0},y_{\oc,-0} ,\xi_{\oc,+0},-\xi_{\oc,+0},\eta_{\oc,+0},\eta_{\oc,-0})$ is in both $\partial \tilde{\beta}_b^*(\mathrm{Gr}(\Clo))$ and $\partial \tilde{\beta}_b^*(\mathrm{Gr}(\Clt))$.

Let $S_{\sct}^{0,0}(\R^n)$ be the class of scattering symbols defined in Section~\ref{subsec:sc-bundle-PsiDO} and let $a_\pm \in S_{\sct}^{0,0}(\R^n)$ be symbols supported near $(0,y_{\oc,\pm},\frac{N_{10}}{2},\eta_{\oc,\pm})$ with $\tilde{A}_\pm \in \Psi_{\sct}^{0,0}(\R^n)$ being their quantizations.

Let $U$ be a neighborhood of $\beta_{SJ}(q_0) \in \ococb$ so that the $N_1$-component on $\beta_{SJ}^*\tilde{\beta}_b^*(\mathrm{Gr}(\Clt) \cap U)$ is not in $[N_{10}-\delta_0,N_{10}+\delta_0]$ for some $\delta_0>0$.
This is possible because as pointed out after \eqref{eq:N1-smooth-expression}, $N_1$ is still smooth when viewed as a function on $\tilde{\beta}_b^*(\mathrm{Gr}(\Clt))$.
We may choose $a_\pm \geq 0$ to have small support so that: 
\begin{enumerate}
    \item The support of $a_\pm$ is contained in the projection of $U$ to the left and right factors and further projected to $(x_{\oc,\pm},y_{\oc,\pm},\eta_{\oc,\pm})$-components, and $\xi_{\oc,+}+\xi_{\oc,-}$ is contained in $[N_{10}-\delta_0,N_{10}+\delta_0]$.   \label{condition:a-pm-1}
    \item   In addition, we may assume
   \begin{equation} \label{eq:a-pm-1}
       a_\pm(x_{\oc,\pm 0},y_{\oc,\pm 0},\frac{N_{10}}{2},\eta_{\oc,\pm 0}) = 1.
   \end{equation}
   \label{condition:a-pm-2}
\end{enumerate}

As pointed out in \eqref{eq:sc-phase}, the phase function of $\tilde{A}_\pm$ parametrizes ${}^{\sct}N^*\Delta_b$, the conormal bundle of the diagonal in $\beta_b^*(\scphase \times \scphase)$.
So if we consider
\begin{equation}
    A_\pm = e^{-\frac{\xi_{\oc,+0}}{2x_{\oc,+}^2}} \tilde{A}_\pm e^{\frac{\xi_{\oc,+0}}{2x_{\oc,+}^2}},
\end{equation}
then $A_\pm \in I_{\oc,\sct-\oc,\sct}^{0}(X_b^2,T_{\xi_{\oc,0}}({}^{\sct}N^*\Delta_b))$, where 
$T_{\xi_{\oc,0}}({}^{\sct}N^*\Delta_b)$ is the $\oc,\sct-\oc-\sct$ Lagrangian submanifold second-microlocalized at $\xi_{\oc,0+}$ in ${}^{\oc,\sct,\xi_{\oc,0+}}T^*X_b^2$ obtained by translating $\xi_{\oc,+}$ by $\xi_{\oc,+0}$ and $\xi_{\oc,-}$ by $-\xi_{\oc,+0}$, where we used the identification in Proposition~\ref{prop:1c-sc-phase-space-relation}.
Now let $C_+',C_-'$ be the part of $T_{\xi_{\oc,+0}}({}^{\sct}N^*\Delta_b))$ such that $T_{\xi_{\oc,+0}}^{-1}(C_\pm')$ projects (by the left projection) to a small neighborhood of $\WF'_{\sct}(\tilde{A}_\pm)$, then we know
\begin{equation}
    A_\pm \in I_{\oc,\sct-\oc,\sct}^{0}(X_b^2,C_\pm').
\end{equation}

On the other hand, applying Lemma~\ref{lemma:N1-equals-xisc1-plus-xisc2}, we know condition \eqref{condition:a-pm-1} above ensures that
\begin{align*}
C_+ \circ \tilde{\beta}_b^*(\mathrm{Gr}(\Clt)) \circ C_- = \emptyset.
\end{align*}
So we apply Proposition~\ref{prop:composition-ALA-AAR-FIOsec} first to see that $A_+S_{g_2}Q_{\oc} \in I^0_{\oc,\sct-\oc,\sct}(X_b^2,C_+ \circ \tilde{\beta}_b^*(\mathrm{Gr}(\Clt)))$ and then by Corollary~\ref{coro:ALAR-trivial} we know
\begin{equation}
    A_+S_{g_2}Q_{\oc}A_- \in \mathcal{S}(\R^n \times \R^n),
\end{equation}
hence acts compactly on $L^2(\R^n)$.

So it remains to show that $A_+S_{g_1}Q_{\oc}A_-$ does not act compactly on $L^2(\R^n)$.
We still apply Proposition~\ref{prop:composition-ALA-AAR-FIOsec} and Proposition~\ref{prop:composition-AL-AR-FIOsec} to see that 
\begin{equation}
    A_+S_{g_1}Q_{\oc}A_- \in I^0_{\oc,\sct-\oc,\sct}(X_b^2,C_+ \circ \tilde{\beta}_b^*(\mathrm{Gr}(\Clt)) \circ C_-).
\end{equation}
In addition, using Lemma~\ref{lemma:N1-equals-xisc1-plus-xisc2}, we know that in terms of coordinates in \eqref{eq:coordinates-near-ffsec},
$\partial(C_+ \circ \tilde{\beta}_b^*(\mathrm{Gr}(\Clt)) \circ C_-)$ are those points 
\begin{equation}
(0,1,y_{\oc,+},y_{\oc,-},\xi_{\sct,+},\xi_{\sct,-},\eta_{\oc,+},\eta_{\oc,-})
\end{equation}
such that 
\begin{equation}
    (0,1,y_{\oc,\pm},y_{\oc,\pm},\xi_{\sct,\pm},-\xi_{\sct,\pm},\eta_{\oc,\pm},-\eta_{\oc,\pm}) \in \partial C_+', 
\end{equation}
and in terms of coordinates in \eqref{eq:coordinate-near-ffSJ}
\begin{equation}
(0,1,y_{\oc,+},y_{\oc,-},\xi_{\sct,+}+\xi_{\sct,-},\xi_{\oc,0},\eta_{\oc,+},\eta_{\oc,-}) \in 
\partial \beta_{SJ}^*\tilde{\beta}_b^*(\mathrm{Gr}(\Clt)).
\end{equation}
So by condition \eqref{eq:a-pm-1} above, it is elliptic at $(0,1,y_{\oc,+0},y_{\oc,-0},\frac{N_{10}}{2},\frac{N_{10}}{2},\etaoco,\etaoct)$.
Then by Corollary~\ref{coro:sec-micro-noncompact}, it is not compact from $L^2(\R^n)$ to itself, which completes the proof.

\end{proof}

\subsection{The proof of the main theorem}
We prove the main result, i.e. the `only if' part of Theorem~\ref{thm:main-intro} now.
Suppose $S_{g_1}-S_{g_2}$ is a compact operator from $L^2(\R^n)$ to itself, then by Proposition~\ref{prop:Cl1=Cl2}, $\Clo$ is microlocally lens equivalent to $\Clt$. 
By Proposition~\ref{prop:determine-truncated-SC-map}, we know that their scattering relation truncated to a ball are the same $\Clo^R = \Clt^R$.
Then by Theorem~\ref{thm:Cl-determines-length}, we know $\ell_{g_1(t)} = \ell_{g_2(t)}$ for all $t$. 
Finally, applying Theorem~\ref{thm:boundary-rigidity} completes the proof.

\appendix

\section{The `forward problem': the diffeomorphism invariance of the scattering map on the leading order}
\label{sec:appendix-forward-invariance}

In this appendix, we prove that if two metrics are related by a pull-back, then the scattering map associated to the corresponding Schr\"odinger equations coincide on the leading order. We give the precise formulation first.

\begin{thm} \label{thm:forward-coincide-leading-order}
Let $\psi: \R^{n+1}_{t,z} \to \R^{n+1}_{t,z}$ be a diffeomorphism in form of \eqref{eq:psi-concrete} (i.e. preserving $t$ and only acts non-trivially on a compact region), and 
\begin{equation} \label{eq:metric-pullback}
    g_1 = \psi^*g_2
\end{equation}
satisfying conditions in Section~\ref{subsec:main-results}.
Then they have the same classical scattering map (defined in \eqref{eq:classical-sc-map-def}) $\Clo=\Clt$ and
\begin{equation} \label{eq:Sg1-Sg2-leading-same}
    S_{g_1} - S_{g_2} \in  I^{-1}_{\oc-\oc}(X_b^2,\beta_b^*(\mathrm{Gr}(\Clo)')), \quad X = \overline{\R^n}. 
\end{equation}
In particular, in combination with the first part of Lemma~\ref{lemma:1c1c-FIO-compact-criterion}, this proves the "if" direction in Theorem~\ref{thm:main-intro}.
\end{thm}
The fact $\Clo=\Clt$ follows from the property of $\psi$ that it acts by identity outside a compact region. Let $\gamma(s)$ be a geodesic with respect to $g_1$, then $\psi(\gamma(s))$ is a geodesic with respect to $g_2$ which coincides with $\gamma(s)$ outside a compact region and the conclusion follows.

Now we turn to \eqref{eq:Sg1-Sg2-leading-same}. 
By the short exact sequence in \eqref{eq:1c-1c-FIO-short-exact-sequence}, proving \eqref{eq:Sg1-Sg2-leading-same} is reduced to prove that $S_{g_1}$ and $S_{g_2}$ have the same principal symbol.
Let $\Poi_{\pm,i}$, $i = 1,2$ be the forward and backward Poisson operators associated to $g_1,g_2$, then $S_{g_i}$ are related to them via the formula in \cite[Equation~(7.33)]{gell2022propagation}:
\begin{equation}
 S_{g_i} = i(2\pi)^n \mathcal{P}_{+,i}^* [P_i,Q_+] \mathcal{P}_{-,i}.
\label{eq:S formula 2} 
\end{equation}

Let $\psi_{\ps}$ be the symplectic lift of $\psi$ induced on $T^*\R^{n+1}$ first and then extended to $\overline{{}^{\ps}T^*\R^{n+1}}$. 
Then this further induces a symplectomorphism, which we denote by $\psi_{\oc-\ps}$, on the part of $\oc-\ps$ phase space $\mathcal{M}$ with $\sigma = x_{\oc}/\rho_{\ps} \lesssim 1$ via acting on components lifted from $\overline{{}^{\ps}T^*\R^{n+1}}$ by $\psi_{\ps}$ and act as identity on $\sigma$ and other components lifted from $\overline{{}^\oc{T^*\RR^n}}$.
In particular, this region includes our admissible fibered-Legendre submanifolds and their Lagrangian extensions.
Let $\Lambda_{\pm}^i$, $i=1,2$ be the microlocalized Lagrangians defined in \eqref{eq:microlocalized sojourn relns} with the metric being $g_i$, then by definition and relations between geodesics of $g_1 = \psi^*g_2$ and $g_2$, we know that potentially after shrinking one of $\Lambda_\pm^i$
\begin{equation} \label{eq:Lambda-12-relation}
    \psi_{\oc-\ps}|_{\Lambda_{\pm}^1}: \; \Lambda_{\pm}^1 \to \Lambda_{\pm}^2
\end{equation}
is a diffeomorphism.
If we can prove \footnote{Here we ignore certain microlocalizers needed to put the Poisson operator into our class of Fourier integral operators. See Proposition~\ref{prop:Poisson-leading-invariance} below for the precise version.}
\begin{equation} \label{eq:Poisson-principal-pullback}
    \psi_{\oc-\ps}^*\sigma_{\oc-\ps}^{-3/4}(\Poi_{\pm,2}) = \sigma_{\oc-\ps}^{-3/4}(\Poi_{\pm,1}),
\end{equation}
then since $[P_i,Q_+]$ have the same principal symbol because $P_i$ have the same principal symbol, 
by Proposition~\ref{prop: PsiDO- 1c-ps FIO composition}, we know
\begin{equation}
    \psi_{\oc-\ps}^*\sigma_{\oc-\ps}^{1/4}([P_2,Q_+]\Poi_{-,2}) = \sigma_{\oc-\ps}^{1/4}([P_1,Q_+]\Poi_{-,1}).
\end{equation}
Combining this with \eqref{eq:Poisson-principal-pullback} with $+$ sign, and notice that $\gamma(s)$ is a geodesic of $g_1$ if and only if $\psi(\gamma(s))$ is a geodesic of $g_2$, then by Proposition~\ref{prop: 1c-ps composition to 1c-1c}, we know $S_{g_i}$ expressed as the right hand side of \eqref{eq:S formula 2} have the same principal symbol. 



By discussion above, the following proposition completes the proof of Theorem~\ref{thm:forward-coincide-leading-order}.

\begin{prop} \label{prop:Poisson-leading-invariance}
Let $g_1,g_2$ be as in Theorem~\ref{thm:forward-coincide-leading-order} and $\psi_{\oc-\ps}$ be as above,
then $\Poi_{\pm,\ell}$ will coincide on the leading order in the sense that:
\begin{equation} \label{eq:principal-1}
    \psi_{\oc-\ps}^*\sigma^{-3/4}_{\oc-\ps}(Q_{\ps}\Poi_{-,2}Q_{\oc}) =  \sigma^{-3/4}_{\oc-\ps}(Q_{\ps} \Poi_{-,1}Q_{\oc}).
\end{equation}
where $Q_{\ps},Q_{\oc}$ are microlocalizers as in Section~\ref{subsec:1c-ps-calculus}.
\end{prop}

\begin{proof}
For definiteness we prove the case with $-$ sign and the other case can be proved in the same way.


From the proof of \cite[Proposition~5.16]{HJ2026-scattering-map}, we know that the non-trivial part of $Q_{\ps} \Poi_{-,\ell}Q_{\oc}$ completely comes from the contribution of the parametrix constructed in  \cite[Proposition~5.15]{HJ2026-scattering-map}, which we denote by $K_{-,\ell}$, $\ell=1,2$.
We recall their basic properties here.
First take $Q_{\oc} \in \Psi^{0,0}(\R^n)$ with $\WF'_{\oc}(Q_{\oc})$ away from fiber infinity of $\overline{{}^{\oc}T^*\R^n}$.
Then take $\tilde{Q}_{\ps,\ell} \in \Psi_{\ps}^{0,0}(\R^{n+1})$ satisfying:
\begin{multline}\label{eq:tilde Qps cond}
\text{no point of $\WF_{\ps}'(\tilde Q_{\ps,\ell})$ is forward, with respect to the Hamilton flow of $p_\ell$,}
\\ \text{of any point of $\WF_{\ps}'(P_\ell-P_0)$.} 
\end{multline}
We decompose $[P_\ell,\tilde{Q}_{\ps,\ell}]$ as
\begin{equation}
    [P_\ell,\tilde{Q}_{\ps,\ell}] = Q_{1,\ell}+Q_{2,\ell}
\end{equation}    
such that 
\begin{equation}
    \pi_{\ps}^{-1}(\WF'_{\ps}(Q_{2,\ell})) \cap \big( \Lambda_0' \circ \WF'_{\oc}(Q_{\oc}) \big) = \emptyset,
\end{equation}
and both of $Q_{1,\ell},Q_{2,\ell}$ still satisfy the wavefront bound in \eqref{eq:tilde Qps cond}.

In particular, we can choose $\tilde{Q}_{\ps,\ell}$ so that their principal symbols $\tilde{q}_{\ps,\ell}$ are related by
\begin{equation}
    \tilde{q}_{\ps,1} = \psi_{\ps}^*\tilde{q}_{\ps,2}.
\end{equation}
Then this allows us to choose $Q_{1,\ell}$ so that their principal symbols $q_{1,\ell},q_{2,\ell}$ are related by
\begin{equation} \label{eq:qi-principal-relation}
    q_{1,1} = \psi_{\ps}^*q_{1,2},\; q_{2,1} = \psi_{\ps}^* q_{2,2}.
\end{equation}
Then with $Q_{\oc}, Q_{1,\ell}$ as above, the parametrices constructed in \cite[Proposition~5.15]{HJ2026-scattering-map} are characterized by:
\begin{align}
K_{-,\ell} = \sum_{j=0}^\infty K_{-,\ell,j}, 
\quad K_{-,\ell,j}  \in I^{-j -\frac{3}{4} }_{\oc-\ps}(\R^{n+1} \times \R^n, \Lambda_{-,\ell}),
\end{align}
such that
\begin{align} \label{eq: parametrix property, up to N}
\Big( P_\ell \sum_{j=0}^N K_{-,\ell,j} - Q_{1,\ell}\mathcal{P}_0 Q_{\oc} \Big) \in I^{-\frac{3}{4}-N}_{\oc-\ps}(\R^{n+1} \times \R^n, \Lambda_{-,\ell}).
\end{align}
The desired conclusion \eqref{eq:principal-1} follows if we can prove
\begin{equation}
    \psi_{\oc-\ps}^*\sigma_{\oc-\ps}^{-3/4}(K_{-,1,0})
    = \sigma_{\oc-\ps}^{-3/4}(K_{-,2,0}).
\end{equation}  
By Proposition~\ref{prop: vanishing principal symbol product},
the amplitudes of $K_{-,\ell,0}$ are obtained from solving transport equations
\begin{align} \label{eq: transport ODE, a-0}
(-i\mathscr{L}_{ H_{p_\ell}^{2,0} } 
+i (\frac{n}{2}+1)(x_{\oc}^{-1} H_{p_\ell}^{2,0} x_{\oc}) 
+ p_{\sub,\ell})(a_{-,k,0}) = \sigma_{\oc-\ps}^{1/4}(Q_{1,\ell}\mathcal{P}_0 Q_{\oc}),
\end{align}
with the initial condition $0$ for sufficiently negative time. Since $p_1 = \psi_{\oc-\ps}^*p_2$ by \eqref{eq:metric-pullback}, we know
\begin{equation}
    (\psi_{\oc-\ps}^{-1})_*H_{p_2}^{2,0}
    = H_{p_2}^{1,0}.
\end{equation}
Also, by the discussion before \cite[Theorem~5.8]{HJ2026-scattering-map} and the reference therein, we know
that $p_{\sub,1}$ coincide with $\psi_{\oc-\ps}^*p_{\sub,2}$ on the leading order.
For the forcing on the right hand side, by the condition on the wavefront set of $\tilde{Q}_{\ps,\ell}$ in \eqref{eq:tilde Qps cond}, we know that the wavefront set of $Q_{1,\ell}\Poiz Q_{\oc}$ is contained in the part of $\Lambda_0$ coincide with $\Lambda_-$: before meeting the metric perturbation. So in particular $Q_{1,\ell}\mathcal{P}_0 Q_{\oc} \in I^{1/4}_{\oc-\ps}(\R^{n+1} \times \R^n,\Lambda_{-,\ell})$ and the principal symbol is only non-vanishing on a region on which $\psi_{\oc-\ps}$ is the identity, or equivalently where $\Lambda_{-,1}$ locally is the same $\Lambda_{-,2}$.

In summary, the transport equation for $a_{-,1,0}$ coincide the one obtained by pull-back the transport equation for $a_{0,1,0}$ on the leading order (in terms of either $\rho_{\ps}$ or $\xoc$) and \eqref{eq:principal-1} follows. 
\end{proof}

\section{The global phase space for 1c,sc analysis}
\label{appendix:global-1csc-phase-space}

In this appendix, we briefly sketch how to implement the second-microlocalization in Section~\ref{sec:second-micro} if one wish to handle large $\etaoc$ via constructing a global phase joint phase space encoding the 1-cusp and sc-level frequencies simultaneously.
This will be needed if one want to handle metric perturbations that are not compactly supported in space.
The key step is to extend the identification in Proposition~\ref{prop:1c-sc-phase-space-relation} up to the boundary of the front face, which corresponds to fiber infinity of the scattering cotangent bundle.
But in other regions, the characterization of this identification is slightly different.
In this case, we should start with the compactified 1-cusp phase space $\ocphase$ and the first resolved submanifold is  $\{ \xioc/\la \etaoc \ra = 0, \, \xoc = 0\}$.
So the phase space we will use is
\begin{equation} \label{eq:oc-sc-bundle-0}
    [\ocphase;\, \{ \xioc/\la \etaoc \ra = 0, \, \xoc = 0\}].
\end{equation}
\begin{figure}
\begin{subfigure}[t]{.45\textwidth}
    \centering
    \includegraphics[width=0.8\linewidth]{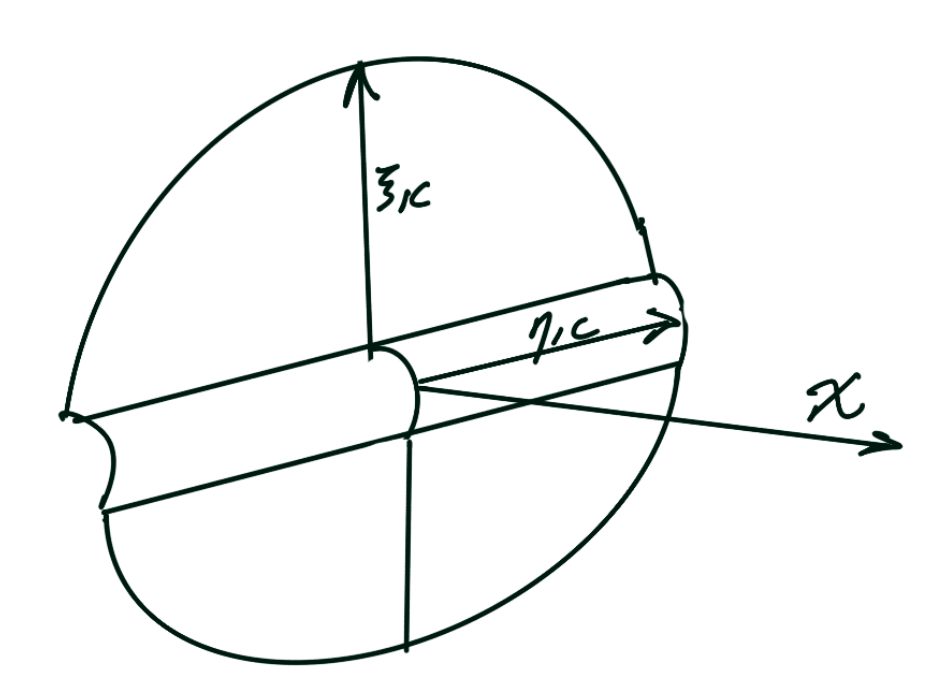}
    \caption{A graphic illustration of the phase space in \eqref{eq:oc-sc-bundle-0}. The disk represents $\overline{\R}^n$, which is the compactified fiber in $\ocphase$, with $\etaoc$ compressed into one dimension. The `bump' in the middle of the disk is the front face, which will be identified as a scattering cotangent bundle. }
    \label{subfig:1c-1}    
\end{subfigure}
\begin{subfigure}[t]{.45\textwidth}
    \centering
    \includegraphics[width=0.9\linewidth]{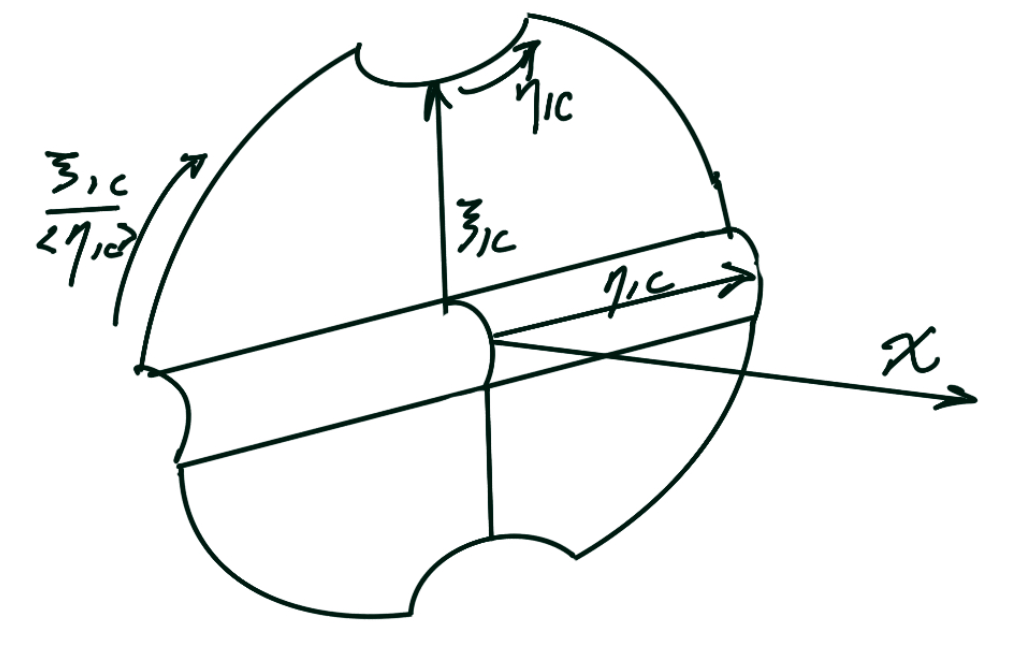}
    \caption{A graphic illustration of the phase space in \eqref{eq:oc-sc-bundle-1}. The top and bottom face created by the further blow-up is parametrized by $\etaoc \in \overline{\R^{n-1}}$. }
    \label{subfig:1c-2}    
\end{subfigure}
\caption{Resolved $\oc$-phase spaces.}
\label{fig:1c-phase-1-2}
\end{figure}

In order to facilitate the identification between objects constructed out of $\overline{{}^{\sct}T^*\R^n}$, we further blow up it at the north and south pole $\{ \etaoc/|\xioc| = 0, \, 1/\xioc =0  \}$:
\begin{equation} \label{eq:oc-sc-bundle-1}
   \Big[ [\ocphase;\, \{ \xioc/\la \etaoc \ra = 0, \, \xoc = 0\}]; \{ \la \etaoc \ra/\xioc = 0, \, 1/\xioc =0  \} \Big].
\end{equation}

This can be identified with the $\oc,\sct$-cotangent bundle constructed out of the scattering cotangent bundle as follows.
Starting from $\overline{{}^{\sct}T^*\R^n}$, we first blow up the north and south poles in the compactified fiber defined by $\{\eta_{\sct}/\xi_{\sct} = 0, 1/\xi_{\sct}=0\}$ to produce
\begin{equation} \label{eq:sc-resolved-1}
    [\overline{{}^{\sct}T^*\R^n}; \{\eta_{\sct}/\xi_{\sct} = 0, 1/\xi_{\sct}=0\}].
\end{equation}

\begin{figure}
\begin{subfigure}[t]{.45\textwidth}
    \centering
    \includegraphics[width=0.8\linewidth]{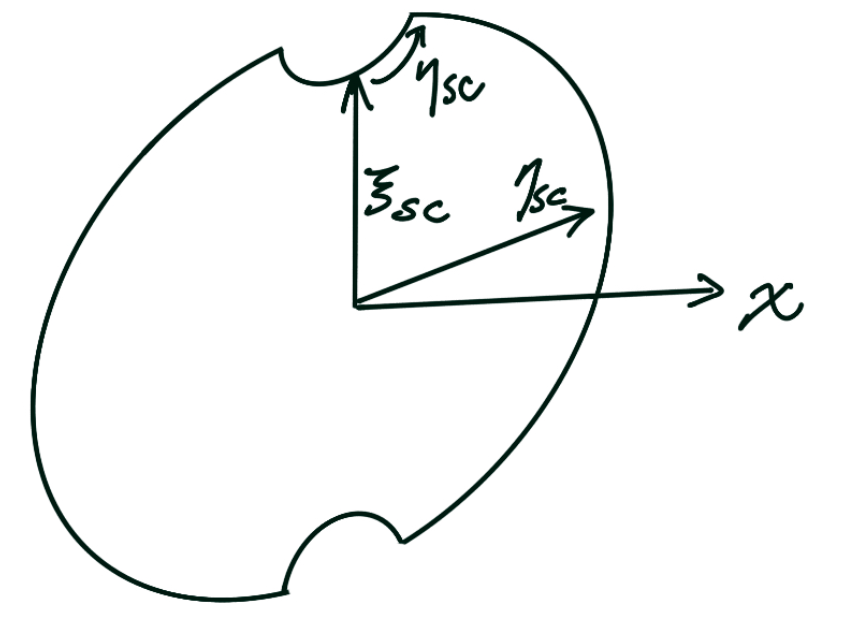}
    \caption{A graphic illustration of the phase space in \eqref{eq:sc-resolved-1}. The top and bottom face created by the further blow-up is parametrized by $\eta_{\sct} \in \overline{\R^{n-1}}$.}
    \label{subfig:sc-1}    
\end{subfigure}
\begin{subfigure}[t]{.45\textwidth}
    \centering
    \includegraphics[width=0.8\linewidth]{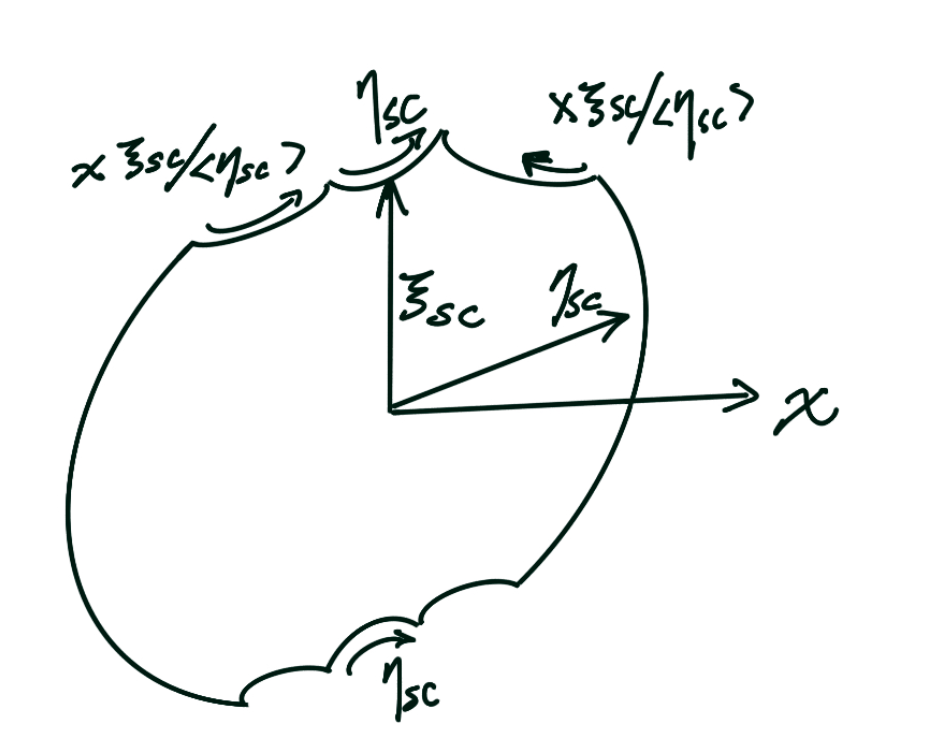}
    \caption{ The phase space in \eqref{eq:oc-sc-bundle-2}. The new front face is parametrized by $x\xi_{\sct}/\la \eta_{\sct}\ra$ (and its reciprocal), and is identified with arcs in
    \eqref{subfig:1c-2} parametrized by $\xioc/\la \etaoc \ra$ (and its reciprocal). }
    \label{subfig:sc-2}    
\end{subfigure}
\caption{Resolved $\sct$-phase spaces.}
\end{figure}
This blow-up `decouples' $\xi_{\sct}$ and $\eta_{\sct}$ and results in a bundle with fiber being a box $\overline{\R} \times \overline{\R^{n-1}}$, \footnote{Of course, one can start with compactifying ${}^{\sct}T^*\R^n$ to be so directly, here we are just building things based on more standard objects.}
where the $\overline{\R}$ is parametrized by $\xi_{\sct}/|\eta_{\sct}|$ or $|\eta_{\sct}|/\xi_{\sct}$, and $\overline{\R^{n-1}}$ is parametrized by $\eta_{\sct}$ in the interior and $(|\eta_{\sct}|^{-1},\hat{\eta}_{\sct}) \in [0,\epsilon) \times \mathbb{S}^{n-2}$ near the boundary.
Let $x$ be the boundary defining function of the base $\overline{\R^n}$, then we further blow up the cap and the bottom of the fiber at the boundary, which is $\{ 1/\xi_{\sct} = 0, \, x = 0\}$  to obtain
\begin{equation} \label{eq:oc-sc-bundle-2}
\Big[ [\overline{{}^{\sct}T^*\R^n}; \{\eta_{\sct}/\xi_{\sct} = 0, 1/\xi_{\sct}=0\}] ;\{ \la \eta_{\sct} \ra/\xi_{\sct} = 0, \, x = 0\} \Big].    
\end{equation}
Now we show that manifolds (with corners) in \eqref{eq:oc-sc-bundle-1} and \eqref{eq:oc-sc-bundle-2} are canonically diffeomorphic to each other via extending the identity map in the interior.
In terms coordinates in \eqref{eq: 1c- canonical form} and \eqref{eq:sc-1-form}, this is 
\begin{equation} \label{eq:1c-to-sc-map}
    (\xoc,\yoc,\xioc,\etaoc) \to (x=\xoc,y=\yoc,\xi_{\sct}=\frac{\xioc}{x},\eta_{\sct}=\etaoc).
\end{equation}

Consider the case with $\etaoc$ in a bounded region first. The interior of the front face created by the first blow-up, which we denote by $\ff_{\oc}$, is identified with the part of $\overline{{}^{\sct}T^*\R^n}$ with finite $\sct$-frequency as before by extending \eqref{eq:1c-to-sc-map}.
Near the corner $\{\frac{\xoc}{\xioc} = 0, \xioc = 0 \}$, which is the part with high-frequency on the sc-level but near $\xioc = 0$, a valid coordinate system is
\begin{equation} \label{eq:coordinate-1csc-corner-low}
    \frac{\xoc}{\xioc}, \, \xioc, \, \etaoc.
\end{equation}
On the other hand, for \eqref{eq:oc-sc-bundle-2}, near the corner formed by the intersection of the lift of $\{\frac{\la \eta_{\sct} \ra}{|\xi_{\sct}|}>0,x=0\}$ and the front face, a valid coordinate system is
\begin{equation} \label{eq:coordinate-1csc-corner-low-2}
  \frac{1}{\xi_{\sct}},\,   \frac{x}{1/\xi_{\sct}}, \,   \eta_{\sct}, 
\end{equation}
which can be identified with \eqref{eq:coordinate-1csc-corner-low}.
Using similar argument, the region near the interior of two front faces at the top and the bottom in \eqref{eq:oc-sc-bundle-2} will be identified with the part of \eqref{eq:oc-sc-bundle-1} with finite frequencies and $\xioc \neq 0$,
and the lift of $\{x>0,\frac{\la \eta_{\sct}\ra}{\xi_{\sct}} = 0\}$ will be identified with front face in \eqref{eq:oc-sc-bundle-1} created by the last blow-up and defined by $\la \etaoc \ra/\xioc = 0$, both of which are parametrized by $\eta_{\oc}=\eta_{\sct}$. 

For the part with large $|\eta_{\oc}|$ or $|\eta_{\sct}|$, the proof is almost the same except for with $\xioc$, $\xi_{\sct}$ and their reciprocals are replaced by $\xioc/\la \eta_{\oc} \ra$ and $\xi_{\sct} / \la \eta_{\sct} \ra$ and their reciprocals.
For example, coordinates in \eqref{eq:coordinate-1csc-corner-low} becomes
\begin{equation} \label{eq:coordinate-1csc-corner-low'}
    \frac{\xoc}{\xioc/\la \etaoc \ra}, \, \xioc/\la \etaoc \ra, \, |\etaoc|^{-1}, \, \hat{\eta}_{\oc} \in \mathbb{S}^{n-2}.
\end{equation}
Coordinate in \eqref{eq:coordinate-1csc-corner-low-2} becomes 
\begin{equation} \label{eq:coordinate-1csc-corner-low-2'}
  \frac{\la \eta_{\sct} \ra}{\xi_{\sct}},\,   \frac{x}{\la \eta_{\sct} \ra/\xi_{\sct}}, \,   |\eta_{\sct}|^{-1}, \hat{\eta}_{\sct} \in \mathbb{S}^{n-2}, 
\end{equation}
and two parts can be identified as before.

The symbol class that is natural to use should be (polynomially weighted) conormal functions on the space in  \eqref{eq:oc-sc-bundle-0}. Equivalently, one can blow-down the face $\{ \frac{\la \eta_{\sct} \ra/\xi_{\sct}}{x}  = 0 \}$ 
in \eqref{eq:oc-sc-bundle-2} and consider conormal functions on that.
This face blows down to the north and south pole of $\overline{\R^n}$ parametrized by 1-cusp frequencies.
This redundant face was inherited from the first blow-up, which was to recover $\eta_{\sct} = \eta_{\oc}$ up to $\xi_{\sct} = \infty$, but this is not needed for the part with large 1-cusp frequencies.
The space in \eqref{eq:oc-sc-bundle-0} have three types of boundary faces:
\begin{enumerate}
    \item The lift of $\{\xioc/\la \etaoc \ra \neq 0, \xoc =0\}$, which we denote by $\mathrm{1cf}$.
    \item The lift of the original fiber infinity in $\ocphase$, which we denote by $\mathrm{df}$ (differential face).
    \item The front face created by blow-up, which we denote by $\ff_{\sct}$.
\end{enumerate}
Let $\rho_{\bullet}$ be the boundary defining function of the face $\bullet$, then we denote conormal functions on the space in \eqref{eq:oc-sc-bundle-0} by $S_{\oc,\sct}^{0,0,0}$ and define
\begin{equation}
    S_{\oc,\sct}^{m,r_1,r_2} = \rho_{\mathrm{df}}^{-m}\rho_{\mathrm{1cf}}^{-r_1}\rho_{\ff_{\sct}}^{-r_2}S_{\oc,\sct}^{0,0,0}.
\end{equation}
Then we use the usual (left) quantization formula 
\begin{equation}
    \Op(a) = (2\pi)^{-n} \int e^{i(Z-Z') \cdot \xi} a(Z,\mk{Z}) \, d\mk{Z}, \quad Z, \mk{Z} \in \R^n,
\end{equation}
when written as Euclidean coordinates, to quantize this symbol class to obtain the $\oc,\sct$ pseudodifferential class $\Psi_{\oc,\sct}^{m,r_1,r_2}$.

As in Section~\ref{sec:second-micro}, if the sojourn time of a geodesic corresponds to the point in the lifted $\Cli$ with $\xioco = \xi_{\oc,0}$ are different,
then we can use 
\begin{equation}
    e^{-i\frac{\xi_{\oc,0}}{2\xoc^2}} A e^{i\frac{2\xi_{\oc,0}}{\xoc^2}}, \; A \in \Psi_{\oc,\sct}^{m,r_1,r_2}
\end{equation}
to microlocalize as in Section~\ref{subsec:determine-1-jet} to produce a contradiction, but now we don't require $A$ to have a symbol compactly supported in $\etaoc$, which corresponds to requiring the projection of points on the geodesic to the plane orthogonal to its asymptotic direction to be uniformly bounded.

\bibliographystyle{plain}
\bibliography{bib_sc_map_inverse}

\end{document}